\documentclass{amsart}
\usepackage{verbatim}
\usepackage[textsize=scriptsize]{todonotes}
\usepackage{etoolbox}
\usepackage{stackrel}
\usepackage{longtable}
\usepackage{xcolor}
\usepackage{mathtools}

\newtoggle{final}

\usepackage{tikz}
\usetikzlibrary{matrix,arrows}
\usepackage{tikz-cd}

\toggletrue{final}

\newcommand{\Sing}{\mathrm{Sing}}
\newcommand{\Grp}{\mathrm{Grp}}
\newcommand{\Ab}{\mathrm{Ab}}

\def\cpartial{\delta}
\newcommand{\KMWn}{\ul{K}'^{MW}}

\def\Pro{\mathrm{Pro}}
\def\aff{\mathrm{aff}}

\usepackage{etoolbox}
\usepackage[margin=1in]{geometry}
\usepackage{amsmath}
\usepackage{amssymb}
\usepackage{amsthm}
\usepackage{amscd}
\usepackage{enumerate}
\usepackage[pdfusetitle,unicode,hidelinks]{hyperref}
\usepackage{bbm}
\usepackage{etoolbox}
\usepackage{stmaryrd}

\usepackage[utf8]{inputenc}
\usepackage[T1]{fontenc}

\usepackage[textsize=scriptsize]{todonotes}
\newcommand{\NB}[1]{\todo[color=gray!40]{#1}}

\newcommand{\tomemail}{\href{mailto:tom.bachmann@zoho.com}{tom.bachmann@zoho.com}}

\newtheorem{proposition}{Proposition}
\newtheorem{corollary}[proposition]{Corollary}
\newtheorem{lemma}[proposition]{Lemma}
\newtheorem{theorem}[proposition]{Theorem}

\newtheorem*{conjecture*}{Conjecture}
\newtheorem*{theorem*}{Theorem}
\newtheorem*{corollary*}{Corollary}
\newtheorem*{proposition*}{Proposition}
\newtheorem*{lemma*}{Lemma}
\theoremstyle{definition}
\newtheorem{definition}[proposition]{Definition}
\newtheorem{construction}[proposition]{Construction}

\newtheorem*{definition*}{Definition}
\newtheorem*{construction*}{Construction}
\theoremstyle{remark}
\newtheorem{remark}[proposition]{Remark}
\newtheorem*{remark*}{Remark}

\newtheorem{example}[proposition]{Example}
\newtheorem*{example*}{Example}
\newtheorem{discussion}[proposition]{Discussion}

\newcommand{\id}{\operatorname{id}}
\newcommand{\Z}{\mathbb{Z}}

\newcommand{\N}{\mathbb{N}}

\let\scr=\mathcal
\let\bb=\mathbb
\newcommand{\Gm}{{\mathbb{G}_m}}
\newcommand{\Gmp}[1]{{\mathbb{G}_m^{\wedge #1}}}
\def\A{\bb A}
\def\P{\bb P}

\newcommand{\SH}{\mathcal{SH}}

\DeclareMathOperator*{\colim}{colim}

\let\lim=\relax
\DeclareMathOperator*{\lim}{lim}
\def\Map{\mathrm{Map}}
\def\iMap{\ul{\mathrm{Map}}}

\def\PSh{\mathcal{P}}

\def\Spc{\mathcal{S}\mathrm{pc}{}}

\def\Fun{\mathrm{Fun}}

\newcommand{\Spec}{\mathrm{Spec}}

\newcommand{\wequi}{\simeq}

\DeclareRobustCommand{\ul}{\underline}

\newcommand{\tr}{\mathrm{tr}}

\newcommand{\iHom}{\ul{\operatorname{Hom}}}
\def\op{\mathrm{op}}

\let\cat=\mathrm
\def\Sm{{\cat{S}\mathrm{m}}}
\def\Ft{{\cat{F}\mathrm{t}}}
\def\Sch{\cat{S}\mathrm{ch}{}}

\def\Nis{\mathrm{Nis}}
\def\Zar{\mathrm{Zar}}

\def\mot{\mathrm{mot}}
\newcommand{\et}{{\acute{e}t}}

\newcommand{\lra}[1]{\langle #1 \rangle}
\def\ph{\mathord-}

\newcommand{\Th}{\mathrm{Th}}

\newcommand{\Shv}{\cat{S}\mathrm{hv}}

\newcommand{\fib}{\mathrm{fib}}

\newcommand{\GL}{\mathrm{GL}}
\newcommand{\SL}{\mathrm{SL}}

\usepackage{xstring}
\usepackage{xr}
\newcommand{\myexternaldocument}[2][] {{
\let\nl\newlabel
\def\nlxx##1##2##3##4##5##6{\nl{##1}{{#1##2}{}{}{}{}}}
\renewcommand\newlabel[2]{\IfBeginWith{##1}{tocindent}{}{\nlxx{##1}##2}}
\externaldocument{#2}
}}

\theoremstyle{definition}

\theoremstyle{theorem}
\newtheorem{proposition+}[proposition]{Proposition*}
\newtheorem{corollary+}[proposition]{Corollary*}
\newtheorem{lemma+}[proposition]{Lemma*}
\newtheorem{theorem+}[proposition]{Theorem*}

\numberwithin{proposition}{section}

\iftoggle{final} {
\renewcommand{\NB}[1]{}
\renewcommand{\todo}[1]{}
}

\title{Strongly $\A^1$-invariant sheaves (after F. Morel)}
\date{\today}

\author{Tom Bachmann}
\address{Institut Mathematik, JGU Mainz, Germany}
\email{\tomemail}

\begin{document}

\maketitle

\begin{abstract}
Strongly (respectively strictly) $\A^1$-invariant sheaves are foundational for motivic homotopy theory over fields.
They are sheaves of (abelian) groups on the Nisnevich site of smooth varieties over a field $k$, with the property that their zeroth and first Nisnevich cohomology sets (respectively all Nisnevich cohomology groups) are invariant under replacing a variety $X$ by $X \times \A^1$.
A celebrated theorem of Fabien Morel \cite[Theorem 5.46]{A1-alg-top} states that if the base field $k$ is perfect, then any strongly $\A^1$-invariant sheaf of abelian groups is automatically strictly $\A^1$-invariant.

The aim of these lecture notes is twofold: (1) provide a complete proof if this result, and (2) outline some of its applications.
\end{abstract}

\tableofcontents

\section{Introduction}
The aim of these notes is to prove some of the foundational results of unstable motivic homotopy theory over perfect fields.
All of our proofs are based on F. Morel's seminal book \cite{A1-alg-top}.

\subsection{Setting the stage}
We fix a field $k$, usually assumed perfect, and write $\Sm_k$ for the category of smooth $k$-varieties.
The category\footnote{We use the language of $\infty$-categories freely throughout this text, as set out in \cite{lurie-htt,lurie-ha}. In particular all notions are homotopy invariant by definition. Thus, e.g., ``colimit'' means ``homotopy colimit'', and so on.} $\Spc(k) \subset \PSh(\Sm_k)$ is obtained by localizing at the \emph{$\A^1$-homotopy equivalences} and the \emph{Nisnevich equivalences}, that is, consists of the subcategory of presheaves (of spaces) which are $\A^1$-invariant and satisfy Nisnevich descent.
The inclusion admits a left adjoint, called \emph{motivic localization} and denoted $L_\mot$.

The inclusion of Nisnevich sheaves into presheaves also has a left adjoint, denoted $L_\Nis$.
This is the sheafification in a Grothendieck topology and so comes with the usual properties, e.g., $L_\Nis$ preserves finite limits.
There are also some properties special to the Nisnevich topology:
\begin{enumerate}
\item Since the Nisnevich topos of finite type $k$-scheme of dimension $d$ has homotopy dimension $\le d$ (see Theorem \ref{thm:coh-dim} below), the topos $\Shv_\Nis(\Sm_k)$ is Postnikov complete \cite[Corollary 7.2.1.12]{lurie-htt}, and so in particular hypercomplete.
\item Since Nisnevich sheaves are characterized by a Brown--Gersten property \cite[Proposition C.5(2)]{hoyois2015quadratic}, the inclusion $\Shv_\Nis(\Sm_k) \hookrightarrow \PSh(\Sm_k)$ preserves filtered colimits (use that filtered colimits of spaces commute with finite limits).
\end{enumerate}

The inclusion of $\A^1$-invariant presheaves into all presheaves has a left adjoint denoted $\Sing$.
It is given by a generalization of the classical singular construction, with affine spaces in place of singular simplices.

We can read off the following formal consequences for the motivic localization functor.
\begin{lemma} \label{lemm:lmot-presn}
\begin{enumerate}
\item The functor $L_\mot$ can be written as \[ \colim \left( L_\Nis \to \Sing L_\Nis \to L_\Nis \Sing L_\Nis \to \cdots \right). \]
\item The functor $L_\mot$ preserves finite products.
\end{enumerate}
\end{lemma}
\begin{proof}
Write $L'$ for the colimiting functor.
It is clear by construction that $\id \to L'$ is a motivic equivalence.
Evaluating $L'$ as a colimit over the cofinal subset of even/odd terms, we see that it can be written as a filtered colimit of either $\A^1$-invariant presheaves or Nisnevich sheaves.
Since both of these categories are closed under filtered colimits, we see that $L'$ is motivic local.
This proves (1).

(2) follows immediately since $L_\Nis$ and $\Sing$ preserve finite products (the latter being a sifted colimit of right adjoint functors), as do filtered colimits.
\end{proof}

\subsection{Homotopy sheaves and connectivity}
We like to think of $\Spc(k)$ as a generalization of spaces suitable for algebraic geometry.
We would thus naturally like to use notions analogous to homotopy groups, connectivity and so on.
This is done by considering the embedding \[ \Spc(k) \hookrightarrow \Shv_\Nis(\Sm_k). \]
The target category is an $\infty$-topos, and so has well-understood homotopy theory \cite[\S6.5]{lurie-htt}.
We transport properties and definitions into $\Spc(k)$ along this embedding.
Thus each motivic space $X \in \Spc(k)$ has a \emph{zeroth homotopy sheaf} $\ul\pi_0(X) \in \Shv_\Nis(\Sm_k)_{\le 0}$.
Given a basepoint $* \in X$ (i.e. morphism $* \to X$), we also obtain \emph{higher homotopy sheaves} $\ul\pi_i(X, *) \in \Shv_\Nis(\Sm_k)_{\le 0}$.
We know that $\ul\pi_1(X,*)$ is a sheaf of groups and $\ul\pi_i(X,*)$ is a sheaf of abelian groups for $i \ge 2$.

However not everything can be transported so straightforwardly.
For example, if $X \in \Shv_\Nis(\Sm_k)$, we can consider the \emph{$n$-truncation} $X_{\le n} \in \Shv_\Nis(\Sm_k)$.
Unfortunately, if $X \in \Spc(k) \subset \Shv_\Nis(\Sm_k)$, then it need not be the case that $X_{\le n} \in \Spc(k)$.
In fact examples of this already arise for $n=0$ \cite{ayoub-counterex}.
A silver lining is that the problems seem to arise \emph{only} for $\ul\pi_0$.
That is, we obtain a very well-behaved theory when considering only \emph{connected} motivic spaces.
We outline this in the next few sections.
Here we contend ourselves with some observations about the permanence of connected motivic spaces.

\begin{lemma} \label{lemm:0-conn}
Let $X \in \Shv_\Nis(\Sm_k)$ be connected.
Then $L_\mot X$ is also connected.
\end{lemma}
\begin{proof}
Since homotopy sheaves commute with filtered colimits, connected objects are stable under filtered colimits.
Using the presentation of $L_\mot$ from Lemma \ref{lemm:lmot-presn}(1), we thus need only prove that $L_\Nis \Sing X$ is connected.
But $\ul\pi_0 L_\Nis \Sing X$ is a quotient of the sheaf $\ul\pi_0 X = *$, whence the result.
\end{proof}

\begin{lemma} \label{lemm:conn-cpt-of-basepoint}
Let $X \in \Spc(k)_* \subset \Shv_\Nis(\Sm_k)_*$.
Write $X_0 \in \Shv_\Nis(\Sm_k)_*$ for the connected component of the base point in $X$.
Then $X_0 \in \Spc(k)_*$.
\end{lemma}
\begin{proof}
The morphism $X_0 \to X$ factors through $L_\mot X_0$, $X$ being motivically local.
Since $L_\mot X_0$ is connected by Lemma \ref{lemm:0-conn}, the map $L_\mot X_0 \to X$ factors through $X_0$.
We have constructed a retraction $X_0 \to L_\mot X_0 \to X_0$, which proves that $X_0$ is motivically local (recall that local objects in any Bousfield localization are stable under retracts, since equivalences are).
\end{proof}

\subsection{Strongly and strictly $\A^1$-invariant sheaves}
\begin{definition}
A sheaf of groups $G$ on $\Sm_k$ is called \emph{strongly $\A^1$-invariant} if the classifying space $B_\Nis G \in \Shv_\Nis(\Sm_k)$ is $\A^1$-invariant (that is, motivically local).
A sheaf of abelian groups $A$ is called \emph{strictly $\A^1$-invariant} if the classifying spaces $K_\Nis(A, n) \in \Shv_\Nis(\Sm_k)$ are $\A^1$-invariant for every $n$.
\end{definition}

\begin{lemma}
\begin{enumerate}
\item $G$ is strongly $\A^1$-invariant if and only if $H^i(X, G) \wequi H^i(X \times \A^1, G)$ for every $X \in \Sm_k$ and $i \le 1$.
\item $A$ is strictly $\A^1$-invariant if and only if $H^i(X, A) \wequi H^i(X \times \A^1, A)$ for every $X \in \Sm_k$ and every $i$.
\end{enumerate}
\end{lemma}
\begin{proof}
(1) Since $\pi_{1-i}((B_\Nis G)(X), *) \wequi H^i(X, G)$, necessity is clear.
For sufficiency, we must show that for $X \in \Sm_k$, the natural map $(B_\Nis G)(X) \to (B_\Nis G)(X \times \A^1)$ is an equivalence.
Both spaces are $1$-truncated, and we have an isomorphism on $\pi_0$ by the expression above.
It remains to prove that for $\xi \in \pi_0 (B_\Nis G)(X)$, the induced map $\pi_1((B_\Nis G)(X), \xi) \to \pi_1((B_\Nis G)(X \times \A^1), \xi)$ is an isomorphism.
Note that the presheaf \[ P: (\Sm_k)_{/X} \to \Grp, Y \mapsto \pi_1((B_\Nis G)(Y), \xi) \] is a sheaf in the Nisnevich topology (it is given by $\Omega_\xi (B_\Nis G)|_{(\Sm_k)_{/X}}$, and sheaves are stable under limits).
Proving it is $\A^1$-invariant (which is what we want) is thus a local problem.
But locally, $\xi = *$ corresponds to the trivial torsor, so we may assume that $\xi=*$.
In this case $P = H^0(\ph, G)$, which is $\A^1$-invariant by assumption.

(2) Necessity is clear as before.
For sufficiency, note that $K_\Nis(A, n)(X) \to K_\Nis(A, n)(X\times \A^1)$ is a morphism of grouplike $H$-spaces (in fact $\scr E_\infty$-monoids), and hence is an equivalence if and only if it induces an equivalence on $\pi_0$ and on $\pi_i$ with the base point $0$, for $i \ge 1$ (i.e. all components of these spaces are canonically equivalent).
As before, the isomorphism on homotopy groups at base point $0$ translates precisely into the isomorphism of cohomology groups that was assumed.
\end{proof}

With this preparation out of the way, we can state the main results of this text.
\begin{theorem}[see Theorem \ref{thm:pi1-strongly-inv-forreal}] \label{thm:pi1-strongly-inv}
Let $k$ be any field.
Let $X \in \Spc(k)_*$.
Then $\ul\pi_1(X)$ is strongly $\A^1$-invariant.
\end{theorem}

\begin{theorem}[see Theorem \ref{thm:strict-A1-invariance}] \label{thm:strongly-strictly-inv}
Let $k$ be a perfect field.
Let $A$ be a sheaf of abelian groups on $\Sm_k$ which is strongly $\A^1$-invariant.
Then $A$ is strictly $\A^1$-invariant.
\end{theorem}

\subsection{Consequences of the theorems}
Many properties of unstable motivic homotopy theory can be deduced from the above theorems.
For a careful exposition, see, e.g., \cite[\S2.2 and \S2.3]{asok2017simplicial}.
Here and in the main text, we will distinguish results relying on Theorems \ref{thm:pi1-strongly-inv} and \ref{thm:strongly-strictly-inv} by an asterisk.

Here are some samples.

\begin{corollary+} \label{cor:pi-i-strongly-inv}
Let $X \in \Spc(k)_*$.
\begin{enumerate}
\item $\ul\pi_i X$ is strictly $\A^1$-invariant for every $i \ge 2$ (strongly invariant for $i=1$).
\item Suppose that $X$ is connected.
  Then for every $i \ge 1$, the sheaf $X_{\le i}$ is $\A^1$-invariant.
\end{enumerate}
\end{corollary+}
\begin{proof}
(1) Each $\ul\pi_i X$ is strongly invariant by Theorem \ref{thm:pi1-strongly-inv} applied to $\Omega^{i-1} X$.
If $i \ge 2$ this sheaf is abelian, and hence strictly $\A^1$-invariant by Theorem \ref{thm:strongly-strictly-inv}.

(2) Inductively apply Lemma \ref{lemm:totalspace-fiber-A1-inv} below to the fiber sequences $K(\ul\pi_{i+1} X, i+1) \to X_{\le i+1} \to X_{\le i}$.
\end{proof}

\begin{lemma} \label{lemm:totalspace-fiber-A1-inv}
Let $F \to E \to B$ be a fiber sequence in $\Shv_\Nis(\Sm_k)_*$ with $B$ connected and $\A^1$-invariant.
Then $F$ is $\A^1$-invariant if and only if $E$ is $\A^1$-invariant.
\end{lemma}
\begin{proof}
Write $\Omega_{\A^1_+}$ for the functor $\Omega_{\A^1_+}(P)(X) = P(X \times \A^1)$.
We have a morphism of fiber sequences
\begin{equation*}
\begin{CD}
F @>>> E @>>> B \\
@VaVV  @VbVV  @VcVV \\
\Omega_{\A^1_+} F @>>> \Omega_{\A^1_+} E @>>> \Omega_{\A^1_+} B.
\end{CD}
\end{equation*}
Since by assumption $c$ is an equivalence and $B$ is connected, $a$ is an equivalence if and only if $b$ is an equivalence \cite[Lemma 6.2.3.16]{lurie-htt}.
\end{proof}

\begin{corollary+} \label{corr:unstable-connectivity}
Let $X \in \Shv_\Nis(\Sm_k)$ be $n$-connected, $n \ge -1$.
Then $L_\mot X$ is $n$-connected.
\end{corollary+}
\begin{proof}
The case $n=-1$ is trivial and the case $n=0$ is Lemma \ref{lemm:0-conn}.
We may thus assume that $X$ is connected, and pick a base point $*$.
Consider the fiber sequence \[ (L_\mot X)_{> n} \to L_\mot X \to (L_\mot X)_{\le n}. \]
Since the second and third space are motivically local, so is the first.
Since $X$ is $n$-connected, the composite $X \to (L_\mot X)_{\le n}$ is null and thus $X \to L_\mot X$ lifts to $(L_\mot X)_{> n}$.
But also $(L_\mot X)_{>n}$ is motivically local, so $X \to (L_\mot X)_{>n}$ factors through $L_\mot X$.
In other words the map $(L_\mot X)_{>n} \to L_\mot X$ has a section.
This implies that $\ul\pi_i L_\mot X = *$ for $i \le n$, as needed.
\end{proof}

\begin{corollary+}
Let $X \in \Shv_\Nis(\Sm_k)$ be connected.
Then $X$ is motivically local if and only if $\ul\pi_i X$ is strongly $\A^1$-invariant for all $i$.
\end{corollary+}
\begin{proof}
Necessity is Corollary \ref{cor:pi-i-strongly-inv}(1).
For sufficiency, we see as in the proof of Corollary \ref{cor:pi-i-strongly-inv}(2) that $X_{\le i}$ is $\A^1$-invariant for every $i$.
Since $\Shv_\Nis(\Sm_k)$ is Postnikov complete and $\A^1$-invariant sheaves are closed under limits, we conclude.
\end{proof}

\begin{corollary+}\label{corr:connectivity-stable}
Let $X_1 \to X_2 \to X_3 \in \Spc(k)_*$ be a cofiber sequence, with $X_i$ connected for each $i$.
If $X_1$ and $X_3$ are $d$-connective, then so is $X_2$.
Similarly if $X_1$ and $X_2$ are $d$-connective, so is $X_3$.
\end{corollary+}
\begin{proof}
If $X_1$ and $X_2$ are $d$-connective, then so is the cofiber when computed in Nisnevich sheaves, and we conclude by Corollary \ref{corr:unstable-connectivity}.
Now let $X_1, X_3$ be $d$-connective.
It will suffice to show that for every $(d-1)$-truncated motivic space $Z$ we have $\Map(X_2, Z) = *$ (take $Z=(X_2)_{<d}$).
But we have the fiber sequence \[ \Map(X_3, Z) \to \Map(X_2, Z) \to \Map(X_1, Z) \] in which the outer two terms are contractible, and hence so is the middle one.
\end{proof}

\subsection{Outline}
Let us summarize briefly the contents of these notes.
Note that there are also individual summaries at the beginning of each section.

We conclude the introduction with \S\ref{subsec:standard-results} on useful standard results and \S\ref{subsec:notation} summarizing some notation.
In \S\ref{sec:pi1} we study $\ul\pi_1$ of a motivic space and establish Theorem \ref{thm:pi1-strongly-inv}.

The rest of the notes is dedicated to proving Theorem \ref{thm:strongly-strictly-inv}, i.e., that strongly $\A^1$-invariant sheaves of abelian groups are strictly $\A^1$-invariant.
From here on, with very few exceptions, we assume that the base field is perfect.
We begin in \S\ref{sec:MWn} by defining the sheaf of \emph{naive} Milnor--Witt $K$-theory\footnote{This differs in general from Morel's \emph{unramified} Milnor--Witt $K$-theory, but is easier to define and ``good enough'' for our purposes.} and relating it to motivic homotopy theory.
Next in \S\ref{sec:contractions} we introduce the important concept of a \emph{contraction} of a strongly $\A^1$-invariant sheaf.
We explain how these contractions relate to local cohomology, and how contracted sheaves carry certain kinds of transfers.
In \S\ref{sec:gersten} we define and study the Gersten complex $C^*(X,M)$ of a strongly $\A^1$-invariant sheaf $M$ on a smooth scheme $X$.
Crucially we manage to identify $C^2(X,M)$ in terms of contractions of $M$.
This allows us to establish an important purity property for the Gersten complexes, at least in codimension $\le 2$ (Corollary \ref{cor:purity-codim2}).
Finally in \S\ref{sec:RS} we define and study the Rost--Schmid complex $C^*_{RS}(X,M)$ of $M$ on $X$.
This is modeled on the identification of $C^{*\le 2}(X,M)$, and in particular has an automatic purity property (almost) built in.
The complex turns out to be a resolution of $M$.
We use this to prove the main theorem, i.e., that $M$ is strictly $\A^1$-invariant.

\subsection{Standard results} \label{subsec:standard-results}
\subsubsection{Distinguished squares} \label{subsub:dist-squares}
A square of schemes
\begin{equation*}
\begin{CD}
V @>>> Y \\
@VVV @V{p}VV \\
U @>i>> X
\end{CD}
\end{equation*}
is called \emph{distinguished} if
\begin{itemize}
\item the square is cartesian,
\item $i$ is an open immersion,
\item $p$ is étale, and
\item $Y \setminus V \to X \setminus U$ is an isomorphism (of reduced schemes).
\end{itemize}
Such squares are \emph{motivically cocartesian}, that is, become pushout squares in $\Spc(k)$.
In fact they already become pushout squares in $\Shv_\Nis(\Sm_k)$ \cite[Proposition C.5(2)]{hoyois2015quadratic}.

\begin{remark} \label{rmk:excision}
Let $X \in \Sm_k$, $Z \subset X$ closed and $Y \to X$ an étale neighborhood (i.e. $Y \to X$ is étale and $Z_Y \to Z$ is an isomorphism).
We obtain a distinguished square by setting $U = X \setminus Z, V = Y \setminus Z_Y$.
The square being motivically cocartesian implies that the induced map on horizontal cofibers \[ Y/Y \setminus Z \xrightarrow{\wequi} X/X \setminus Z \] is an equivalence.
This is usually called \emph{excision}.
\end{remark}

\subsubsection{Cohomological dimension}
One standard fact that we shall use all the time is the following.
\begin{theorem} \label{thm:coh-dim}
Let $X$ be a qcqs scheme of Krull dimension $\le d$.
Then the Nisnevich and Zariski homotopy dimensions (in particular cohomological dimensions) of $X$ are bounded by $d$.
\end{theorem}
\begin{proof}
See \cite[Theorems 3.12 and 3.18]{clausen2019hyperdescent}.
\end{proof}

\subsubsection{Spheres} \label{subsub:spheres}
In motivic homotopy theory there are two basic ``spheres'': the ``topological circle'' $S^1 = * \amalg_{* \amalg *} *$, and the ``geometric circle'' $\Gm = (\A^1 \setminus 0, 1)$.
Both are objects in $\Spc(k)_*$.
Out of this we can build further spheres, using the notation \[ S^{i,j} = S^{i-j} \wedge \Gmp{j}. \]
This makes sense only if $i \ge j \ge 0$.
Other spaces that appear like they ought to be spheres can often be identified in terms of these basic ones.
In particular we have \cite[\S3 Corollary 2.18, Lemma 2.15]{A1-homotopy-theory} \cite[Proposition 4.40]{elmanto-primer} \[ \A^n \setminus 0 \wequi S^{2n-1,n} \quad\text{and}\quad \P^n/\P^{n-1} \wequi \A^n/\A^n \setminus 0 \wequi S^{2n,n}. \]
For example, $\P^1 \wequi S^{2,1}$.

Just as in topology, for any vector bundle $E$ on a smooth scheme $X$ we can form an associated Thom space, that is, a bundle of spheres over $X$.
It is given by the formula \[ \Th(E) := E/E \setminus 0_X. \]
Here $0_X \subset E$ denotes the closed subscheme which is the image of the zero section.

\begin{example}
If $X = *$ and $E = \scr O^n$ is (necessarily) trivial of rank $n$, then $\Th(\scr O^n) \wequi S^{2n,n}$.
More generally, if $X$ is arbitrary and $E = \scr O_X^n$, then \[ \Th(\scr O_X^n) = \A^n \times X/(\A^n \setminus 0) \times X \wequi \A^n/\A^n \setminus 0 \wedge X_+ \wequi S^{2n,n} \wedge X_+. \]
\end{example}

\begin{remark}
Suppose that $f: Y \to X$ is a morphism of smooth schemes and $F$ is a vector bundle on $Y$, together with a morphism of vector bundles $\tilde f: F \to f^* E$ which is fiberwise injective (that is, $\tilde f^{-1}(0_Y^{f^*E}) \subset 0_Y^F$).
Then there is a canonical induced map \[ \Th(\tilde f, f): \Th(F) \to \Th(E). \]
\end{remark}

\subsubsection{Purity} \label{subsub:purity}
Let $X \in \Sm_S$ and $Z \subset X$ be a closed subscheme such that $Z \to S$ is also smooth.
We call $(X, Z)$ a \emph{smooth closed pair}.
A morphism of smooth closed pairs $(X, Z) \to (X', Z')$ is a morphism $f: X \to Z$ such that $f^{-1}(Z') = Z$, as subschemes of $X$.
Recall that for any smooth closed pair $(X, Z)$, the conormal bundle of $Z/X$ is a vector bundle, and its dual is called the \emph{normal bundle} $N_{Z/X}$ \cite[Definition 19.21]{goertz-wedhorn-2}.

The \emph{homotopy purity theorem} is the following result.
For a proof, see, e.g., \cite[Theorem 3.23]{hoyois-equivariant}.
\begin{theorem} \label{thm:purity}
For every smooth closed pair $(X, Z)$ over $S$, there is a canonical equivalence in $\Spc(S)$ of the form \[ X/X \setminus Z \wequi \Th(N_{Z/X}). \]
This equivalence is natural in the smooth closed pair $(X,Z)$.
For a pair $(E,0_X)$ of a vector bundle $E$ over $X$ and its zero section, the equivalence arises from the canonical isomorphism $N_{0_X/E} \wequi E$.
\end{theorem}
Note that the normal bundle of $Z$ in $X$ is invariant under passing to étale neighborhoods of $Z$, consistent with excision (see Remark \ref{rmk:excision}).

\subsubsection{Essentially smooth base change}
Let $K/k$ be a separable field extension, so that $\Spec(K)$ is the cofiltered limit of a family of smooth $k$-varieties.
We have the base change functor \[ \Shv_\Nis(\Sm_k) \to \Shv_\Nis(\Sm_K), F \mapsto F|_K. \]
In this situation of \emph{essentially smooth base change}, the functor behaves particularly well.

\begin{lemma} \label{lemm:essentially-smooth-bc}
\begin{enumerate}
\item Let $\{X_\lambda\}$ be a cofiltered system of smooth $k$-varieties with $\lim_\lambda X_\lambda \in \Sm_\eta$.
  In this case one has \[ F|_\eta(\lim_\lambda X_\lambda) \wequi \colim_\lambda F(X_\lambda). \]
\item The functor $F \mapsto F|_\eta$ preserves $\A^1$-invariant Nisnevich sheaves of spaces.
\item The functor commutes with taking homotopy sheaves.
\item The functor preserves strongly/strictly $\A^1$-invariant sheaves.
\end{enumerate}
\end{lemma}
\begin{proof}
The formula in (1) holds for base change on the level of presheaves presheaves, and one needs to see that it preserves Nisnevich sheaves.
This is true because any Nisnevich cover of $\lim_\lambda X_\lambda$ is already defined at some finite stage.
(2), (3) and (4) are easy consequences.
\end{proof}

\subsubsection{Essentially smooth schemes} \label{subsub:ess-smooth}
Write $\Pro_\aff(\Ft_k)$ for the category of pro-objects \cite[Tag 05PX]{stacks-project} in finite type $k$-schemes, such that the transition maps are affine.
One may show that if $X_\bullet \in \Pro_\aff(\Ft_k)$ then $\lim X_\bullet \in \Sch_k$ exists, and the resulting functor \[ \lim: \Pro_\aff(\Ft_k) \to \Sch_k \] is fully faithful \cite[Proposition 8.13.5]{EGAIV}.
Its essential image is called the category of \emph{essentially finite type schemes}.
Write $\Pro_\et(\Sm_k) \subset \Pro_\aff(\Ft_k)$ for the full subcategory on those objects $X_\bullet$ such that each $X_i \in \Sm_k$ and each transition map $X_j \to X_i$ is étale.
The essential image in $\Sch_k$ is called the category of \emph{essentially smooth schemes}.

\begin{example}
Let $X \in \Sm_k$ and $x \in X$.
The \emph{localization of $X$ in $x$} is \[ \bigcap_{x \ni U} U. \]
Here the intersection is over all open neighborhoods of $x$.
The set of such neighborhoods is filtered, and hence $X_x$ is an essentially smooth scheme.

Similarly the \emph{henselization $X_x^h$ of $X$ in $x$} is the (cofiltered) limit over all étale neighborhoods of $x$ in $X$, that is, all pairs $(V \to X, \tilde x)$ with $V \to X$ étale and $\tilde x \in V$ such that $\tilde x$ lies over $x$ and the induced map $\tilde x \to x$ is an isomorphism.
\end{example}

Let $F: \Sm_k \to \scr C$ be any functor and $X \in \Pro_\et(\Sm_k)$, that is, $X$ is an essentially smooth scheme.
Applying $F$, we obtain $FX \in \Pro(\scr C)$.
For many purposes, we can work with $\Pro(\scr C)$ the same way we can work with $\scr C$, and we will often silently do this.
For example, for $Y \in \scr C$ we have the space $\Map(X, Y)$.
What this actually means is $\Map_{\Pro(\scr C)}(X, cY)$, where $c: \scr C \to \Pro(\scr C)$ denotes the ``constant pro-object'' functor.
In other words, if $X = \lim_i X_i$ is a presentation of $X$ as a cofiltered limit of smooth schemes along étale maps, then \[ \Map(FX, Y) = \colim_i \Map(FX_i, Y). \]
It is very convenient that this does not depend on the choice of presentation.

\begin{remark}
Everything we have said here would have worked equally well if we defined essentially smooth schemes to be cofiltered inverse limits of smooth schemes along affine transition maps.
To the best of my knowledge, the restriction to étale transition maps was first considered in \cite{A1-alg-top}, and has become standard in motivic homotopy theory.
This ensures, for example, that essentially smooth schemes have finite dimension.
\end{remark}

\subsubsection{Smooth retractions} \label{subsubs:smooth-retract}
The geometric intuition behind purity (\S\ref{subsub:purity}) is that smooth closed pairs ``look étale-locally like $(\A^{n+c}, \A^n)$''.
One can make this precise: a \emph{parametrization} of a smooth closed pair $(X, Z)$ over $S$ consists is a morphism of smooth closed pairs $(X,Z) \to (\A^{n+c}_S, \A^n_S)$ over $S$, such that $X \to \A^{n+c}_S$ is étale.
For any $z \in Z$ there exists an open neighborhood $U$ of $Z$ in $X$ and a parametrization of $(U, U \cap Z)$ \cite[Tag 0FUE]{stacks-project}.

\begin{lemma}
Let $\varphi: (X, Z) \to (\A^{n+c}_S, \A^n_S)$ be a parametrized smooth closed pair over $S$.
There exists an étale neighborhood $X' \to X$ of $Z$ such that $Z \hookrightarrow X'$ admits a smooth retraction.
\end{lemma}
\begin{proof}
Consider the following cartesian square
\begin{equation*}
\begin{CD}
X_1 @>{\beta}>> X \\
@VVV    @V{\varphi}VV \\
\A^c_Z @>{\alpha}>> \A^{n+c}_S.
\end{CD}
\end{equation*}
The map $\alpha$ is the product with $\A^c$ of the étale map $Z \to \A^n_S$, and thus is étale.
We deduce that $\beta$ is étale.
By construction, the canonical map $Z \to X_1$ admits a smooth retraction (via the étale projection to $\A^c_Z$).
Since $Z \to \beta^{-1}(Z)$ is a map between étale $Z$-schemes it is étale \cite[Tag 02GW]{stacks-project}, whence open.
Consequently we may set $X' = X_1 \setminus (\beta^{-1}(Z) \setminus Z)$.
\end{proof}

The following is the main way we will use this result.
\begin{corollary} \label{cor:ess-smooth-retraction}
Let $k$ be a perfect field, $X \in \Sm_k$ and $x \in X$.
Then $x \to X_x^h$ admits an essentially smooth retraction.
\end{corollary}
\begin{proof}
By generic smoothness \cite[Tag 0B8X]{stacks-project}, up to shrinking $X$ around $x$, we may assume that $\overline{\{x\}} =: Z$ is smooth.
Then shrinking $X$ further, we can find a parametrization of the smooth closed pair $(X, Z)$.
We may thus pass to an étale neighborhood $X'$ of $Z$ (whence $x$) such that $Z \to X'$ has a smooth retraction $X' \to Z$.
The induced morphism $X_x^h \wequi (X')_x^h \to Z$ factors through $x \hookrightarrow Z$, yielding the desired essentially smooth retraction.
\end{proof}

\begin{remark} \label{rmk:reduce-to-closed-point}
Thus in order to prove something about $X_x^h$ over $k$, we may instead prove something about (the isomorphic scheme) $Y_y^h$ over $k'$.
Here $k'$ is some (in general non-perfect) field extension of $k$ (in fact the residue field of $x$), $Y \in \Sm_{k'}$, and $y$ is a \emph{rational} point of $Y$.

In other words, if we are willing to pass to a non-perfect ground field, we may assume that $x$ is a rational point.
\end{remark}

\begin{remark} \label{rmk:smooth-retraction-dimensionality}
The proof of Corollary \ref{cor:ess-smooth-retraction} shows the following slightly stronger statement: if $x \in X$ is a point of codimension $c$, then $X_x^h$ can be written as a cofiltered limit of smooth $x$-schemes of dimension $c$ (with affine étale transition maps).
\end{remark}

\subsubsection{Cotangent complex of lci morphisms} \label{subsub:cotangent-cx}
Let $X \to Y$ be a morphism of schemes.
Then we have the \emph{cotangent complex} $L_{X/Y}$ (see, e.g., \cite[Tag 08P5]{stacks-project}), which is a complex of quasi-coherent sheaves on $X$.
If $X \to Y$ is lci (e.g., any morphism between smooth schemes over a field), then $L_{X/Y}$ is perfect \cite[Tag 08SL]{stacks-project} and hence has a \emph{determinant} \cite[Tag 0FJW]{stacks-project}.
This is a line bundle on $X$ which we denote by $\omega_{X/Y}$.

\begin{example} \label{ex:omega-lci}
Suppose that $X \to Y$ factors as $X \hookrightarrow A \to Y$, with $X \to A$ a regular closed immersion and $A \to Y$ smooth.
Denote by $C_{X/A}$ the conormal bundle and by $\Omega^1_{A/Y}$ the sheaf of differentials.
In this case, we have a representation \cite[proof of Tag 08SL]{stacks-project} \[ L_{X/Y} \wequi [C_{X/A} \xrightarrow{d} \Omega^1_{A/Y}|_X]; \] here the $\Omega^1$-term is placed in degree zero.
It follows that \[ \omega_{X/Y} \wequi \det(C_{X/A})^* \otimes \det \Omega^1_{A/Y} \wequi \det(N_{X/A}) \otimes \det(\Omega^1_{A/Y}). \]
\end{example}

\begin{remark} \label{rmk:omega-composition}
One of the key properties \cite[Tag 08QR]{stacks-project} of the cotangent complex is that, given morphisms $X \to Y \to Z$, there is a cofiber sequence \[ L_{Y/Z}|_X \to L_{X/Z} \to L_{X/Y}. \]
Assuming the morphisms are lci, this translates into a canonical isomorphism of line bundles \[ \omega_{X/Z} \wequi \omega_{X/Y} \otimes \omega_{Y/Z}|_X. \]
\end{remark}

\subsection{Notation} \label{subsec:notation}
Here is a non-exhaustive list of notation we use commonly.
\begin{center}
\begin{tabular}{rl}
$\Spc(S)$                 & $\infty$-category of motivic spaces over $S$ \\
$\PSh(\scr C)$            & $\infty$-category of presheaves (of spaces) on $\scr C$ \\
$\Shv_\Nis(\Sm_S)$        & $\infty$-category of Nisnevich sheaves (of spaces) on $\Sm_K$ \\
$H^*(X,F)$                & Nisnevich cohomology groups \\
$H^*_Z(X,F)$              & cohomology with support \\
$H^*_x(X,F)$              & local cohomology \\
$X^{{(c)}} \subset X$     & points of codimension $c$ on a scheme \\
$X_{(d)} \subset X$       & points of dimension $d$ on a scheme \\
$X_x, X_x^h$              & localization and Henselization of a scheme in a point \\
$\omega_{X/Y}$            & determinant of the cotangent complex of $X \to Y$ \\
$\Z[\Gmp{n}]$             & sheaf of abelian groups freely generated by the pointed sheaf of sets $\Gmp{n}$ \\
\end{tabular}
\end{center}

\subsection{Acknowledgements}
First and foremost I would like to thank Fabien Morel for having the vision to propose this theory, to write it up in his book \cite{A1-alg-top}, and for teaching me about all of it.
I would further like to thank Aravind Asok and Mike Hopkins for numerous discussions about the material.
Finally I want to thank Niels Feld for his thorough reading of a draft of this text and the many comments he provided.

\section{Strong $\A^1$-invariance of $\ul\pi_1$} \label{sec:pi1}
In this section we prove that if $X$ is a pointed motivic space (over a field), then $\ul\pi_1(X)$ is strongly $\A^1$-invariant.
This heavily relies on Gabber's presentation lemma, which we introduce in \S\ref{subsec:gabber}.
We draw some first consequences for $\ul\pi_1$ in \S\ref{subsec:pi1-1}, and finally prove the main result in \S\ref{subsec:pi1-2}.

\subsection{Gabber's lemma} \label{subsec:gabber}
Gabber's presentation lemma is a geometric result, which is used in a variety of foundational results in motivic homotopy theory.
Its formulation is as follows.

\begin{theorem}[Gabber's lemma] \label{thm:gabber-lemma}
Let $k$ be a field, $X \in \Sm_k$, $Z \subset X$ closed and everywhere of positive codimension, and $z \in Z$.
Then there exist:
\begin{enumerate}
\item An open neighborhood $X' \subset X$ of $z$.
\item A smooth affine $k$-scheme $S$ and an étale map $\varphi: X' \to \A^1_S$,
\end{enumerate}
such that
\begin{enumerate}[(a)]
\item The composite $Z' := Z \cap X' \to X' \xrightarrow{\varphi} \A^1_S \to S$ is finite.
\item The composite $Z' \cap X' \to \A^1_S$ is a closed immersion.
\item The preimage of $Z' \subset \A^1_S$ under the map $X' \to \A^1_S$ coincides with $Z'$.
\end{enumerate}
\end{theorem}

Conditions (2) and (3) of Theorem \ref{thm:gabber-lemma} assert that $X' \to \A^1_S$ is an \emph{étale neighborhood} of (the image of) $Z'$.
This ensures that the commutative square
\begin{equation*}
\begin{CD}
X' \setminus Z' @>>> X' \\
@VVV               @VVV \\
\A^1_S \setminus Z' @>>> \A^1_S
\end{CD}
\end{equation*}
is a distinguished Nisnevich square (in particular, is cartesian).
The square thus becomes a pushout upon applying $L_\mot$ (see \S\ref{subsub:dist-squares}), and hence the induced map \[ X'/X' \setminus Z' \to \A^1_S/\A^1_S \setminus Z' \] is a motivic equivalence.
This is the main way in which we use Gabber's lemma.

\begin{remark}[proof of Gabber's lemma]
If the field $k$ is infinite, then the proof of Theorem \ref{thm:gabber-lemma} is a fairly straightforward general position argument \cite[\S3]{gabber-presentation-lemma}.
The case of finite base fields is significantly more delicate \cite{gabber-lemma-finite-fields}.
\end{remark}

Here is a sample application.
\begin{proposition} \label{prop:XmodUconn}
Let $k$ be any field, $X \in \Sm_k$ and $Z \subset X$ closed, everywhere of positive codimension.
Then $L_\mot(X/X \setminus Z)$ is connected.
\end{proposition}
\begin{proof}
Applying Lemma \ref{thm:gabber-lemma} to every point in $X$ we find a Zariski covering $\{X_i \to X\}_{i \in I}$ such that \[ X_i/X_i \setminus Z \wequi \A^1_{S_i}/\A^1_{S_i} \setminus Z \stackrel{\mot}{\wequi} \tilde\Sigma (\A^1_{S_i} \setminus Z). \]
Here $\tilde\Sigma$ denotes the reduced suspension, and so in particular this motivic space is connected (use Lemma \ref{lemm:0-conn}).
Set $U = \amalg_i X_i$, so that $U \to X$ is a Nisnevich cover of $X$.
Then $X$ is the colimit in $\Shv_\Nis(\Sm_k)$ of the associated Čech nerve $U^\bullet$, with $U^n = U \times_X U \times_X \dots \times_X U$.
Similarly $X \setminus Z$ is the Čech nerve of $(U \setminus Z)^\bullet$.
We deduce that \[ X/X \setminus Z \stackrel{\mot}{\wequi} |L_\mot(U^\bullet/U^\bullet \setminus Z)|_\Nis, \] where $|(\ph)|_\Nis$ denotes the geometric realization in the topos $\Shv_\Nis(\Sm_k)$.
Recall that the geometric realization of a simplicial object (in a topos) is connected if and only if the object in zeroth degree is connected.
It thus remains to show that $L_\mot(U/U \setminus Z)$ is connected.
But this is \[ U/U \setminus Z = \bigvee_i X_i/X_i \setminus Z, \] which is connected since each $X_i/X_i \setminus Z$ is.
\end{proof}

\begin{corollary} \label{cor:A1-inv-unramified}
Let $G$ be an $\A^1$-invariant Nisnevich sheaf of groups.
Then $G$ is unramified, in the sense that for every dense open immersion $U \to X$, the map $G(X) \to G(U)$ is an injection.
\end{corollary}
\begin{proof}
Indeed mapping the cofiber sequence $U_+ \to X_+ \to X/U$ into $G$ we obtain a fiber sequence \[ \Map(X/U, G) \to G(X) \to G(U). \]
The first term is contractible by Proposition \ref{prop:XmodUconn}, whence $G(X) \to G(U)$ is injective (here we use that $G$ is a sheaf of groups).
\end{proof}

\begin{remark} \label{rmk:unramifiedness-use}
Let $X \in \Sm_k$ be connected with generic point $\eta$.
Corollary \ref{cor:A1-inv-unramified} implies that for $G$ an $\A^1$-invariant Nisnevich sheaf of groups, the canonical map $F(X) \to F(\eta)$ to the stalk is injective.
\end{remark}
\begin{example}\label{ex:unramifiedness-use}
Let $K$ be a finitely generated, separable field extension of $k$, i.e. the fraction field of a smooth $k$-scheme.
Write $K(t)$ for the purely transcendental extension of $K$.
Let $G$ be an $\A^1$-invariant Nisnevich sheaf of groups.
Then the composite \[ G(K) \wequi G(\A^1_K) \to G(K(t)) \] is injective.
This trick can sometimes be used to reduce problems to the case infinite fields (but beware that $K(t)$ is not perfect, even if $k$ is).
\end{example}

We can strengthen Proposition \ref{prop:XmodUconn} to establish higher connectivity of $X/X\setminus Z$, depending on the codimension of $Z$.
For this we need the full strength of our main theorems.
\begin{corollary+}
Let $k$ be a perfect field, $X \in \Sm_k$, and $Z \subset X$ closed, everywhere of codimension $\ge d$.
Then $X/X \setminus Z$ is $d$-connective.
\end{corollary+}
\begin{proof}
Suppose first that $Z$ is smooth.
Then by purity (see \S\ref{subsub:purity}) we find that $X/X \setminus Z \wequi \Th(N_{Z/X})$, which is $d$-connective (working locally on $Z$, we may assume that $N_{Z/X}$ is trivial, so that $\Th(N_{Z/X}) \wequi \Sigma^d \Gmp{d} \wedge Z_+$ (see \S\ref{subsub:spheres}), which is $d$-connective by the unstable connectivity theorem, Corollary \ref{corr:unstable-connectivity}).
Now let $Z$ be general.
We prove the result by induction on the dimension of $Z$.
In the base case $Z = \emptyset$ nothing needs to be proved.
By generic smoothness \cite[Tag 0B8X]{stacks-project} there is a closed subscheme $Z_1 \subset Z$ such that $Z \setminus Z_1$ is smooth, and $\dim Z_1 < \dim Z$.
We thus have a cofiber sequence \[ X \setminus Z/X \setminus Z_1 \to X/X \setminus Z \to X/X \setminus Z_1. \]
Since each term is connected by Proposition \ref{prop:XmodUconn}, by Corollary \ref{corr:connectivity-stable} it suffices to show that the two outer terms are $d$-connective.
Since $Z \setminus Z_1$ is smooth this holds for the left hand term by what we just said; since $\dim Z_1 < \dim Z$ it holds for the right hand term by induction.
\end{proof}

Let us point out the following variant which does not require the main theorems.
\begin{lemma} \label{lem:remove-codim-2}
Let $k$ be a perfect field, $X \in \Sm_k$, and $Z \subset X$ closed, everywhere of codimension $\ge 2$.
Then for any strongly $\A^1$-invariant sheaf of groups $G$ we have $G(X) \wequi G(X \setminus Z)$.
\end{lemma}
\begin{proof}
We must prove that $G(X) \to G(X \setminus Z)$ is surjective.
As in the proof of the previous corollary, we argue by induction on the dimension of $Z$.
Let $Z_1 \subset Z$ with $\dim Z_1 < \dim Z$ and $Z \setminus Z_1$ smooth.
(In the base case, where $Z$ is smooth, we can pick $Z_1 = \emptyset$.)
Mapping the cofiber sequence \[ \Th(N_{Z \setminus Z_1 /X \setminus Z_1}) \wequi X \setminus Z_1/X \setminus Z \to \Sigma X_+ \setminus Z \to \Sigma X_+ \setminus Z_1 \] into $K(G,1)$, it will suffice to show that if $Z' \in \Sm_k$ and $V$ is a vector bundle on $Z'$ of rank $\ge 2$, then $\Map(\Th(V), K(G,1)) = *$.
Writing $\Th(V)$ as the colimit of a Čech nerve on which $V$ trivializes reduces to the case where $V$ is trivial.
But then $\Th(V)$ is a double suspension, whence the desired vanishing.
\end{proof}

Here is another typical application.
\begin{proposition} \label{prop:gersten-injectivity}
Let $F \in \Spc(k)_*$, $X \in \Sm_k$ connected, $x \in X$ arbitrary and $\eta \in X$ the generic point.
The canonical map \[ \pi_0 F(X_x) \to \pi_0 F(\eta) \] has trivial fiber.
\end{proposition}
\begin{proof}
\textbf{First case: $k$ infinite.}

Let $X_1$ be an open neighborhood of $x$ in $X$ and $Z \subset X_1$ a closed subset containing $x$.
Applying Gabber's Lemma to $X_1, Z, x$ we obtain a smaller open neighborhood $X_2$ together with an étale neighborhood $X_2 \to \A^1_S$ of $Z \cap X_2$.
Since $Z \cap X_2 \to S$ is finite, the canonical map $\A^1_S \setminus Z \cap X_2 \to S$ has a section over an open neighborhood $S_1$ of the image of $x$ in $S$ (this is the only place where we use that $k$ is infinite).\footnote{Since $k$ is infinite, $\A^1_{\bar s} \setminus Z_{\bar s}$ has a rational point, where $\bar s \in S$ is a closed specialization of the image of $x$. This defines an element $f \in \scr O(\bar s)$, which lifts to any affine open neighborhood $U$ of $\bar s$. We obtain $\tilde f: U \to \A^1_U$. Now $\tilde f^{-1}(Z) \subset U$ is a closed subscheme which, by construction, does not contain $\bar s$. Set $S_1 = U \setminus \tilde f^{-1}(Z)$ to ensure that $\tilde f|_{S_1}$ factors through $\A^1_S \setminus Z$.}
Write $X_3$ for the preimage of $S_1$ in $X_2$.
Consider the diagram of cofiber sequences
\begin{equation*}
\begin{CD}
X_3 \setminus Z @>>> X_3 @>>> X_3/X_3 \setminus Z \\
@VVV               @VVV         @V{\wequi}V{\mot}V \\
\A^1_{S_1} \setminus Z @>>> \A^1_{S_1} @>>> \A^1_{S^1}/\A^1_{S_1} \setminus Z.
\end{CD}
\end{equation*}
Taking homotopy classes of pointed maps into $F$ this translates into a diagram of fiber sequences of sets
\begin{equation}
\begin{CD}
\pi_0 F(X_3 \setminus Z) @<d<< \pi_0 F(X_3) @<c<< \pi_0 F(X_3/X_3 \setminus Z) \\
@AAA               @AAA         @A{\wequi}AA \\
\pi_0 F(\A^1_{S_1} \setminus Z) @<a<< \pi_0F(\A^1_{S_1}) @<b<< \pi_0F(\A^1_{S^1}/\A^1_{S_1} \setminus Z) \\
@. @V{\wequi}VV \\
@. \pi_0 F(S_1).
\end{CD}
\end{equation}
Since $\A^1_{S_1} \setminus Z \to S_1$ has a section, the map denoted $a$ is an injection, and so the map denoted $b$ is null.
It follows that $c$ is null and so $d$ has trivial fiber.

Taking the filtering colimit over all open neighborhoods $X_1$ and all closed subsets $Z$ yields the desired conclusion.

\textbf{Second case: $k$ finite.}

We will prove by induction on $d$ the following statement:
For every $X \in \Sm_k$ connected of dimension $\le d$, $x \in X$, $Z \subset X$ a proper closed subset, there exists an open neighborhood $X'$ of $x$ in $X$ such that $\pi_0 F(X') \to \pi_0 F(X' \setminus Z)$ has trivial fiber.
Taking the filtering colimit over all such $Z$ yields the desired conclusion as before.
If $\dim X = 0$ then $Z=\emptyset$ and there is nothing to prove.

Now let $\dim X = 1$.
If $x \not\in Z$ then we can take $X' = X \setminus Z$.
Thus $x \in Z$ and, since $\dim X = 1$, we have $Z = Z_1 \amalg \{x\}$.
Replacing $X$ by $X \setminus Z_1$, we may assume that $Z=\{x\}$.
Applying Gabber's lemma as before, we obtain $X_1 \to \A^1_S$, where $S$ has dimension $0$.
Let $s \in S$ be the image of $x$.
The map $\A^1_s \setminus Z \to s$ has a section, since $\A^1_s$ has at least two rational points, and $Z$ consists of only one point.
It follows as before that $\A^1_S \setminus Z \to S$ has a section in an open neighborhood of $s$, and the desired result follows by the same argument as before.

Now let $\dim X > 1$.
Construct $X_1 \to \A^1_S$ as before.
The generic point of $S$ has infinite residue field (since $\dim S = \dim X -1 > 0$), so $\A^1_S \setminus Z \to S$ has a section on the complement of some proper closed subscheme $W \subset S$.
Applying the inductive assumption to $(S,s,W)$ we find some open neighborhood $s \in S_1 \subset S$ such that $\pi_0 F(S_1) \to \pi_0 F(S_1 \setminus W)$ has trivial fiber.
In the commutative square
\begin{equation*}
\begin{CD}
\pi_0 F(\A^1_{S_1}) @>>> F(\A^1_{S_1} \setminus Z) \\
@VVV                         @VVV \\
\pi_0 F(\A^1_{S_1 \setminus W}) @>>> \pi_0 F(\A^1_{S_1 \setminus W} \setminus Z)
\end{CD}
\end{equation*}
the composite via the bottom left hand corner thus has trivial fiber (for the vertical morphism this holds by construction of $S_1$ and $\A^1$-invariance of $F$, whereas for the horizontal morphism this holds by the existence of a section over $S \setminus W$), whence so does the top map.
Now the argument proceeds as before.
\end{proof}

\begin{corollary} \label{cor:connectivity-crit}
Let $F \in \Spc(k)_*$.
Then $F$ is connected if and only if for every finitely generated, separable field extension $K/k$ we have $\ul\pi_0(F)(K) = *$.
\end{corollary}
\begin{proof}
We must show that $\ul\pi_0(F) \to *$ is an isomorphism, for which it suffices to prove that $\ul\pi_0(F)(X_x) = *$ for every $X \in \Sm_k$ and $x \in X$.
This follows from our assumption and Proposition \ref{prop:gersten-injectivity}.
\end{proof}

\begin{corollary} \label{cor:Zariski-vanishing}
Let $F$ be strongly invariant (respectively strictly invariant) and $X \in \Sm_k$.
Then $H^i_\Zar(X, F) = H^i_\Nis(X, F)$ for $i \le 1$ (respectively $i \in \Z$).
\end{corollary}
\begin{proof}
By standard arguments, it suffices to show that $H^i_\Nis(X_x, F) = 0$ for all $x \in X$ and $i \ne 0$.\footnote{Write $\Gamma_\Nis F, \Gamma_\Zar F: \Sm_k^\op \to \SH$ for the (pre)sheaves of spectra of derived global sections. Both are Zariski sheaves, so it suffices to show that $\Gamma_\Zar F \to \Gamma_\Nis F$ is an equivalence on Zariski stalks. This translates into the stated condition.}
Now by Proposition \ref{prop:gersten-injectivity} applied to $K(F, i)$, the canonical map $H^i(X_x, F) \to H^i(\eta, F)$ has trivial fiber, where $\eta$ is the generic point of the component of $x$.
But since fields have Nisnevich cohomological dimension $0$, the target set has only one point, whence the same is true for the source.
\end{proof}

\subsection{Some consequences for $\ul\pi_1$} \label{subsec:pi1-1}
Let $X \in \Spc(k)_*$ be connected and set $G = \ul\pi_1 X$.

We first observe that $G$ is unramified.
\begin{lemma} \label{lemm:pi1-unramified}
Let $X \in \Sm_k$ be connected with generic point $\eta$.
Then $G(X) \to G(\eta)$ is injective.
\end{lemma}
\begin{proof}
For $x \in X$ write $\eta_x \in X_x^h$ for the generic point of the henselization.
Consider the commutative square
\begin{equation*}
\begin{CD}
G(X) @>>> \prod_{x\in X} G(X_x^h) \\
@VVV         @VVV \\
G(\eta) @>>> \prod_{x \in X} G(\eta_x).
\end{CD}
\end{equation*}
Applying Proposition \ref{prop:gersten-injectivity} to $\Omega X$ we see that $G(X_x^h) \to G(\eta_x)$ is injective, and so the right hand map is injective.
The top map is injective since $G$ is a sheaf; it follows that the left hand map is injective.
\end{proof}

We want a similar result for $H^1(\ph, G)$.
We begin with the following abstract fact.
\begin{lemma} \label{lemm:compute-coh}
Let $\scr C$ be $\infty$-topos and $G \to E \in \scr C_{\le 0}$ be an injection of group objects.
Suppose that $H^{1}(*,E)=*$.
Then $H^{1}(*,G)\wequi H^{0}(*,E/G)/H^{0}(*,E).$
\end{lemma}
\begin{proof}
Take global sections in the fiber sequence $E/G\to BG\to BE$.
By assumption $\Gamma (BE)$ is connected and hence $\Gamma (BG)\wequi \Gamma (E/G)_{h\Omega \Gamma (BE)}$ (write $\Gamma (BE) \wequi *_{h \Omega \Gamma (BE)}$ \cite[Lemma 7.2.2.1]{lurie-htt} and use universality of colimits \todo{more explicit reference?}).
The claim follows. 
\end{proof}

Write $G^c$ for the constant sheaf corresponding to $G$, that is, if $U$ is connected with generic point $\eta$, then $G^c(U) = G(\eta)$.
Then $G^c$ is acyclic (see, e.g., Lemma \ref{lemm:acyclic-sheaves}) and by Lemma \ref{lemm:pi1-unramified} the canonical map $G\to G^c$ is injective.
Lemma \ref{lemm:compute-coh} thus applies and we deduce that (note $H^0(U, G^c) \wequi G(\eta) \wequi H^0(U_2, G^c)$) \begin{equation} \label{eq:inj-crit} \begin{gathered}\text{In order for $H^1(U,G) \to H^1(U_2,G)$ to be injective,} \\ \text{it suffices that $H^0(U,G^c/G) \to H^0(U_2, G^c/G)$ is injective.} \end{gathered} \end{equation}

Let $U_1 \subset U_2 \subset U$ be open subsets, with $U_i$ of codimension $i$. From the cofiber sequences \[ U \to U/U_1 \to \Sigma U_1 \to \Sigma U \quad\text{and}\quad U_2/U_1 \to U/U_1 \to U/U_2 \]\NB{add plus subscripts?} we obtain parts of ``long exact sequences'' \[ [\Sigma U, X] \to [\Sigma U_1, X] \to [U/U_1, X] \to [U, X] \quad\text{and}\quad [U/U_2, X] \to [U/U_1, X] \to [U_2/U_1, X]. \]
Here as usual we have to be careful with the meaning of exactness, in that these are in general only sequences of pointed sets.
Consider now the case where $U$ is henselian local and take the further colimit over all $U_1, U_2$.
We can now recognize certain terms in the sequences, which then take the following form: \[ G(U) \to G(\eta) \to \colim [U/U_1, X] \to *, \quad\text{and} \] \[ \colim [U/U_2, X] \xrightarrow{*} \colim [U/U_1, X] \to \prod^w_{x \in U^{(1)}}[\Sigma^{2,1}x, X] =: C^2(U). \]
Here $\prod^w$ means the weak product, i.e. filtered colimit of finite products.
(For the identification of $C^2(U)$, we can argue as follows.
By a cofinality argument, we can assume that $U_2 \setminus U_1 = \amalg_{i=1}^n Z_i$, where $Z_i \subset U_2$ is smooth, connected, and $N_{Z_i/U_2}$ is trivial of rank $1$.
Then $U_2/U_1 \wequi \vee_{i=1}^n \Sigma^{2,1}Z_{i+}$.
In the colimit, the number of $Z_i$ grows, while each $Z_i$ shrinks to its generic point, yielding the claimed expression.)

\begin{lemma} \label{lemm:pi1-map-null}
The map marked $*$ is zero.
\end{lemma}
\begin{proof}
Let $\tilde U$ be a smooth scheme, $Z_2 \subset \tilde U$ closed of codimension $2$, $z \in Z_2$.
We shall show that after shrinking $\tilde U$ around $z$, there exists $Z_2 \subset Z_1^0 \subset \tilde X$ with $Z_1$ of codimension $1$, such that $U/U \setminus Z_1^0 \to U/U \setminus Z_2$ becomes null homotopic in $\Spc(k)_*$.
Note that if $Z_1^0 \subset Z_1 \subset \tilde U$ then $U/U \setminus Z_1 \to U/U \setminus Z_2$ factors through $U/U \setminus Z_1^0 \to U/U \setminus Z_2$ and hence is also null.
Taking the colimit over schemes $\tilde U$ approximating $U$, all $Z_1$ and all (sufficiently large) $Z_2$ proves the result.

We now apply Gabber's lemma (Theorem \ref{thm:gabber-lemma}) to $(\tilde U, Z_2, z)$.\NB{This is basically redoing Proposition \ref{prop:gersten-injectivity}, but with a slicker argument.}
Up to shrinking $U$, we thus obtain an étale neighborhood $\varphi: U \to \A^1_S$ of $Z_2$.
Write $V$ for the image of $Z_2$ in $S$, which is a closed subscheme of codimension $\ge 1$.
Put $Z_1^0 = \varphi^{-1}(\A^1_V)$; this is codimension $\ge 1$ and contains $Z_2$.
The commutative square
\begin{equation*}
\begin{CD}
@. U/U \setminus Z_1^0 @>>> U/U \setminus Z_2 \\
@. @VVV                       @VV{\wequi}V \\
S/S\setminus V @>{0}>{\wequi}> \A^1_S/\A^1_S \setminus \A^1_V @>>> \A^1_S/\A^1_S \setminus Z_2 @>{\wequi}>> \P^1_S/\P^1_S \setminus Z
\end{CD}
\end{equation*}
shows that it suffices to prove the bottom composite is null (note that $Z \subset \P^1_S$ is closed, $Z \to S$ being finite).
Let $\lambda: S \to \P^1_S$ be any section.
Noting that $Z \subset \P^1_V$ we see that $\lambda$ induces a map $\lambda': S/S \setminus V \to \P^1_S/\P^1_S \setminus Z$.
For $\lambda = 0$, this is the map we are trying to prove is null.
For $\lambda = \infty$, the image of $\lambda$ is disjoint from $Z$ and hence $\lambda'$ is null.
Since any two maps $\lambda$ are $\A^1$-homotopic (via a linear homotopy), this concludes the proof.
\end{proof}

\begin{proposition} \label{prop:inj-H1}
For an open immersion $U_2 \to U$ with complement of codimension $\ge 2$, the canonical map $H^1(U,G) \to H^1(U_2,G)$ is injective.
\end{proposition}
\begin{proof}
It follows from Lemma \ref{lemm:pi1-map-null} that $G(\eta)/G(U) \to C^2(U)$ has trivial fiber.
Moreover, $C^2(U)$ has a natural action by $G(\eta)$\todo{details} which makes this map equivariant, and thus since the source consists of a single orbit, the map is actually injective.
Recall that here $U$ was henselian local.

Now we consider general $U$ again.
By considering the henselizations in all points, we deduce that \[ H^0(U, G^c/G) \to C^2(U) := \prod^w_{x \in U^{(1)}}[\Sigma^{2,1}x, U] \] is injective.
Since $C^2(U) \wequi C^2(U_2)$, the desired result follows via \eqref{eq:inj-crit}.
\end{proof}

\subsection{Main result}\label{subsec:pi1-2}
\begin{remark}
Let $F$ be a presheaf of sets on $\Sm_k$, and $X \in \Sm_k$.
We have the retraction $X \to \A^1 \times X \to X$ consisting of inclusion at $1$ (say) and projection.
It follows that the following three statements are equivalent:
\begin{enumerate}
\item $F(X) \wequi F(\A^1 \times X)$
\item $F(X) \to F(\A^1 \times X)$ is surjective
\item $F(\A^1 \times X) \to F(X)$ is injective.
\end{enumerate}
We will use this repeatedly in the sequel.
\end{remark}

\begin{lemma} \label{lemm:pi1-obstruction}
Let $\scr X$ be an $\infty$-topos, $X \in \scr X$ and $U \in \scr X$ an object of homotopy dimension $\le n+1$.
The map $[U, X] \to [U, \tau_{\le n} X]$ is surjective.
\end{lemma}
\begin{proof}
Pulling back along a fixed map $U \to \tau_{\le n} X$ and working in $\scr X_{/U}$, we may assume that $U=*$ and $X$ is $n$-connected (use that $n$-connected maps are stable under base change \cite[Proposition 6.5.1.6.6(6)]{lurie-htt}).
By the definition of homotopy dimension \cite[Definition 7.2.1.1]{lurie-htt}, it follows that $[*, X] \ne \emptyset$, as needed.
\end{proof}

\begin{lemma} \label{lemm:A1-inv-dimsmall}
Let $X \in \Spc(k)$ be connected and $U$ an essentially smooth scheme.
Set $G = \ul\pi_1 X$.
\begin{enumerate}
\item If $\dim U \le 1$ then $H^1(U, G) \wequi H^1(U \times \A^1, G)$.
\item If $\dim U = 0$ then $G(U) \wequi G(\A^1 \times U)$.
\end{enumerate}
\end{lemma}
\begin{proof}
We know that $U \times \A^1 \in \Shv_\Nis(\Sm_k)$ has homotopy dimension $\le \dim U + 1$ (combine the bound on the Nisnevich homotopy dimension by the Krull dimension from Theorem \ref{thm:coh-dim} with the fact that $\dim \A^1_U) = \dim U + 1$\NB{ref?}).

(1) Applying Lemma \ref{lemm:pi1-obstruction} with $n=1$ to $X \in \Shv_\Nis(\Sm_k)$ and $U \times \A^1$, we learn that the top horizontal map in the following commutative diagram is surjective
\begin{equation*}
\begin{CD}
[U \times \A^1, X] @>>> [U \times \A^1, \tau_{\le 1} X] \wequi H^1(U \times \A^1, G) \\
@AAA                         @AAA \\
[U , X] @>>> [U , \tau_{\le 1} X] \wequi H^1(U , G). \\
\end{CD}
\end{equation*}
Since $X$ is $\A^1$-invariant, the left hand map is surjective (in fact, an isomorphism), and hence the right hand map is surjective.

(2) Reason similarly by applying Lemma \ref{lemm:pi1-obstruction} with with $n=0$ to $\Omega X \in \Shv_\Nis(\Sm_k)$ and $U \times \A^1$.
\end{proof}

\begin{theorem} \label{thm:pi1-strongly-inv-forreal}
Let $k$ be any field and $X \in \Spc(k)_*$.
Then $\ul\pi_1 X$ is strongly $\A^1$-invariant.
\end{theorem}
\begin{proof}
By Lemma \ref{lemm:conn-cpt-of-basepoint}, we may assume $X$ connected.
Set $G = \ul\pi_1 X$.
Let $U \in \Sm_k$ be connected.
We shall show that $H^i(\A^1 \times U, G) \to H^i(U, G)$ is injective for $i \le 1$.
Thus let $U^{(\le 0)} = U^{(0)}$ be the set of generic points of $U$, and $U^{(\le 1)}$ the set of points of dimension $\le 1$.
In other words $U^{(0)}$ is the essentially smooth scheme obtained by removing all closed subschemes of positive codimension, and $U^{(\le 1)}$ is obtained similarly by removing all closed subschemes of codimension $\ge 2$.
Lemma \ref{lemm:pi1-unramified} and Proposition \ref{prop:inj-H1} imply that in the following commutative diagram
\begin{equation*}
\begin{CD}
H^i(\A^1 \times U, G) @>>> H^i(\A^1 \times U^{(\le i)}, G) \\
@VVV @VVV \\
H^i(U, G) @>>> H^i(U^{(\le i)}, G) \\
\end{CD}
\end{equation*}
the top horizontal map is injective.
The right hand horizontal map is injective by Lemma \ref{lemm:A1-inv-dimsmall}.
It follows that the left hand horizontal map is injective, as needed.
\end{proof}

\section{Naive Milnor--Witt $K$-theory} \label{sec:MWn}
\begin{definition}\label{def:naive-KMW}
We define a sheaf of graded (non-commutative) rings $\KMWn_*$ on $\Sm_k$ as generated by a copy of $\Gm$ in degree 1, with elements denoted like $[a]$, and an element $\eta$ in degree -1, subject to the following relations:
\begin{itemize}
\item (\textit{Steinberg relation.}) For $a \in \scr O(X)^\times$ with $1-a \in \scr O(X)^\times$ we require $[a][1-a] = 0$.
\item (\textit{Logarithmic relation.}) For $a,b \in \scr O(X)^\times$ we require $[ab] = [a] + [b] + \eta[a][b]$.
\item (\textit{Centrality relation.}) For $a \in \scr O(X)^\times$ we require $\eta[a] = [a]\eta$.
\item (\textit{Hyperbolic relation.}) $\eta(2 + \eta[-1]) = 0$.
\end{itemize}
We call $\KMWn_*$ the \emph{naive Milnor--Witt $K$-theory sheaf}.

For $a \in \scr O(X)^\times$, the element $1 + \eta[a] \in \KMWn_0(X)$ is denoted by $\lra{a}$.
\end{definition}

\begin{remark} \label{rmk:KMW-naive-vs-not}
There is a sheaf of graded rings $\ul{K}_*^{MW}$ defined in \cite{A1-alg-top}, which is called \emph{unramified Milnor--Witt $K$-theory}.
There is a map $\KMWn_* \to \ul{K}_*^{MW}$ which induces an isomorphism on sections over fields.
It is an isomorphism of sheaves if $k$ is infinite of characteristic $\ne 2$ \cite[Theorem 6.3]{KMW-for-local-rings}.
I believe the assumption that $char(k) \ne 2$ is unnecessary for this, but I do not have a proof at the moment.
Infinitude of $k$ is, however, necessary (for the same reason that it is necessary in ordinary Milnor $K$-theory \cite{kerz-improved-milnor-K}).
In order to forgo the complications on the characteristic and/or infinitude of $k$, in this text we shall only use $\KMWn$.
\end{remark}

\subsubsection*{Outline}
In \S\ref{subsec:KMWn-1}, we outline some easy consequences of the above definition.
Our goal for the remainder of this section is to relate $\KMWn_*$ to motivic homotopy theory.
The key idea is that $\KMWn_n$ for $n \ge 1$ should be close to the free strongly $\A^1$-invariant sheaf of abelian groups under $\Gmp{n}$.
To prove this, we must find analogs of the defining relations for $\KMWn_*$ ``in the wild''.
We begin with this in \S\ref{subsec:hopf} where we show, for example, that the two possible ``Hopf maps'' $\Sigma \Gmp{3} \to \Sigma \Gmp{2}$ (roughly speaking, these are given by $(x,y,z) \mapsto (xy, z)$ and $(x,y,z) \mapsto (x, yz)$, respectively) are homotopic over any base.
Then in \S\ref{subsec:KMWn-universal} we utilize this, together with the proof of the Steinberg relation due to Hu--Kriz--Hoyois, to establish our desired universal property for $\KMWn_*$.

\subsection{First consequences} \label{subsec:KMWn-1}
We summarize some easy consequences of the relations.
\begin{lemma} \label{lemm:KMWn-basics}
\begin{enumerate}
\item We have $[1] = 0$ and $\lra{1} = 1$.
\item We have $\lra{a}\lra{b} = \lra{ab}$.
\item The sheaves $\KMWn_i$ for $i \le 0$ are generated by the elements $\eta^{-i} \lra{a}$.
  The ring $\KMWn_0$ is commutative and central in $\KMWn_*$.
\item The sheaves $\KMWn_i$ for $i>0$ are generated by elements $[a_1] \dots [a_i]$.
\end{enumerate}
\end{lemma}
\begin{proof}
We can rewrite the logarithmic relation as $[ab] = [a] + \lra{a}[b]$ and shall use this several times below.

(2) Multiplying the above expression by $\eta$ yields $\lra{ab}-1=\lra{a}-1 + \lra{a}(\lra{b}-1)$, which is the required relation.

(1) Multiplying the hyperbolic relation by $[1]$ (on the left) yields $(\lra{1}-1)(1+\lra{-1}) = 0$.
Multiplying out the left hand side using (1) yields $\lra{1}-1 = 0$.
Now applying the logarithmic relation $[1] = [1\cdot 1] = [1] + \lra{1}[1]$ yields $[1]=0$.

(4) The sheaves $\KMWn_i$ for $i \ge 0$ are clearly generated by expressions of the form $[a_1] \dots [a_{i+d}] \eta^d$.
Using the logarithmic relation, for $i>0$ we can get rid of the extra factors of $\eta$.

(3) The above argument shows that $\KMWn_0$ is generated by elements of the form $\eta[a] = \lra{a}-\lra{1}$, and the sheaves $\KMWn_i$ for $i<0$ are generated by $\eta^{-i}$-multiples of these.
To show that $\KMWn_0$ is central (in particular, commutative), it suffices to show that $\lra{a}$ commutes with the generators $[b]$ and $\eta$.
For $\eta$ this is true by definition.
For $[b]$, observe that the logarithmic relation implies $\eta[a][b] = \eta[b][a]$ (since $[ab] = [ba]$ and $[a]+[b] = [b]+[a]$).
Using that $\eta$ is central this yields $(\lra{a}-1)[b] = [b](\lra{a}-1)$, from which the result follows.
\end{proof}

\begin{remark} \label{rmk:Gm-to-KMW0}
Lemma \ref{lemm:KMWn-basics}(1,2) imply that there is a morphism of sheaves of abelian groups \[ \Gm \to (\KMWn_0)^\times, a \mapsto \lra{a}. \]
\end{remark}

\begin{remark} \label{rmk:KMW-extra-relations}
The sections of $\ul{K}_*^{MW}$ (the non-naive Milnor--Witt $K$-theory sheaf) satisfy some further interesting relations, in particular:
\begin{itemize}
\item $\lra{a^2} = 1$
\item $[a][b] = -\lra{-1}[b][a]$
\end{itemize}
For a proof\NB{include it, since we use it?}, see \cite[Lemma 3.7(3,4)]{A1-alg-top}.\footnote{This does not immediately generalize from fields, because the proof, e.g., uses that for $1 \ne a \in \scr O^\times$ have $1-a \in \scr O^\times$.}
These relations do hold on sections over fields, though, because there $\KMWn_* = \ul{K}_*^{MW}$.
\end{remark}

Let us note that the sheaves $\KMWn_n$ for $n \ge 1$ admit slightly more explicit presentations.
\begin{lemma} \label{lemm:KMWn-pres}
Let $n \ge 1$.
Consider the sheaf $P_n$ generated by copies of $\Gm^{\times i}$ for $i \ge n$, with generators denoted by $[\eta^i, u_1, \dots, u_{n+i}]$.
Write $R_n$ for the following sheaf of relations:
\begin{itemize}
\item (\textit{Steinberg relation.}) $[\eta^i, u_1, \dots, u_{n+i}] = 0$ if $u_m + u_{m+1} = 1$ for some $m$.
\item (\textit{Logarithmic relation.}) For $a,b \in \scr O(X)^\times$ and each $m,i$ we have \begin{gather*} [\eta^i, u_1, \dots, u_{m-1}, ab, u_{m+1}, \dots, u_{n+i}] = [\eta^i, u_1, \dots, u_{m-1}, a, u_{m+1}, \dots, u_{n+i}] \\ + [\eta^i, u_1, \dots, u_{m-1}, b, u_{m+1}, \dots, u_{n+i}] + [\eta^{i+1}, u_1, \dots, u_{m-1}, a, b, u_{m+1}, \dots, u_{n+i}]. \end{gather*}
\item (\textit{Hyperbolic relation.}) For each $m, i$ we have \[ [\eta^{i+2}, u_1, \dots, u_{m-1}, -1, u_{i+1}, \dots, u_{n+i+1}] + 2[\eta^{i+1}, u_1, \dots, u_{m-1}, u_{i+1}, \dots, u_{n+i+1}] = 0. \]
\end{itemize}
Then $P_n/R_n \wequi \KMWn_n$.
\end{lemma}
\begin{proof}
This just lists generators in degree $n$ of the two-sided ideal generated by the defining relations of Definition \ref{def:naive-KMW}, except for the centrality of $\eta$, which is enforced by pulling it to the front.
\end{proof}

\subsection{Some Hopf maps} \label{subsec:hopf}
Before proceeding, we need to recall various maps between motivic spheres.
We can work in $\Spc(\Z)$, or, equivalently, over any base.

\begin{remark}
Recall the canonical equivalences $\SL_2 \wequi \A^2 \setminus 0 \wequi \Sigma \Gmp{2} \in \Spc(S)_*$, which we shall use constantly in the sequel.
(For the first one, check that projection to the first row $\SL_2 \to \A^2 \setminus 0$ is a (Zariski) locally trivial bundle with fiber $\A^1$; for the second one, see the proof of Lemma \ref{lemm:A1-homotopies}.)
\end{remark}

\begin{lemma} \label{lemm:SL2-globally-conn}
The space $\Map_{\Spc(\Z)}(*, \SL_2)$ is connected.
\end{lemma}
\begin{proof}
This is true because $(L_\mot \SL_2)(\Z) \wequi (\Sing \SL_2)(\Z)$ \cite[Theorems 2.4.2 and 3.3.1]{asok2015affine} and $\SL_2(\Z)$ is generated by elementary matrices\todo{ref?}.
\end{proof}
\begin{remark}
It is automatic that the \emph{sheaf} $\ul\pi_0 L_\mot \SL_2$ is connected (Lemma \ref{lemm:0-conn} works over any base), since $\SL_2 \wequi \Sigma \Gmp{2}$.
Consequently the reference to \cite{asok2015affine} is not necessary if $S$ is henselian local, e.g., a field.
\end{remark}

\begin{corollary} \label{cor:SL2-pted-unpted}
Let $X \in \Spc(\Z)_*$.
The canonical map $\pi_0 \Map_{\Spc(\Z)_*}(X, \SL_2) \to \pi_0 \Map_{\Spc(\Z)}(X, \SL_2)$ is a bijection.
\end{corollary}
\begin{proof}
Consider the fiber sequence $F= \Map_{\Spc(\Z)_*}(X, \SL_2) \to E=\Map_{\Spc(\Z)}(X, \SL_2) \to B=\Map_{\Spc(\Z)}(*, \SL_2)$.
Considering constant maps supplies us with a section of $E \to B$.
Noting that the middle space is a group (since $\SL_2$ is) we can thus form a map $\alpha: F \times B \to E$ over $B$.
Since $B$ is connected (Lemma \ref{lemm:SL2-globally-conn}) and the base change of $\alpha$ along $* \to B$ is an equivalence, $\alpha$ is an equivalence \cite[Lemma 6.2.3.16]{lurie-htt}.
We deduce that $\pi_0 F \wequi \pi_0 E$, as needed.
\end{proof}

\begin{lemma} \label{lemm:A1-homotopies}
Let $S$ be a scheme.

The following maps $\Gm \wedge \Gm \to \Gm \wedge \Gm \in \Spc(S)_*$ become homotopic after one suspension: the switch map, the map $\id \wedge i$ and the map $i \wedge \id$.
In fact they become homotopic to the opposites of the suspensions of the map $\id \wedge i'$ or the map $i' \wedge \id$.

Here $i: \Gm \to \Gm$ is the map $x \mapsto x^{-1}$, the map $i': \Gm \to \Gm$ is $x \mapsto -x$ (which is not pointed, but its unreduced suspension is), and the opposite refers to taking the inverse in the group structure (coming from mapping out of a suspension, or equivalently the group structure on $\Sigma \Gm \wedge \Gm \wequi \SL_2$).
\end{lemma}
\begin{proof}
We may assume $S=\Spec(\Z)$.
By Corollary \ref{cor:SL2-pted-unpted}, we need only consider homotopy classes of unpointed maps.
Consider the pushout squares
\begin{equation*}
\begin{CD}
\Gm \times \Gm @>>> \Gm \times \A^1  @. \quad\text{and}\quad @. \Gm @>>> \A^1 \\
@VVV                    @VVV                 @.                  @VVV    @VVV \\
\A^1 \times \Gm @>>> \A^2 \setminus 0 @.                      @. \A^1 @>>> \P^1. \\
\end{CD}
\end{equation*}
These show, respectively, that $\Sigma \Gm \wedge \Gm \wequi \A^2 \setminus 0$ and $\Sigma \Gm \wequi \P^1$.
The automorphism $(x, y) \mapsto (y,x)$ induces automorphism of both squares, interchanging the top right and bottom left corner, and inducing on the top left corner the switch map respectively $i$.
It follows that the switch map on $\Sigma \Gm \wedge \Gm$ is homotopic to the opposite in the group structure of the switch automorphism on $\A^2 \setminus 0$, and similarly $i$ is homotopic to the opposite of the ``switch'' automorphism on $\P^1$.
Analogously we have the automorphism $(x,y) \mapsto (-x,y)$ which preserves the corners of the squares, and hence the evident analogs of $i'$ on $\Sigma \Gm \wedge \Gm$ and $\A^2 \setminus 0$ (respectively $\Sigma \Gm$ and $\P^1$) correspond.

Recall that we have an action of $\GL_2(\Z)$ on $\A^2 \setminus 0$ (respectively $\P^1$) under which our automorphism corresponds to the matrix $\begin{pmatrix} 0 & 1 \\ 1 & 0 \end{pmatrix}$.
Since any two matrices in $\GL_2$ with the same determinant differ by an element of $\SL_2$, and $\SL_2(\Z)$ is generated by elementary matrices, it follows that the switch automorphisms on $\A^2 \setminus 0$ and $\P^1$ are also homotopic to the map induced by the matrix $\begin{pmatrix} -1 & 0 \\ 0 & 0 \end{pmatrix}$ or $\begin{pmatrix} 1 & 0 \\ 0 & -1 \end{pmatrix}$, corresponding to $i'$.

Putting these observations together, the result follows.\NB{details?}
\end{proof}

Let $X, Y \in \Spc(S)_*$.
We have the three canonical (pointed) maps $X \times Y \to X$, $X \times Y \to Y$ and $X \times Y \to X \wedge Y$.
Recall that suspensions are cogroups, and hence maps out of suspensions can be added.
We may thus add the suspensions of the three previous maps in order to obtain \[ \Sigma(X \times Y) \to \Sigma X \vee \Sigma Y \vee \Sigma (X \wedge Y). \]
This map is in fact an equivalence, see, e.g., \cite[Corollary 2.24]{devalapurkar2019james}.
The extension to larger finite products is immediate.

Now consider the multiplication map $\mu: \Gm \times \Gm \to \Gm, (x, y) \mapsto xy$.
This is a pointed map, and hence after suspending it splits into a sum of three maps: two maps $\Sigma \Gm \to \Sigma \Gm$ and one map $\Sigma \Gmp{2} \to \Sigma \Gm$.
It is easily verified that the two former maps are homotopic to the identity.
The latter map is denoted $\eta_{12}$.

\begin{definition}
More generally, let $I = I_1 \amalg \dots \amalg I_n$ be a partition of a finite set.
We define a map \[ \eta_{I_1, \dots, I_n}: \Sigma \Gmp{I} \to \Sigma \Gmp{n} \] by analogously passing to a summand of the suspension of \[ \mu_{I_1, \dots, I_n}: \Gm^{\times I} \wequi \prod_{i=1}^n \Gm^{\times I_i} \to \Gm^{\times n}, \] where each of the maps $\Gm^{\times I_i} \to \Gm$ is multiplication.
\end{definition}

We deduce the following key results.
\begin{proposition} \label{prop:eta-central}
Let \[ \{1, 2, \dots, n\} = I_1 \amalg \dots \amalg I_m = I_1' \amalg \dots \amalg I_m' \] be two ordered (i.e. $I_i < I_{i+1}$, $I_i' < I_{i+1}'$) partitions into nonempty sets.
Then the two maps \[ \eta_{I_1, \dots, I_m}, \eta_{I'_1, \dots, I'_m}: \Sigma \Gmp{n} \to \Sigma \Gmp{m} \in \Spc(\Z)_* \] are homotopic.
\end{proposition}
\begin{proof}
It will suffice to consider the two maps $\eta_{12}, \eta_{23}: \Sigma \Gmp{3} \to \Sigma \Gmp{2}$.\NB{details?}
We express $\mu_{23}: \Gm^{\times 3} \to \Gm^{\times 2}$ in terms of switch maps $\tau_{ab}$ and $\mu_{12}$: \[ \mu_{23}: (x,y,z) \xmapsto{\tau_{12}} (y,x,z) \xmapsto{\tau_{23}} (y, z, x) \xmapsto{\mu_{12}} (yz,x) \xmapsto{\tau_{12}} (x, yz). \]
On the other hand, consider the following composite, where $i_n$ denotes the map inverting the $n$-th factor: \[ \mu_{12}: (x,y,z) \xmapsto{i_1}(x^{-1},y,z) \xmapsto{i_2} (x^{-1}, y^{-1}, z) \xmapsto{\mu_{12}} ((xy)^{-1}, z) \xmapsto{i_1} (xy,z). \]
Suspending and passing to the appropriate summands, the composites become homotopic by Lemma \ref{lemm:A1-homotopies}.
This yields the desired homotopy $\Sigma \eta_{12} \wequi \Sigma \eta_{23}$.
\end{proof}

\begin{proposition} \label{prop:hyperbolic}
The map \[ \eta \circ (\id + \id + \eta \circ [-1]): \Sigma \Gmp{2} \to \Sigma \Gm \in \Spc(\Z)_* \] is null, where $+$ means addition in the group structure\footnote{Beware that the group is not abelian.} and $[-1]: \Gmp{2} \to \Gmp{3}$ is the map $(x,y) \mapsto (x,y,-1)$.
\end{proposition}
\begin{proof}
The map $i': \Gm \to \Gm$ (i.e. $x \mapsto -x$) can be written as the composite $\Gm \xrightarrow{(\id, -1)} \Gm \times \Gm \xrightarrow{\mu} \Gm$.
Splitting the unreduced suspension\NB{basepoint shenanigans}, we see that $\Sigma i' \wequi \id + \eta \circ [-1]$.
Smashing with the identity map on $\Gm$, using Lemma \ref{lemm:A1-homotopies} and adding another identity map, we see that \[ \id - \text{(twist)} \sim \id + \id + \eta \circ[-1]: \Sigma \Gmp{2} \to \Sigma \Gmp{2}. \]
Since $\eta$ comes from multiplication and $\Gm$ is commutative, the left hand map becomes null after composing with $\eta$ once more.
This being true for the right hand map establishes the result.
\end{proof}

\subsection{Universal property} \label{subsec:KMWn-universal}
Now we can prove the following.
Denote by $\Z[\Gmp{n}]$ the sheaf of abelian groups freely generated by the pointed sheaf of sets $\Gmp{n}$.
\begin{theorem} \label{thm:KMWn-univ}
Let $k$ be any field.
Let $F$ be a strongly $\A^1$-invariant sheaf of abelian groups and $n \ge 1$.
Any map $\Z[\Gmp{n}] \to F$ factors uniquely through $\KMWn_n$.
\end{theorem}
\begin{proof}
Since $\Z[\Gmp{n}] \to \KMWn_n$ is surjective (Lemma \ref{lemm:KMWn-basics}(4)), the factorization is unique if it exists.
Our map $\Z[\Gmp{n}] \to F$ corresponds to a map $\alpha: \Sigma \Gmp{n} \to BF \in \Spc(k)_*$.
By applying the Hopf construction we obtain various maps $\eta_\bullet: \Sigma \Gmp{n+i} \to \Sigma \Gmp{n}$, which by Proposition \ref{prop:eta-central} are all $\A^1$-homotopic (for fixed $i$).
We denote them by $\eta^i$.
Now consider the composite \[ \bigvee_{i \ge 0} \Sigma \Gmp{n+i} \xrightarrow{\vee \eta^i} \Sigma \Gmp{n} \xrightarrow{\alpha} BF. \]
Applying $\ul\pi_1$, this yields (in the notation of Lemma \ref{lemm:KMWn-pres}) a map $P_n \to F$, which we must prove annihilates $R_n$.

Let us begin with the Steinberg relation.
Using Proposition \ref{prop:eta-central}, it suffices to prove that for any $X \in \Sm_k$ and $a \in \scr O^\times(X)$ such that $1-a \in \scr O^\times(X)$, the map $X_+ \xrightarrow{(a,1-a)} \Gmp{2} \in \Spc(k)_*$ becomes nullhomotopic after one suspension.
By adjunction, for this it is enough to show that the composite \[ S^1_X = \Sigma X_+ \xrightarrow{\Sigma a)} \Sigma(\A^1_X \setminus \{0,1\})_+ \xrightarrow{\Sigma \mathfrak{st}} \Sigma(\Gmp{2})_X \in \Spc(X)_* \] is null.
For this, see \cite{hoyois-steinberg}.

Now we establish the logarithmic relation.
Again using Proposition \ref{prop:eta-central}, for this it suffices to show that the multiplication map $\Gm \times \Gm \to \Gm$ splits after one suspension as the sum of the two projection maps and $\eta$.
This holds by construction.

The hyperbolic relation is established similarly, using Proposition \ref{prop:hyperbolic}.
\end{proof}

\begin{remark} \label{rmk:multn-by-a-map}
For $a \in \scr O^\times$ we have the map $\Gm \to \Gm, x \mapsto ax$.
Its unpointed suspension is homotopic to $1+\eta[a]$: use the same argument as at the end of the proof of Theorem \ref{thm:KMWn-univ} (which treats $a=-1$).
\end{remark}

\begin{remark} \label{rmk:KMW-univ}
Abstract nonsense considerations show that there exists an initial strongly $\A^1$-invariant sheaf of abelian groups $F_n$ under $\Z[\Gmp{n}]$.
Theorem \ref{thm:KMWn-univ} is equivalent to stating that $\Z[\Gmp{n}] \to F_n$ factors through $\KMWn_n$.
It is proved in \cite[Theorem 3.37]{A1-alg-top} that in fact $F_n \wequi \ul{K}_n^{MW}$.
\end{remark}

Combining the above work with the main theorems of these notes, we can deduce the following result.
The assumptions on the cardinality, characteristic and perfectness of $k$ are not actually necessary, and only an artefact of our simplified exposition.
\begin{corollary+}
Let $k$ be an infinite perfect field of characteristic $\ne 2$, $n \ge 1$ and $m \ge 2$.
Then $\ul\pi_m \Sigma^m \Gmp{n} \wequi \ul{K}_n^{MW}$.
\end{corollary+}
\begin{proof}
It follows easily from the main theorems\NB{details?} that $\Sigma^m \Gmp{n}$ is $m$-connective and $\ul\pi_m \Sigma^m \Gmp{n}$ is the free strictly $\A^1$-invariant sheaf of abelian groups on $\Gmp{n}$.
Combining Theorem \ref{thm:KMWn-univ} and the main theorem, this is the same as the initial strictly $\A^1$-invariant sheaf under $\KMWn_n$.
By Remark \ref{rmk:KMW-naive-vs-not}, $\KMWn_n \wequi \ul{K}_n^{MW}$, so it will suffice to show that the latter is strictly $\A^1$-invariant.
For this, see \cite[\S2 and \S9]{feld-cycle-modules}.
\end{proof}

\section{Contractions} \label{sec:contractions}
Recall that we set $\Gm = (\A^1 \setminus 0, 1) \in \Spc(k)_*$.
\begin{definition}
For $X \in \Spc(k)_*$, we denote by $\Omega_\Gm X := \iMap(\Gm,X)$ the $\Gm$-loop space.

In the case where $X=F$ is a sheaf of pointed sets $\Sm_k$, we also write $F_{-1} := \Omega_\Gm F$ and call this the \emph{contraction of $F$}.
It is a pointed sheaf of sets.

We define the \emph{iterated contractions} by $F_{-n-1} = (F_{-n})_{-1}$.
\end{definition}
In other words, we have \[ F_{-1}(X) = \{ s \in F(X \times \Gm) \mid i_X^*(s) = *\}, \] where $i_X: X \to X \times \Gm$ is the inclusion at the base point.

\subsubsection*{Outline} 
The aim of this section is to study this construction.
We begin in \S\ref{subsec:contr-1} with some straightforward observations.
For example, we use the results of the previous section to show that there is an action $\KMWn_n \otimes F_{-n} \to F$.

In \S\ref{subsec:contr-local-coho}, we relate contractions to local cohomology.
If $F$ is \emph{strictly} $\A^1$-invariant, $X \in \Sm_k$ and $x \in X$, it is easy to see using purity that the local cohomology groups $H^*_x(X,F)$ can all be expressed using contractions of $F$.
If $F$ is only strongly invariant, this a priori only remains true for $* < 2$.
One crucial result, the proof of which we only start in this section (and finish in \S\ref{sec:gersten}) is that the identification still works for $*=2$.

In \S\ref{subsec:contractions-transfers} we construct certain transfers on sheaves of the form $F_{-1}$.
These are a priori not even well-defined.
We begin the proof that on sheaves of the form $F_{-2}$, the transfers are indeed well-defined (again the proof is finished in \S\ref{sec:gersten}).

\subsection{First properties} \label{subsec:contr-1}
The following are immediately verified:
\begin{itemize}
\item If $G$ is a sheaf of (abelian) groups, then $G_{-1}$ also is.
\item If $F$ is $\A^1$-invariant, then so is $F_{-1}$.
\item If $A$ is a sheaf of abelian groups, then $A_{-1}$ is a summand of $A(\Gm \times (\ph))$.
  (Namely we always have a natural retraction $A(X) \xrightarrow{p^*} A(X \times \Gm) \xrightarrow{i^*} A(X)$, and $A_{-1}(X)$ is the complementary summand.)
\end{itemize}

Contractions interact well with strong $\A^1$-invariance:
\begin{lemma} \label{lemm:contract-strong}
Let $F$ be a Nisnevich sheaf of groups and $n \ge 1$ such that $K(F,n)$ is $\A^1$-invariant (assume $F$ abelian if $n \ge 2$).

There is a canonical equivalence $\Omega_\Gm K(F,n) \wequi K(F_{-1},n)$.
In particular, if $F$ is strongly (or strictly) $\A^1$-invariant, so is $F_{-1}$.
\end{lemma}
\begin{proof}
Since $K(F,n)$ is $n$-truncated, so is $\Omega_\Gm K(F,n)$.
Now \[ \ul\pi_n \Omega_\Gm K(F,n) \wequi \Omega_{S^n} \Omega_\Gm K(F,n) \wequi \Omega_\Gm \Omega_{S^n} K(F,n) \wequi \Omega_\Gm F = F_{-1}. \]
It remains to prove that the other homotopy sheaves of $\Omega_\Gm K(F,n)$ vanish.
Using induction on $n$, it is enough to show that $\Omega_\Gm K(F,n)$ is connected.
By Corollaries \ref{cor:connectivity-crit} and \ref{cor:Zariski-vanishing}, for this it suffices to show: if $K/k$ is a separable field extension, then $H^n_\Zar(\Gm_K, F) = 0$.
If $n \ge 2$ this is clear, because $\Gm_K$ has dimension $1$.
We must thus treat the case $n=1$.

As a preparatory remark, let $X$ be a quasi-compact scheme and $F$ a Zariski sheaf of groups on $X$.
Let $U_1, \dots, U_n$ an open cover, for $i < j$ let $g_{ij} \in F(U_i \cap U_j)$, and for $i<j<k$ assume that $g_{ij} g_{jk} = g_{ik} \in F(U_i \cap U_j \cap U_k)$.
Then the $\{g_{ij}\}$ define a cocycle for $F$, whence an element in $H^1_\Zar(X,F)$, and all elements of that group arise in this way (for appropriately fine covers)\footnote{Indeed to evaluate $H^1_\Zar(X, F) = \pi_0 ((L_\Zar BF)(X))$, we may take the filtered colimit over all hypercovers $X_\bullet \to X$ of $\pi_0 BF(|X_\bullet|) \wequi \check{H}^1(X_\bullet, F)$ \cite[Theorem 8.6(b)]{dugger2004hypercovers}.}

We now apply this to $\Gm_K$.
Thus let $U_1, \dots, U_n$ be an open cover of $\Gm_K$, $\{g_{ij} \in F(U_i \cap U_j)\}_{i<j}$ a cocycle.
We may assume that $U_1 \ne \emptyset$.
Set $\tilde U_1 = U_1 \cup \{0\} \subset \A^1_K$, $\tilde U_i = U_i$ for $i > 1$.
Then the $\tilde U_i$ form an open cover of $\A^1_K$.
For $i<j$ (respectively $i<j<k$) we have $\tilde U_i \cap \tilde U_j = U_i \cap U_j$ (respectively $\tilde U_i \cap \tilde U_j \cap \tilde U_k = U_i \cap U_j \cap U_k$).
It follows that the $\{g_{ij}\}_{i<j}$ also define a cocycle for $F$ on $\A^1_K$, with respect to the cover $\{\tilde U_i\}_i$.
This means that the class of $H^1_\Zar(\Gm_K, F)$ corresponding to the $g_{ij}$ lies in the image of the map $H^1_\Zar(\A^1_K, F) \to H^1_\Zar(\Gm_K, F)$.
Since the cocycle was arbitrary, the map is thus surjective.
One more application of Corollary \ref{cor:Zariski-vanishing} shows that $H^1_\Zar(\A^1_K, F) = H^1_\Nis(\A^1_K, F)$, which is zero by assumption.
This concludes the proof.
\end{proof}

The following is one of the main reasons why contractions are relevant in motivic homotopy theory.
\begin{corollary+}
Let $k$ be perfect, $X \in \Spc(k)_*$.
For $i \ge 1$ we have $\ul\pi_i \Omega_\Gm X \wequi \ul\pi_i(X)_{-1}$, and if $X$ is connected then so is $\Omega_\Gm X$.
\end{corollary+}
\begin{proof}
For $T \in \Sm_k$ and $i \ge 1$ the space $S^i \wedge \Gm \wedge T_+$ is connected, and hence any map $S^i \wedge \Gm \wedge T_+ \to X$ factors (uniquely) through $X_0$, the connected component of the base point.
In other words $\ul\pi_i \Omega_\Gm X \wequi \ul\pi_i \Omega_\Gm X_0$, and so for the first statement we may assume that $X$ is connected.
By considering the Postnikov tower, both statements follow from Lemma \ref{lemm:contract-strong}.
(In more detail: we have fiber sequences $K(\ul\pi_{n+1} X, n+1) \to X_{\le n+1} \to X_{\le n}$.
Applying $\Omega_\Gm$ we get fiber sequences $K(\ul\pi_{n+1}(X)_{-1}, n+1) \to \Omega_\Gm(X_{\le n+1}) \to \Omega_\Gm(X_{\le n})$.
Inductively these show that $\ul\pi_i \Omega_\Gm(X_{\le n}) \wequi \ul\pi_i(X)_{-1}$, for $i \le n$.
The result follows since $\Omega_\Gm X \wequi \lim_n \Omega_\Gm(X_{\le n})$.)
\end{proof}

\NB{mention exactness of contraction?}

\begin{construction}
The canonical maps $\Z[\Gmp{n}] \otimes \KMWn_m \to \KMWn_{n+m}$ induce by adjunction canonical maps \[ \KMWn_m \to (\KMWn_{m+n})_{-n}. \]
\end{construction}

\begin{remark}
One similarly has $\ul{K}_m^{MW} \to (\ul{K}_{m+n}^{MW})_{-n}$.
This map can be shown to be an isomorphism \cite[Lemma 2.48]{A1-alg-top}.
This is one of the justifications for the slightly curious notation for contractions.
\end{remark}

Let $F$ be a strongly $\A^1$-invariant abelian sheaf.
Since $F_{-n} \wequi \iHom(\Z\{\Gmp{n}\}, F)$, we get by adjunction an action $\Z[\Gmp{n}] \otimes F_{-n} \to F$.
\begin{proposition} \label{prop:KMW-action}
Let $n \ge 1$ and $k$ be any field.
The above action factors through the surjection $\Z[\Gmp{n}] \to \KMWn_n$, inducing an action \[ \KMWn_n \otimes F_{-n} \to F. \]
\end{proposition}
\begin{proof}
The map is surjective by Lemma \ref{lemm:KMWn-basics}(4), say with kernel $K_n$.
We must show that the composite $K_n \otimes F_{-n} \to \Z[\Gmp{n}] \otimes F_{-n} \to F$ is null.
Let $X \in \Sm_k$ be connected with generic point $\eta$, $a \in K_n(X)$ and $b \in F_{-n}(X)$.
We have to show that $0 = a \cdot b \in F(X)$.
Using Remark \ref{rmk:unramifiedness-use}, we may prove this for the image in $F(\eta)$.
By essentially smooth base change (Lemma \ref{lemm:essentially-smooth-bc}), we may thus assume that $X=\Spec(k)$.
We have the map \[ \Z[\Gmp{n}] \to F, x \mapsto x \cdot b. \]
Since $F$ is strongly $\A^1$-invariant and abelian, this map factors through $\KMWn_n$, by Theorem \ref{thm:KMWn-univ}.
The vanishing of $a \cdot b$ follows.
\end{proof}

Now let $0<k \le n$.
\begin{construction}
Consider the composite \[ \KMWn_{n-k} \otimes F_{-n} \xrightarrow{u} (\Z[\Gmp{k}] \otimes \KMWn_{n-k} \otimes F_{-n})_{-k} \to (\KMWn_n \otimes F_{-n})_{-k} \to F_{-k}, \] where $u$ is a unit of adjunction, the second map comes from multiplication in $\KMWn_*$, and the third is the action from Proposition \ref{prop:KMW-action}.
In the special case $k=n$, we in particular obtain the following important action: \[ \KMWn_0 \otimes F_{-n} \to F_{-n}. \]
In other words, for $n \ge 1$, the sheaf $F_{-n}$ is automatically a module over the sheaf of rings $\KMWn_0$.
\end{construction}

\begin{remark} \label{rmk:action-of-Otimes}
Tracing through the definitions\NB{details?}, we find that the element $\lra{a}-1 = [a] \eta$ acts on $F_{-1} = [\Sigma \Gm, BF]$ by precomposition with the endomorphism $\Sigma \Gm \xrightarrow{[a]} \Sigma \Gmp{2} \xrightarrow{\eta} \Gmp{2}$.
Combining with Remark \ref{rmk:multn-by-a-map}, we see that $\lra{a}$ acts by precomposition with the unreduced suspension of the multiplication-by-$a$ map.
\end{remark}

We can use this to define \emph{twisted versions} of contracted sheaves.
This works as follows.
\begin{definition}
Let $F$ be strongly $\A^1$-invariant, $n \ge 1$, $X \in \Sm_k$ and $L$ a line bundle on $X$.
We define \[ F_{-n}(X, L) = F_{-n}(X) \times_{\scr O^\times} L^\times. \]
\end{definition}
Here $\scr O^\times$ acts on $F_{-n}$ via the map $\scr O^\times \to \ul{K}_0^{MW}$ from Remark \ref{rmk:Gm-to-KMW0}, and $L^\times$ denotes the set of non-vanishing sections of $L$ (on which $\scr O^\times$ acts in the evident way).
This construction makes sense also for $X$ essentially smooth, and we will in fact mainly use it when $X$ is the spectrum of a field.
Otherwise it may be more reasonable to sheafify the construction.

\begin{remark}
We claim that the action of $\scr O^\times$ on $F_{-n}$ factors through $\scr O^\times/2$, that is, squares act trivially.
Using Remark \ref{rmk:unramifiedness-use}, this reduces to the case of sections over a field.
But in this case, by Remark \ref{rmk:KMW-extra-relations}, the map $\scr O^\times(K) \to \KMWn_0(K)$ sends squares to $1$.
The claim follows.

Using this we can, for every line bundle $E$, construct a canonical isomorphism \[ F_{-n}(X, L) \wequi F_{-n}(X, L \otimes E^{\otimes 2}). \]
In particular (taking $E=L^*$ the dual bundle) \[ F_{-n}(X, L) \wequi F_{-n}(X, L^*). \]
\end{remark}

One use of twisted sheaves is via the following construction.
\begin{construction} \label{constr:twisting-V}
Let $n \ge 1$, $K/k$ a separable field extension, and $V$ an $n$-dimensional $K$-vector space.
Assume that $K(F, n)$ is $\A^1$-invariant.
We attempt to define a map \[ \alpha_V: [V/V \setminus 0, K(F, n)] \to F_{-n}(K, \det V) \] as follows.
Pick a basis $e_1, \dots, e_n$ of $V$.
This induces $V/V \setminus 0 \wequi \A^n_K/\A^n_K \setminus 0 \wequi \Sigma^n \Gmp{n} \wedge \Spec(K)_+$ and hence $[V/V \setminus 0, K(F, n)] \wequi F_{-n}(K)$.
Combine this with the isomorphism $F_{-n}(K) \wequi F_{-n}(K, \det V)$ sending $a$ to $a \otimes (e_1 \wedge \dots \wedge e_n)$.
\end{construction}
\begin{lemma}\label{lemm:twisting-V}
Assumptions as in Construction \ref{constr:twisting-V}.
The map $\alpha_V$ is a well-defined isomorphism.
\end{lemma}
\begin{proof}
It is clear that the map is an isomorphism, we need only check that it is well-defined.
Let us write $\alpha_V^e$ for the isomorphism corresponding to the choice of basis $e_1, \dots, e_n$.
Let $M: \A^n_K \to \A^n_K$ be a $K$-linear automorphism, corresponding to a change of basis $e_1, \dots, e_n \mapsto Me_1, \dots, Me_n$.
Then we have $\alpha_V^{Me}(a) = \det{M} \cdot \alpha_V^e((M^{-1})^* a)$, where $(M^{-1})^*$ acts on $F_{-n}$ via the action of $M^{-1}$ on $\A^n/\A^n \setminus 0$.
It thus suffices to show that the map \[ M^*: F_{-n}(K) \wequi [\A^n_K/\A^n_K\setminus 0, K(F,n)] \to [\A^n_K/\A^n_K\setminus 0, K(F,n)] \wequi F_{-n}(K) \] is multiplication by $\det M$.
Recall that $\mathrm{GL}_n(K)$ is generated by matrices of the form $diag(a,1,\dots,1)$ and elementary matrices\NB{ref?}, the latter of which have determinant $1$ and are $\A^1$-homotopic to the identity.
It thus suffices to consider the case where $M$ is of the former form, and in fact we may assume $n=1$.
This case was treated in Remark \ref{rmk:action-of-Otimes}.
\end{proof}

\subsection{Contractions and local cohomology} \label{subsec:contr-local-coho}
Let $X \in \Sm_k$ and $Z \subset X$ closed.
Let $F$ be a Nisnevich sheaf of abelian groups on $\Sm_k$.
We have the \emph{cohomology with support} \[ H^i_Z(X,F) = [X/X \setminus Z, K(F, i)]_{\Shv_\Nis(\Sm_k)}. \]
Now let $x \in X$.
Write $\overline{\{x\}}$ for the closure of $x$.
The \emph{local cohomology groups} are defined as \[ H^i_x(X, F) = \colim_{x \in U \subset X} H^i_{\overline{\{x\}} \cap U}(U, F), \] the (filtering) colimit being over all open neighborhoods $U$ of $x$.
By generic smoothness \cite[Tag 0B8X]{stacks-project}, assuming $k$ perfect, once $U$ is small enough, $Z' := U \cap \overline{\{x\}}$ is smooth.
Thus by purity (see \S\ref{subsub:purity}), $U/U \setminus Z' \stackrel{\A^1}{\wequi} \Th(N_{Z/U})$.
Shrinking $U$ further, we can assume that $N_{Z/U}$ is trivial of rank equal to the codimension $c$ of $x$, so that $U/U \setminus Z' \wequi Z'_+ \wedge \Sigma^c \Gmp{c}$.
If $K(F,i)$ is $\A^1$-invariant, then we deduce that \[ H^i_{Z'}(U, F) \wequi \tilde H^{i-c}(Z'_+ \wedge \Gmp{c}, F). \]
If $i=c$, then in the colimit we obtain $F_{-c}(x)$.
Unfortunately this computation is not quite canonical: we needed to use a trivialization of $N_{x/X}$.
Via Construction \ref{constr:twisting-V} and Lemma \ref{lemm:twisting-V}, this is precisely what twisting is designed to avoid.
We have thus proved the following.
(See \S\ref{subsub:cotangent-cx} and Example \ref{ex:omega-lci} for the meaning of $\omega_{x/X}$.)
\begin{lemma} \label{lemm:identify-local-cohomology}
Work over a perfect field $k$.

Let $F$ be a Nisnevich sheaf of abelian groups such that $K(F, c)$ is $\A^1$-invariant.
Let $X \in \Sm_k$ and $x \in X$ be a point of codimension $c$.
Then there is a canonical isomorphism \[ H^c_x(X, F) \stackrel{\alpha_x}{\wequi} F_{-c}(x, \omega_{x/X}). \]
If $y \in X$ has codimension $>c$, then $H^c_y(X, F) = 0$.
\end{lemma}
\begin{proof}
Only the last point has not been explained.
But we have seen that $H^c_y(X, F) \wequi H^{c-d}(y_+ \wedge \Gmp{d}, F)$, where $d > c$ is the codimension of $y$.
This group vanishes since $c-d<0$.
\end{proof}

From now on let $F$ be a strongly $\A^1$-invariant sheaf of abelian groups and $k$ a perfect field.
Thus we may apply Lemma \ref{lemm:identify-local-cohomology} with $c \in \{0,1\}$.
The first step in our quest to ultimately show that $F$ is strictly $\A^1$-invariant will be to prove that Lemma \ref{lemm:identify-local-cohomology} also holds for $F$ when $c=2$.
We begin as follows.
Let $X \in \Sm_k$, $x \in X$ of codimension $2$, $y \in X$ of codimension $1$ specializing to $x$.
Assume that $x$ is a smooth point of $\overline{\{y\}} := Y$.
Consider the two cofiber sequences (recall our conventions on essentially smooth schemes from \S\ref{subsub:ess-smooth}) \[ y \to Y_x  \to Y_x/y \text{ and } X_x \setminus x/X_x \setminus Y_x \to X_x/X_x \setminus Y_x \to X_x/X_x \setminus x. \]
A choice $\lambda$ of trivialization of $N_{Y_x/X_x}$ induces an $\A^1$-equivalence between the $\P^1 \wedge$ the first sequence and the second sequence.
Denote by $\mu$ a choice of trivialization of $N_{x/Y_x}$.
Consider the following part of the long exact sequence coming from mapping the second cofiber sequence into $F$ \[ H^1(X_x/X_x \setminus Y_x, F) \to H^1(X_x \setminus x/X_x \setminus Y_x, F) \xrightarrow{\partial} H^2_x(X, F). \]
Consider also the following part of the long exact sequence coming from mapping the first cofiber sequence into $F_{-1}$ \[ F_{-1}(Y_x) \to F_{-1}(y) \xrightarrow{\partial} H^1_x(Y, F_{-1}) \to H^1(Y_x, F). \]
Using $\mu$, the second to last term of the second sequence is identified with $F_{-2}(x)$:  \[ [Y_x/Y_x \setminus x, K(F_{-1}, 1)] \wequi [\Th(N_{x/Y}), K(F_{-1},1)] \stackrel{\mu}{\wequi} [\Sigma\Gm \wedge x_+, K(F_{-1},1)] \wequi [\Gm \wedge x_+, F_{-1}] = F_{-2}(x). \]
(Here we use that $F_{-1}$ is strongly $\A^1$-invariant, by Lemma \ref{lemm:contract-strong}.)
Similarly using $\lambda$, the first two terms in the respective long exact sequences are identified.
The last term of the second sequence vanishes, by Corollary \ref{cor:Zariski-vanishing}.
\begin{construction} \label{constr:local-H2-comparison}
Given $X, x, y, \mu, \lambda$ as above, and using the identifications described, we obtain a map \[ \alpha_{x,y}^{\mu,\lambda}: F_{-2}(x) \xleftarrow{\wequi} F_{-1}(y)/F_{-1}(Y_x) \wequi H^1(X_x \setminus x/X_x \setminus Y_x, F)/H^1(X_x/X_x \setminus Y_x, F) \xrightarrow{\partial} H^2_x(X, F). \]
\end{construction}

We shall prove that this map is an isomorphism.

\begin{remark} \label{rmk:alpha-etale-extension}
The map is compatible with étale extensions of $X$: if $X' \to X$ is étale, $x' \in X'$ a point above $x$ and $y' \in X'$ a point above $y$, specializing into $x$, then we get a commutative diagram of the form
\begin{equation*}
\begin{CD}
F_{-2}(x') @>{\alpha_{x',y'}^{\mu',\lambda'}}>> H^2_{x'}(X',F) \\
@AAA @AAA \\
F_{-2}(x) @>{\alpha_{x,y}^{\mu,\lambda}}>> H^2_x(X,F)
\end{CD}
\end{equation*}
Here $\mu', \lambda'$ are the canonical induced trivializations.\NB{details?}
\end{remark}

We begin with the following lemma, which is somewhat similar to, but much easier to prove than, Gabber's lemma.
\begin{lemma} \label{lemm:contr1}
Let $X$ be smooth of dimension $2$ over a field $k$, $C \subset X$ a smooth closed curve, $z \in C$ a rational point, and $Z \subset X$ a closed subscheme of positive codimension, not containing any component of $C$.
There exists a henselian local scheme $S$ of dimension $1$, essentially smooth over $k$, an étale morphism $\varphi: X' \to \mathbb{A}^1_S$ and a pro-étale morphism $p: X' \to X$, such that the following hold:
\begin{enumerate}
\item $p$ is a pro-(étale neighborhood) of $z$,
\item $\varphi^{-1}(\mathbb{A}^1_s) = p^{-1}(C)$, where $s$ is the closed point of $S$, and
\item $\varphi^{-1}(\varphi(Z')) = Z'$ and $\varphi|_{Z'}: Z' \to \mathbb{A}^1_S$ is a closed immersion.
  Here $Z' = Z \times_X X'$.
\end{enumerate}
\end{lemma}
\begin{proof}
Shrinking $X$ around $z$, we may assume that $C$ is cut out by a single function, say $t$.
Set $S_0 = \mathbb{A}^1$ and view $t: X \to S$.
Pick another function $f: X \to \mathbb{A}^1$ such that $df,dt$ generate the cotangent space of $X$ at $z$; this is possible because $C$ is smooth, and so $dt \ne 0$.
Let $\varphi_0: X \to \mathbb{A}^1_{S_0}$ be given by $(t,f)$.
By construction $\varphi_0$ is étale at $z$ (use \cite[17.11.1]{EGAIV}).
If $z \not\in Z$ then shrinking $X$ further, we may assume $Z=\emptyset$.
Henselizing $S_0$ in the origin then yields what we want.
Thus from now one we assume $z \in Z$. Noting that $t^{-1}(0) \cap Z = C \cap Z$ is a finite set of closed points, we see that $t|_Z$ is quasi-finite at $z$.
Shrinking $X$ further, we can ensure that $\varphi_0$ is étale (the étale locus being open \cite[Tag 02GI(1)]{stacks-project}), $t|_Z$ is quasi-finite (the quasi-finite locus being open \cite[Tag 01TI]{stacks-project}), and $t^{-1}(0) \cap Z = \{z\}$.
Let $s = 0 \in S_0$ and set $S = S_{0,s}^h$. Write $X_S$ for the base change of $X$ along $S \to S_0$, and so on.
Noting that $Z_S$ is quasi-finite over the henselian local scheme $S$, we see that $Z_S = A_1 \amalg \dots \amalg A_n \amalg A'$, with each $A_i$ local and finite over $S$ and $A'$ having no points over $s$ \cite[Tag 04GJ]{stacks-project}.
We have $z \in A_i$ for precisely one $i$, say $z \in A_1$. We claim that $A_1 \to \mathbb{A}^1_S$ is a closed immersion.
Indeed it is finite over $S$, so we may check this in the special fiber (use Nakayama's lemma\NB{details?}).
But since $t^{-1}(0) \cap Z = \{z\}$, the special fiber of $A_1$ consists of precisely one rational point, and hence maps to $\mathbb{A}^1$ via a universal injection \cite[Tag 01S4(2,3)]{stacks-project}.
The map is also unramified (being a composite of an étale morphism and a closed immersion), and hence a closed immersion \cite[Tag 04XV]{stacks-project}.
This proves the claim. Now set $X' = (X_S \setminus \bigcup_{i>1} A_i) \setminus (\varphi^{-1}(\varphi(A_1)) \setminus A_1)$.
Noting that $A_1 \to \varphi^{-1}(\varphi(A_1))$ is open (being a section of an étale morphism), we see that $X'$ is an open subscheme of $X_S$. The required properties hold by construction.
\end{proof}

We can deduce the following.
\begin{lemma} \label{lemm:contr2}
Let $X$ be smooth over a perfect field $k$, $z \in X$ a point of codimension 2, $y \in X_z^h$ a point of codimension 1 specializing to $z$ and $F$ a strongly $\A^1$-invariant sheaf.
Assume that $z$ is a smooth point of $\overline{\{y\}}$.
Then $H^1(X_z^h \setminus \overline{\{y\}}, F) = 0$.
\end{lemma}
\begin{proof}
Using that $k$ is perfect, we can find an essentially smooth map $X_z^h \to z$ and thus assume that $z$ is rational and $X$ has dimension 2 (though now $k$ may no longer be perfect!); see \S\ref{subsubs:smooth-retract}.
Let $a \in H^1(X_z^h \setminus \overline{\{y\}},F)$.
Replacing $X$ by an étale neighborhood of $z$, we may assume given a smooth curve $C \subset X$ with preimage $\overline{\{y\}}$ in $X_z^h$, and also that $a$ is the image of $\tilde a \in H^1(X \setminus C, F)$.
Since $H^1$ of the generic point of $X$ vanishes, $\tilde a$ is supported on a proper closed subscheme $Z \setminus C \subset X \setminus C$, where $C \not\subset Z$.
Apply Lemma \ref{lemm:contr1} to $(X,Z,C,z)$ and replace $X$ by $X'$. We now have an étale map $\varphi: X \to \mathbb{A}^1_S$ with $\varphi^{-1}(\mathbb{A}^1_s) = C$, which is moreover an étale neighborhood of $Z$.
It induces $X \setminus C \to \mathbb{A}^1_S \setminus \mathbb{A}^1_s \simeq \mathbb{A}^1_K$.
Since $\tilde a$ is supported on $Z \setminus C$ (that is, lies in the image of the map $H^1_{Z \setminus C}(X \setminus C, F) \to H^1(X \setminus Z, F)$), it lifts to $H^1(\mathbb{A}^1_K, F)$. But this group vanishes by strong invariance of $F$, concluding the proof.
\end{proof}

\begin{proposition}
Notation and assumptions as in Construction \ref{constr:local-H2-comparison}.
The map $\alpha_{x,y}^{\mu,\lambda}$ is an isomorphism.
\end{proposition}
\begin{proof}
By the construction involving exact sequences, the map is injective; we thus need only prove it is surjective.
We shall use the fact that the construction of $\alpha$ is compatible with étale extensions, as explained in Remark \ref{rmk:alpha-etale-extension}, and in fact we shall use the same notation.
Assuming that $x' \to x$ is an isomorphism, we find that both vertical maps are isomorphisms.
In particular, the bottom horizontal map is an isomorphism if and only if the top one is.
It follows that we may replace $X$ by $X_x^h$, and so assume that $X$ is henselian local with closed point $x$.
Note that $\{y\} = Y \setminus x$ is a closed subscheme of $X \setminus x$.
The relevant boundary map thus has the form $H^1_y(X, F) \to H^2_x(X, F)$.
Set $C=\overline{\{y\}}$.
The cofiber sequence $X \setminus C \to X \setminus x \to X \setminus x/X \setminus C$ implies by passing to cohomology and using Lemma \ref{lemm:contr2} that $H^1_{y}(X, F) \to H^1(X \setminus x,F)$ is surjective.
On the other hand the boundary map $H^1(X \setminus x,F) \to H^2_x(X,F)$ is surjective since $H^2(X,F)=0$.
The map $\alpha_{x,y}^{\mu,\lambda}$ is thus the composite of two surjections.
This concludes the proof.
\end{proof}

Next we must study the dependence of the isomorphism $\alpha_{x,y}^{\mu,\lambda}$ on $y,\mu,\lambda$.
Fix $x \in X$ and consider the local ring $\scr O_{X,x}$.
Choose elements $\pi, \rho \in \scr O_{X,x}$ which generate the maximal ideal (this is possible since $\scr O_{X,x}$ is regular of dimension $2$ \cite[Tag 056S]{stacks-project}).
Set $Y=Z(\pi)$.
This is a smooth closed subscheme of $X_x$.
The elements $\rho$, $\pi$ define trivializations $\mu(\pi,\rho)$ of $N_{x/Y_x}$ and $\lambda(\pi,\rho)$ of $N_{Y_x/X}$.
Thus for each choice $\pi, \rho$ we obtain $y,\mu,\lambda$, hence hence an isomorphism \[ \alpha_x^{\pi,\rho} := \alpha_{x,y(\pi,\rho)}^{\mu(\pi,\rho),\lambda(\pi,\rho)}: F_{-2}(x) \xrightarrow{\wequi} H^2_x(X,F). \]
We can also build \[ \tilde\alpha_x^{\pi,\rho}: F_{-2}(x,\omega_{x/X}) \stackrel{\bar\pi \wedge \bar\lambda}{\wequi} F_{-2}(x) \stackrel{\alpha_x^{\pi,\rho}}{\wequi} H^2_x(X,F). \]
\begin{proposition} \label{prop:pi-rho-independent}
The isomorphism $\tilde\alpha_x^{\pi,\rho}$ does not depend on $\pi$ or $\rho$.
\end{proposition}
\begin{proof}
For fixed $\pi$, the element $\lambda$ only enters through the trivialization of $N_{x/Y_x}$, i.e., the backwards arrow in Construction \ref{constr:local-H2-comparison}.
In particular, for $a \in \scr O^\times_{X,x}$, $b \in \scr O_{X,x}$, $c \in m_{X,x}^2$ we get \[ \alpha_x^{\pi,\rho} = \alpha_x^{\pi,a \lambda + b \rho + c}\lra{a} \text{ and thus } \tilde\alpha_x^{\pi,\rho} = \tilde\alpha_x^{\pi,a \lambda + b \rho + c}. \]
Here we have used that $\partial^y_x: M_{-1}(y) \to M_{-2}(x)$ is linear over $\KMWn_0(Y_x)$, which is proved in Lemma \ref{lemm:differential-linear} in the next section.
In the next section we shall also see that (Lemma \ref{lemm:H2-coordinate-swap}) \[ \alpha_x^{\pi,\rho} =\alpha_x^{\rho,\pi}\lra{-1} \text{ so that } \tilde\alpha_x^{\pi,\rho} = \tilde\alpha_x^{\rho,\pi}. \]
This implies what we want since any other system of generators $(\pi', \rho')$ can be reached from $(\pi,\rho)$ by finitely many operations of these two forms.\NB{details?}
\end{proof}

\begin{definition} \label{def:identify-M-2}
The upshot is that we have, for any strongly invariant sheaf $F$ of abelian groups and $x \in X \in \Sm_k$, a canonical isomorphism \[ \alpha_x: F_{-2}(x, \omega_{x/X}) \xrightarrow{\wequi} H^2_x(X,F). \]
Indeed we set $\alpha_x = \tilde \alpha_x^{\pi,\rho}$ for any choice of $(\pi, \rho)$; this is independent of the choice by Proposition \ref{prop:pi-rho-independent}.
\end{definition}

\subsection{Contractions and transfers} \label{subsec:contractions-transfers}
We still assume $k$ perfect.

Let $L/K$ be a finite extension of fields which are finitely generated over $k$.
Suppose that $L$ can be generated by one element over $K$ and pick such a generator $x \in L$, that is, pick an embedding $x: \Spec(L) \hookrightarrow \A^1_K$.
Applying $[(\ph), K(F,1)]$ to the projection map $\P^1_K \to \P^1_K/\P^1_K \setminus \Spec(L)$ we get a map\todo{sloppy} \[ H^1_x(\P^1_K, F) \to [\Sigma \Gm \wedge \Spec(K)_+, K(F,1)] \wequi F_{-1}(K). \]
Lemma \ref{lemm:identify-local-cohomology} supplies us with an isomorphism $H^1_x(\P^1_K, F) \wequi F_{-1}(L, \omega_{x/\P^1_K})$.
Note that $\omega_{x/\P^1_K} \wequi \omega_{L/K} \otimes (\omega_{\P^1_K/K}|_x)^{-1}$ (use Remark \ref{rmk:omega-composition}).
The second term is canonically trivial, since $\omega_{\P^1_K/K}|_{\A^1_K}$ is trivialized via the coordinate $t$ on $\A^1$.
We can thus rewrite our map as \[ \tr_{L/K}^x: F_{-1}(L, \omega_{L/K}) \to F_{-1}(K). \]

The following will be proved in the next section (see \S\ref{subsub:gersten-M-1}).
\begin{lemma} \label{lemm:transfer-linear}
If $F=M_{-1}$, the map $\tr_{L/K}^x: F_{-1}(L, \omega_{L/K}) \to F_{-1}(K)$ is linear over $\KMWn_0(K)$.
\end{lemma}

Now let $\scr L$ be a line bundle on $\Spec(K)$.
We can define the \emph{twisted transfer map} \[ \tr_{L/K}^x = \tr_{L/K}^x \otimes \scr L: F_{-1}(L, \omega_{L/K} \otimes \scr L|_L) \to F_{-1}(K, \scr L) \] by choosing a nonvanishing section $s \in \scr L^\times$ and setting $\tr_{L/K}^x(m \otimes s) = \tr_{L/K}^x(m) \otimes s$.
This is well-defined by Lemma \ref{lemm:transfer-linear}.

\begin{construction}
Let $L/K$ be a finite extension of fields which are finitely generated over $k$.
Choose a sequence of subfields $K=L_0 \subset L_1 \subset \dots \subset L_n = L$ such that $L_{i+1}/L_i$ is monogeneic, and pick a generator $x_i \in L_i$.
Define the map \[ \tr_{L/K}^{(x_i)_i}: F_{-1}(L, \omega_{L/K}) \to F_{-1}(K) \] as the composite \[ \tr_{L/K}^{(x_i)_i} = \tr_{L_1/K}^{x_1} \circ (\tr_{L_2/L_1}^{x_1} \otimes \omega_{L_1/K}) \circ \dots \circ (\tr_{L/L_{n-1}}^{x_n} \otimes \omega_{L_{n-1}/K}). \]
Here we use the twisted transfers, together with the canonical isomorphisms $\omega_{L_{i+1}/L_i} \otimes \omega_{L_i/K} \wequi \omega_{L_{i+1}/K}$.
\end{construction}

The following will also be proved in the next section (see \S\ref{label:proof-transfers-welldefined}).
\begin{theorem} \label{thm:transfers}
Assume that $F = M_{-1}$, for some strongly $\A^1$-invariant sheaf of abelian groups $M$.
Then $\tr_{L/K}^{(x_i)_i}$ is independent of the choice of the $x_i$.
\end{theorem}

Using this, we can define the ``canonical transfer''.
\begin{definition}
Let $L/K$ be a finite extension of fields which are finitely generated $k$, and $M$ a strongly $\A^1$-invariant sheaf of abelian groups.
Denote by \[ \tr_{L/K} := \tr_{L/K}^{(x_i)_i}:  M_{-2}(L, \omega_{L/K}) \to M_{-2}(K) \] the map obtained by some choice of sequence of generators $x_1, \dots, x_n$ of $L/K$.
This is independent of the choices by Theorem \ref{thm:transfers}.
\end{definition}

\begin{remark}
A choice of generator $x$ for an extension $L/K$ induces a trivialization of $\omega_{L/K}$: either $\Omega^1_{L/K}$ is trivial, or it is one dimensional with basis $dx$; in either case the determinant is trivialized.
The resulting transfer $M_{-2}(L) \to M_{-2}(K)$ is however not independent of choices.
\end{remark}

\begin{remark}
Let $L$ be a finite product of finite field extensions of $K$, $L = \prod_{i=1}^n L_i$.
Then $M_{-1}(L, \omega_{L/K}) \wequi \prod_{i=1}^n M_{-1}(L_i, \omega_{L_i}/K)$ and we define $\tr_{L/K}: M_{-1}(L, \omega_{L/K}) \to M_{-1}(K)$ by using the canonical transfer on each factor.
\end{remark}

Let us point out for future reference that the canonical transfers are compatible with smooth base change.
\begin{lemma} \label{lemm:transfer-smooth-base-change}
Let $L/K$ be a finite extension and $K'/K$ be a separable extension, so that $L' := L \otimes_K K'$ is a finite product of finite extensions of $K'$.
Then the following diagram commutes
\begin{equation*}
\begin{CD}
M_{-2}(L, \omega_{L/K}) @>>> M_{-2}(L', \omega_{L'/K'}) \\
@V{\tr_{L/K}}VV @V{\tr_{L'/K'}}VV \\
M_{-2}(K) @>>> M_{-2}(K').
\end{CD}
\end{equation*}
Here we have used the canonical isomorphism $\omega_{L'/K'} \wequi \omega_{L/K}$, and the horizontal maps are restrictions.
\end{lemma}
\begin{proof}
We may reduce to the case where $L=K(x)$ is monogeneic.
The commutative diagram of cartesian squares
\begin{equation*}
\begin{CD}
\Spec L @>>> \P^1_K @>>> \Spec K \\
@AAA  @AAA @AAA \\
\Spec L' @>>> \P^1_{K'} @>>> \Spec K' \\
\end{CD}
\end{equation*}
induces the commutative square on the left hand side of the following diagram
\begin{equation*}
\begin{CD}
\P^1_K @>>> \P^1_K/\P^1_K \setminus \Spec L @>>> \Th(N_{\Spec L/\P^1_K}) \\
@AAA             @AAA                               @AAA \\
\P^1_{K'} @>>> \P^1_{K'}/\P^1_{K'} \setminus \Spec L' @>>> \Th(N_{\Spec L'/\P^1_{K'}}). \\
\end{CD}
\end{equation*}
The right hand square commutes by naturality of the purity equivalence (Theorem \ref{thm:purity}).
The outer square encodes our desired commutativity (using Lemma \ref{lemm:twisting-V}).
\end{proof}

\section{The Gersten complex} \label{sec:gersten}
In this section we define and study the Gersten complex of a strongly $\A^1$-invariant sheaf.
First in \S\ref{subsec:gersten-diff} we define the Gersten differentials and use them to express some of the isomorphisms denoted $\alpha$ in the previous section.
Then in \S\ref{subsec:coniveau-filtn} we introduce the coniveau filtration and associated spectral sequence.
This allows us to quickly define the Gersten complex and establish some of its basic properties.
Next in \S\ref{subsec:canonical-diff} we define what we call the \emph{canonical differential}; this is just a twisted form of the first Gersten differential.
Putting all our hard hard work up to this point together, we can establish purity of the Gersten complex in codimension $\le 2$ (Corollary \ref{cor:purity-codim2}).
We put this to good use in \S\ref{subsec:gersten-comp} to actually compute some cohomology groups (e.g., we show that $H^2(\A^2, M) = 0$ for any strongly $\A^1$-invariant sheaf\footnote{Note how surprising this is: a priori, strong $\A^1$-invariance of $M$ is only a statement about $H^1$.}).
Finally in \S\ref{label:proof-transfers-welldefined} show that the transfers on $M_{-2}$ are well-defined.

\subsubsection*{Assumptions and notation}
We work over a perfect field $k$ throughout.

Recall that for a scheme $X$ and a point $x \in X$, we call $\dim \scr O_{X,x}$ the \emph{codimension of $x$}.
We write $X^{(c)} \subset X$ for the set of points of codimension $c$.

A point $y \in X$ is called a \emph{specialization} of $x$ if $y \in \overline{\{x\}}$ (equivalently, $x \in X_y$).
We call $y$ an \emph{immediate specialization} if for every point $z \in X$ such that $x$ specializes to $z$ and $z$ specializes to $y$, either $x=z$ or $y=z$.
In reasonable cases, whenever $y$ is a specialization of $x$, it is an immediate specialization if and only if $\mathrm{codim}(x) = \mathrm{codim}(y) - 1$ (this will hold on catenary schemes, e.g., essentially finite type schemes over fields \cite[Tag 0ECE]{stacks-project}.

\subsection{The Gersten differentials} \label{subsec:gersten-diff}
Let $X \in \Sm_k$, $x \in X$ of codimension $c$ and $z \in X$ an immediate specialization of $x$.
We have the cofiber sequence \[ X_x/X_x \setminus x \wequi X_z \setminus z/X_z \setminus \overline{\{x\}} \to X_z/X_z \setminus \overline{\{x\}} \to X_z/X_z \setminus z \in \Shv_\Nis(\Sm_k). \]
(Recall our conventions regarding essentially smooth schemes and pro-objects from \S\ref{subsub:ess-smooth}.)
Here the first equivalence arises from excision (see Remark \ref{rmk:excision}), since $X_x$ is a pro-(open neighborhood) of $x = X_z \cap \overline{\{x\}}$.
Let $M$ be a Nisnevich sheaf of abelian groups on $\Sm_k$.
Extending the cofiber sequence to the right and taking maps into $K(M, \ph)$ we obtain a long exact sequence, which is called the long exact sequence of cohomology with support.
\begin{definition} \label{def:gersten-diff}
For a sheaf of abelian groups $M$, the boundary map in the long exact sequence of cohomology with support of the form \[ \partial^x_z: H^c_x(X,M) \to H^{c+1}_z(X, M) \] is called the \emph{Gersten differential}.
\end{definition}

We explain the naturality of this construction.
\begin{lemma} \label{lemm:gersten-diff-natural}
Let $f: X \to Y$ be a morphism of essentially smooth schemes, $x \in X$ of codimension $c$, $z \in X$ an immediate specialization.
Assume that $f(z)$ is an immediate specialization of $f(x)$ (e.g., $f$ étale).

Then the following square commutes
\begin{equation*}
\begin{CD}
H^c_x(X,M) @>{\partial^x_z}>> H^{c+1}_z(X, M) \\
@A{f^*}AA                     @A{f^*}AA \\
H^c_{f(x)}(Y,M) @>{\partial^{f(x)}_{f(z)}}>> H^{c+1}_{f(z)}(Y, M).
\end{CD}
\end{equation*}
\end{lemma}
\begin{proof}
Since $X_z \cap \overline{\{x\}} = \{x,z\}$ and similarly for $Y$, we have an induced square of pairs of schemes
\begin{equation*}
\begin{CD}
(X_x, X_x \setminus x) @>>> (X_z, X_z \setminus \overline{\{x\}}) \\
@VVV                           @VVV \\
(Y_{f(x)}, Y_{f(x)} \setminus f(x)) @>>> (Y_{f(z)}, Y_{f(z)} \setminus \overline{\{f(x)\}}).
\end{CD}
\end{equation*}
From this we obtain an induced morphism between the cofiber sequences used in the construction of the Gersten differential.
\end{proof}

We shall express some of the maps denoted $\alpha$ in the previous section using the Gersten differentials.
Here is a start.
\begin{lemma} \label{lemm:boundary-iso-trick}
Let $K/k$ be finitely generated, $X/K$ essentially smooth, $x \in X$ of codimension $1$ with generic point $\eta$.
Assume that $x$ is a $K$-rational point.
Pick a uniformizer $\pi$ of $\scr O_{X,x}$.
Then the composite \[ M_{-1}(x, \omega_{x/X}) \stackrel{\pi}{\wequi} M_{-1}(x) = M_{-1}(K) \to M_{-1}(\eta) \xrightarrow{[\pi]} M(\eta) \xrightarrow{\partial} H^1_x(X, M) \] coincides with the canonical isomorphism from Lemma \ref{lemm:identify-local-cohomology}.
\end{lemma}
Here the isomorphism denoted $\pi$ arises from the identification $\omega_{x/X} \wequi (m_x/m_x^2)^*$ (by definition $\pi$ is a generator of the maximal ideal $m_x$), and the morphism $[\pi]$ denotes the action of the element $[\pi] \in \KMWn_1(\eta)$ (Proposition \ref{prop:KMW-action}).
\begin{proof}
The map $\pi$ induces a cartesian square
\begin{equation*}
\begin{CD}
X_x \setminus x @>>> X_x \\
@V{\pi}VV         @V{\pi}VV \\
\Gm_K @>>> \A^1_K.
\end{CD}
\end{equation*}
The map $X_x \to \A^1_K$ is a (pro-)étale neighborhood of $x$, since $d\pi$ generates the cotangent space of $X_x$ at $x$.
Consider the induced commutative square
\begin{equation*}
\begin{CD}
H^0(\Gm_K, M)/H^0(\A^1_K,M) @>{\partial}>> H^1_{\Gm_K}(\A^1_K, M) \\
@V{\pi^*}VV                                @V{\pi^*}VV \\
H^0(X_x \setminus x,M)/H^0(X_x, M) @>{\partial}>> H^1_x(X_x, M).
\end{CD}
\end{equation*}
The top left hand corner is $M_{-1}(K)$, and the map to the bottom right hand corner via top right hand corner is the isomorphism from Lemma \ref{lemm:identify-local-cohomology} (use that the purity equivalence is natural, and for $0 \subset \A^1$ is the obvious one; see Theorem \ref{thm:purity}).
Since the map via the bottom left hand corner is $\partial([\pi](\ph))$ (indeed the action by $[\pi]$ is by definition given by pullback along the map $\eta \wequi X_x \setminus X \xrightarrow{\pi} \Gm$), the result follows.\NB{this proof is less clear than I would like}
\end{proof}

Now we can treat a more complicated case.
\begin{lemma} \label{lemm:double-boundary-iso-trick}
Let $K/k$ be finitely generated, $X/K$ essentially smooth, $x \in X$ of codimension $2$.
Pick $\pi, \rho$ as before Proposition \ref{prop:pi-rho-independent}.
Assume that $x$ is $K$-rational.
The composite \[ M_{-2}(x) \to M_{-2}(\eta) \xrightarrow{[\pi][\rho]} M(\eta) \xrightarrow{\partial^\eta_y} H^1_y(X,M) \xrightarrow{\partial^y_x} H^2_x(X,M) \] coincides with $\tilde\alpha_x^{\pi,\rho}$.
\end{lemma}
\begin{proof}
Let $m \in M_{-2}(x)$.
By definition we have $\alpha_x^{\pi,\rho}(m) = \partial^y_x(m')$, where $m'$ corresponds to $m$ under the canonical isomorphisms \[ M_{-2}(x) \wequi H^1_y(Y,M_{-1}) \wequi M_{-1}(y)/M_{-1}(Y_x) \wequi H^1(X_x \setminus x/X_x \setminus Y_x, M)/H^1(X_x/X_x \setminus Y_x, M) \wequi H^1_y(X,M)/(\dots). \]
By Lemma \ref{lemm:boundary-iso-trick}, the image of $m$ under the first isomorphism is $\partial^\eta_y([\rho]m)$.
Consequently its image in the fourth group coincides with the image of $[\rho]m \in M_{-1}(y)$ under the isomorphism $M_{-1}(y) \wequi H^1_y(X, M) \wequi  H^1(X_x \setminus x/X_x \setminus Y_x, M)$.
Let $X_y^h$ denote the henselization of $y$ in $X$.
Then there is a pro-smooth map $X_y^h \to y$ (see Corollary \ref{cor:ess-smooth-retraction}), making $y$ into a $y$-rational point.
Let $\eta^h$ be the generic point of $X_y^h$.
Then applying Lemma \ref{lemm:boundary-iso-trick} to $X_y^h$, we learn that the image of $[\rho]m$ in $H^1_y(X,M)$ is given by $\partial^{\eta^h}_y([\pi][\rho]m|_{\eta^h})$.
Since the differential is compatible with étale extensions (see Lemma \ref{lemm:gersten-diff-natural}), this is the same as $\partial^\eta_y([\pi][\rho]m)$.
This concludes the proof.
\end{proof}

\subsubsection{Gersten differentials on $M_{-1}$} \label{subsub:gersten-M-1}
Let $X$ be essentially smooth, $y \in X$ a point and $x \in X$ an immediate specialization.
\begin{remark} \label{rmk:gersten-diff-linear}
The sheaf $M_{-1}|_{(X_x)_\Nis}$ is a module over $\KMWn_0(X_x)$.
It follows that the Gersten differential \[ \partial^y_x: H^d_y(X, M_{-1}) \to H^{d+1}_x(X, M) \] is linear over $\KMWn_0(X_x)$.
\end{remark}

\begin{lemma} \label{lemm:alpha-1d-linear}
Let $X$ be essentially smooth, $\eta \in X$ a generic point and $x \in X$ an immediate specialization.
The isomorphism \[ \alpha_x: M_{-2}(x, \omega_{x/X}) \to H^1_x(X, M_{-1}) \] is linear over $\KMWn_0(X_x)$.
\end{lemma}
\begin{proof}
Using compatibility of Gersten differentials with étale extensions (Lemma \ref{lemm:gersten-diff-natural}), as well as the existence of étale-local smooth retractions (Corollary \ref{cor:ess-smooth-retraction}), we may assume $X$ pro-smooth over $x$.
Lemma \ref{lemm:boundary-iso-trick} now tells us that $\alpha_x((\ph) \otimes d\pi) = \partial^\eta_x([\pi] (\ph)|_\eta)$, which is a composite of linear maps.
\end{proof}

Using these ideas, we can prove that the transfer is linear over the base.
\begin{lemma} \label{lemm:transfer-coh-P1}
Let $K/k$ be a finitely generated, separable extension and $M$ a strongly $\A^1$-invariant sheaf of abelian groups.
\begin{enumerate}
\item Let $x \in \P^1_K$ be rational.
  The canonical map $H^1_x(\P^1_K, M) \to H^1(\P^1_K, M)$ is an isomorphism.
  Identifying both sides with $M_{-1}(K)$, the map becomes the identity.
\item Let $L/K$ be a finite extension with chosen generator $y$.
  The composite \[ M_{-1}(L, \omega_{L/K}) \stackrel{\alpha_y}{\wequi} H^1_y(\P^1_K,M) \to H^1(\P^1_K,M) \stackrel{\alpha_\infty}{\wequi} M_{-1}(K) \] is $\tr_{L/K}^y$.
\end{enumerate}
\end{lemma}
\begin{proof}
(1) This just says that $\P^1_K \to \P^1_K/\P^1_K \setminus x$ is an equivalence, which is clear since $\P^1_K \setminus X \wequi \A^1_K \wequi \Spec(K)$.\todo{sloppy}
The second claim is that $\P^1_K \to \P^1_K/\P^1_K \setminus x \wequi \P^1 \wedge x_+$ is the identity, which is part of the definition of the purity isomorphism (see Theorem \ref{thm:purity}).

(2) This is just a translation of the definition via (1).
\end{proof}

We can now establish one of the lemmas left over from the previous section: transfers on $M_{-2}$ are linear over $\ul{K}_0^{MW}$.
\begin{proof}[Proof of Lemma \ref{lemm:transfer-linear}.]
We must prove that in the situation of (2), the transfer for $M_{-1}$ of the form $\tr_{L/K}^y: M_{-2}(L, \omega_{L/K}) \to M_{-2}(K)$ is linear over $\ul{K}_0^{MW}(K)$.
But each of the maps in the composite from (2) is linear, by Lemma \ref{lemm:alpha-1d-linear} and the fact that $M_{-1}|_{\P^1_K}$ is a module over $\ul{K}_0^{MW}(K)$.
\end{proof}

\subsection{The coniveau filtration} \label{subsec:coniveau-filtn}
Let $X$ be an essentially smooth $k$-scheme and $E$ a sheaf of spectra on $X_\Zar$.
The \emph{coniveau filtration} on $E$ is obtained as \[ (F_iE)(\ph) = \colim_{Z \subset X, codim(Z) > i} F((\ph) \setminus Z). \]
The colimit is over closed subschemes of codimension $>i$; note that this colimit is filtered.
We have $F_{-1} E = 0$ and $F_n E = E$ for $n = \dim X$.
Since a subscheme of codimension $>i$ is also of codimension $>i-1$, we have natural maps $F_iE \to F_{i-1}E$ and thus a filtration \[ E = F_n E \to F_{n-1} E \to \dots \to F_0 E \to F_{-1} E = 0. \]
Taking global sections and homotopy groups yields a spectral sequence converging to $\pi_* E(X)$ (see, e.g., \cite[\S1.2.2]{lurie-ha}).
On the first page we find the homotopy groups of the layers.
Observe that \[ \fib(F_i E \to F_{i-1} E) \wequi \bigoplus_{x \in X^{(i)}} E_x(X). \]
Here the sum is over points of codimension $i$, and $E_x(X)$ denotes the evident analog of local cohomology.
(See \cite[Lemma 1.2.1]{bloch-ogus-gabber-theorem} for a proof.)

Consider now the case $E = HM$, $M$ a sheaf of abelian groups on $\Sm_k$.
In this case the spectral sequence has signature \[ E_1^{p,q} = \bigoplus_{x \in X^{(p)}} H^{p+q}_x(X,M) \Rightarrow H^{p+q}(X, M). \]
The first differential takes the form \[ d_1^{p,q}: E_1^{p,q} \to E_1^{p+1,q}. \]
Concentrating on the line $q=0$, we thus find a complex of abelian groups.
\begin{definition}
The \emph{Gersten complex} for $M$ over $X$ is the complex $C^*(X,M)$ of abelian groups arising from the first page of the coniveau spectral sequence.
Its groups are \[ C^p(X,M) = \bigoplus_{x \in X^{(p)}} H^p_x(X,M). \]
\end{definition}
The details of the construction of a spectral sequence (see, e.g., \cite[\S1]{bloch-ogus-gabber-theorem}) imply the following:
\begin{itemize}
\item If $x \in X^{(p)}$ and $y \in X^{(p+1)}$ is a specialization of $x$, then the component of the differential in $C^*(X,M)$ from $H^p_x(X,M)$ to $H^q_y(X,M)$ is the Gersten differential $\partial^x_y$ of Definition \ref{def:gersten-diff}.
\item If $x \in X^{(p)}$ and $y \in X^{(p+1)}$ is not a specialization of $x$, then the corresponding component of the differential in $C^*(X,M)$ is zero.
\end{itemize}

For schemes of dimension $\le 2$, the Gersten complex in fact computes cohomology:
\begin{lemma} \label{lemm:gersten-cohomology-dim2}
Let $X$ be essentially smooth of dimension $\le 2$ and $M$ strongly $\A^1$-invariant.
Let $x \in X$ be a point of codimension $i$.
Then $H^j_x(X,M) = 0$ for $j \ne i$.

There is also a canonical isomorphism $H^* C^*(X,M) \wequi H^*(X,M)$.
\end{lemma}
\begin{proof}
The long exact sequence for cohomology with support yields an exact sequence \[ H^{j-1}(X_x \setminus x, M) \to H^j_x(X,M) \to H^j(X_x, M). \]
For $j>i$ the outer-most groups are zero (since, respectively, $\dim X_x \setminus x = i-1$ and $\dim X_x = i$) and we get the desired vanishing.
If $j<i$ then $j \le 1$, whence $K(M,j)$ is $\A^1$-invariant and we compute $[X_x/X_x \setminus x, K(M,j)] \wequi [\Sigma^i \Gmp{i} \wedge x_+, K(M,j)] = 0$, since $\Omega^i K(M,j)=0$.

We have now proved that the $E_1$-page consists of $C^*(X,M)$ on the line $q=0$ and zero else, hence collapses at $E_2$ to $H^* C^*(X,M)$.
Since it converges to $H^*(X,M)$, the final claim follows.
\end{proof}

We can now provide a crucial ingredient to the proof of Proposition \ref{prop:pi-rho-independent}.
\begin{lemma} \label{lemm:H2-coordinate-swap}
In the notation of Lemma \ref{prop:pi-rho-independent}, we have $\tilde\alpha_x^{\pi,\rho} = \tilde\alpha_x^{\rho,\pi}\lra{-1}$.
\end{lemma}
\begin{proof}
Henselizing $X$ in $x$, we may assume given a pro-smooth map $X \to x$ (see Remark \ref{rmk:reduce-to-closed-point}).
By Lemma \ref{lemm:double-boundary-iso-trick} we have \[ \tilde\alpha_x^{\rho,\pi}m = \partial^{y_1}_x \partial^\eta_{y_1} ([\rho][\pi]m'), \] where $m'$ is the pullback of $m$ to $\eta$ along $\eta \to X \to x$.
Here $y_1$ is the generic point of $Z(\pi)$.
Writing $y_2$ for the generic point of $Z(\rho)$, we similarly have \[ \tilde\alpha_x^{\pi,\rho}m = \partial^{y_2}_x \partial^\eta_{y_2} ([\pi][\rho]m'). \]

Consider the Gersten complex $C^*(X_x,M)$.
The class $[\pi][\rho]m'$ lifts to $M(X_x \setminus (Z(\pi) \cup Z(\rho)))$ (since $\pi$ and $\rho$ are units there) and consequently $\partial^\eta_y([\pi][\rho]m') = 0$ for all $y \in (X_x \setminus (Z(\pi) \cup Z(\rho)))^{(1)}$, i.e. for all $y \in X^{(1)} \setminus \{y_1,y_2\}$.
Since $C^*(X_x,M)$ is a complex, we find that \[ 0 = \partial^2([\pi][\rho]m') = \partial^{y_2}_x \partial^\eta_{y_2} ([\pi][\rho]m') + \partial^{y_1}_x \partial^\eta_{y_1} ([\pi][\rho]m'). \]
Combining this with $-[\pi][\rho] = \lra{-1}[\rho][\pi]$ (see Remark \ref{rmk:KMW-extra-relations}) we see that \[ \tilde\alpha_x^{\pi,\rho}(m) = \tilde\alpha_x^{\rho,\pi}(\lra{-1}m). \]
\end{proof}

Now that the map $\alpha_x$ (for $x$ of codimension $2$) is well-defined, let us observe that it is linear.
\begin{lemma}\label{lemm:alpha-2d-linear}
Let $X$ be essentially smooth, $x \in X$ of codimension $2$, $\eta \in X$ the corresponding generic point and $M$ a strongly $\A^1$-invariant sheaf of abelian groups.
The map $\alpha_x$ for $M_{-1}$, \[ \alpha_x: M_{-3}(x, \omega_{x/X}) \to H^2_x(X, M_{-1}) \] is linear over $\KMWn_0(X_x)$.
\end{lemma}
\begin{proof}
As usual (Remark \ref{rmk:reduce-to-closed-point}) we can assume $X$ pro-smooth over $x$, and choose generators $\pi, \rho$ of the maximal ideal.
Now by Lemma \ref{lemm:double-boundary-iso-trick} we have $\alpha_x((\ph) \otimes d\pi \wedge d\rho) \wequi \partial^y_x \partial^\eta_y ([\pi][\rho](\ph))$, which is a composite of linear maps (Remark \ref{rmk:gersten-diff-linear}).
\end{proof}

\subsection{The canonical differential} \label{subsec:canonical-diff}
Let $X \in \Sm_k$, $\eta \in X$ a generic point and $x \in X$ an immediate specialization of $\eta$.
For a sheaf of abelian groups $M$, we obtain the boundary map in the long exact sequence of cohomology with support of the form $\partial: M(\eta) \to H^1_x(X, M)$.
If $M$ is strongly $\A^1$-invariant, then the target identifies canonically (Lemma \ref{lemm:identify-local-cohomology}) with $M_{-1}(x, \omega_{x/X})$.
We call the resulting map \[ \cpartial^\eta_x: M(\eta) \to M_{-1}(x, \omega_{x/X}) \] the \emph{canonical differential}.

\begin{lemma} \label{lemm:differential-linear}
In the situation above, the canonical differential for $M_{-1}$ \[ \cpartial^\eta_x: M_{-1}(\eta) \to M_{-2}(x, \omega_{x/X}) \] is linear over $\ul{K}_0^{MW}(X_x)$.
\end{lemma}
\begin{proof}
This follows by combining Remark \ref{rmk:gersten-diff-linear} and Lemma \ref{lemm:alpha-1d-linear}.
\end{proof}

\begin{construction} \label{cons:twisted-canonical-differential}
Let $X \in \Sm_k$, $\eta \in X$ a generic point and $x \in X$ an immediate specialization of $\eta$.
Let $\scr L$ be a line bundle on $X_x$.
We define the \emph{twisted canonical differential} \[ \cpartial^\eta_x = \cpartial^\eta_x \otimes \scr L: M_{-1}(\eta,\scr L|_\eta) \to M_{-2}(x, \omega_{x/X} \otimes \scr L|_x) \] by choosing a nonvanishing section $s \in \scr L^\times(X_x)$ and setting $\cpartial^\eta_x(m \otimes s|_\eta) = \cpartial^\eta_x(m) \otimes s|_x$.
This is well-defined by Lemma \ref{lemm:differential-linear}.
\end{construction}

\begin{proposition} \label{prop:second-gersten-diff-smooth}
Let $X$ be essentially smooth, $y \in X$ a point of codimension $1$ and $x \in X$ an immediate specialization of $y$ which is a smooth point of $\overline{\{y\}}$.
The following diagram commutes
\begin{equation*}
\begin{CD}
H^1_y(X,M) @>{\partial^y_x}>> H^2_x(X,M) \\
@A{\wequi}AA       @A{\wequi}A{\alpha_x}A \\
M_{-1}(y, \omega_{y/X}) @>{\cpartial^y_x \otimes \omega_{x/\overline{\{y\}}}}>> M_{-2}(x, \omega_{x/X}).
\end{CD}
\end{equation*}
\end{proposition}
\begin{proof}
In the previous subsection we have finished the proof of Proposition \ref{prop:pi-rho-independent}, that is, we know that $\alpha_x$ does not depend on the choice of intermediate point.
We may thus choose it to be $y$, and then the diagram commutes by definition (see Construction \ref{constr:local-H2-comparison}).
\end{proof}

\begin{corollary} \label{cor:second-gersten-diff-smooth-twisted}
Let $X$ be essentially smooth, $y \in X$ a point of codimension $1$ and $x \in X$ an immediate specialization of $y$ which is a smooth point of $\overline{\{y\}}$.
Let $\scr L$ be a line bundle on $X$.
The following diagrams commute
\begin{equation*}
\begin{CD}
H^0_\eta(X,M_{-1}(\scr L)) @>{\partial^\eta_y}>> H^1_y(X,M_{-1}(\scr L)) \\
@A{\wequi}A{\alpha_\eta \otimes \scr L}A       @A{\wequi}A{\alpha_y \otimes \scr L}A \\
M_{-1}(\eta, \scr L|_\eta) @>{\cpartial^\eta_y \otimes \scr L|_\eta}>> M_{-2}(y, \omega_{y/X} \otimes \scr L|_y)
\end{CD}
\end{equation*}
\begin{equation*}
\begin{CD}
H^1_y(X,M_{-1}(\scr L)) @>{\partial^y_x}>> H^2_x(X,M_{-1}(\scr L)) \\
@A{\wequi}A{\alpha_y \otimes \scr L}A       @A{\wequi}A{\alpha_x \otimes \scr L}A \\
M_{-2}(y, \omega_{y/X} \otimes \scr L|_y) @>{\cpartial^y_x \otimes \omega_{x/\overline{\{y\}}} \otimes \scr L}>> M_{-3}(x, \omega_{x/X} \otimes \scr L|_x).
\end{CD}
\end{equation*}
Here the maps denoted $\alpha \otimes \scr L$ are twists of the maps $\alpha$, and they are well-defined isomorphisms.
\end{corollary}
\begin{proof}
Choosing a trivialization of $\scr L|_y$ induces an isomorphism $C^*(X_y, M_{-1}) \wequi C^*(X_y, M_{-1}(\scr L))$, and the commutative diagrams hold by definition and Proposition \ref{prop:second-gersten-diff-smooth}.
The maps $\alpha \otimes \scr L$ are well-defined because the maps $\alpha$ are linear over $\ul{K}_0^{MW}$.
\end{proof}

\begin{remark}
The meaning of Proposition \ref{prop:second-gersten-diff-smooth} and Corollary \ref{cor:second-gersten-diff-smooth-twisted} is that in the Gersten complexes $C^*(X, M)$ (respectively $C^*(X, M_{-1})$) we can identify the groups $C^0, C^1$ and $C^2$ in terms of contractions of $M$, and we can identify many of the differentials in terms of the canonical differentials.
\end{remark}

Note that the formation of the coniveau filtration is functorial in (at least) open immersions.
Thus for $U \subset X$ open we have an obvious map $C^*(X,M) \to C^*(U,M)$, which is immediately seen to be surjective.
\begin{definition}
For $Z \subset X$ closed, we put \[ C^*_Z(X,M) = \fib(C^*(X,M) \to C^*(U,M)). \]
\end{definition}
\begin{remark}
Since $C^*(X,M) \to C^*(U,M)$ is surjective, the fiber is just the degreewise kernel.
We thus have \[ C^i_Z(X,M) = \bigoplus_{x \in X^{(i)} \cap Z} H^i_x(X,M). \]
\end{remark}

\begin{corollary} \label{cor:purity-codim2}
Let $X$ be essentially smooth of dimension $2$, $Y \subset X$ an essentially smooth closed subscheme of codimension $1$.
The isomorphisms \[ M_{-1-i}(x, \omega_{x/Y}) \otimes \omega_{Y/X} \wequi M_{-1-i}(x, \omega_{x/X}) \] induce an isomorphism of graded abelian groups \[ \alpha_{Y/X}: C^{*-1}(Y, M_{-1}(\omega_{Y/X})) \wequi C^*_Y(X,M) \] which is in fact an isomorphism of complexes.
\end{corollary}
\begin{proof}
The isomorphism of abelian groups comes from the identification of $H^i_x(X,M) \wequi M_{-i}(x,\omega_{x/X})$ for $i \le 2$ and similarly for $Y$.
The fact that this is an isomorphism of complexes is Proposition \ref{prop:second-gersten-diff-smooth} and Corollary \ref{cor:second-gersten-diff-smooth-twisted}.
\end{proof}

\subsection{Some Gersten complex computations} \label{subsec:gersten-comp}
We shall eventually prove Theorem \ref{thm:transfers} (showing that transfers are well-defined on $M_{-2}$) by analysing the Gersten complex of $\P^1 \times \P^1$.
Before we can do so, we need to analyse some related Gersten complexes.

\begin{lemma} \label{lemm:coh-A2}
Let $M$ be a strongly $\A^1$-invariant sheaf of abelian groups and $K/k$ finitely generated.
We have $H^2(\A^2_K, M) = 0$.
\end{lemma}
\begin{proof}
Let $S \subset \A^1_K$ be a finite set of closed points.
We have \[ H^2_{\A^1 \times S}(\A^2_K, M) \wequi H^1(\A^1 \times S, M(\omega_{S/K})) \wequi H^1(\A^1 \times S, M) = 0. \]
Here in the first isomorphism we have used Lemma \ref{lemm:gersten-cohomology-dim2} to compute cohomology using the Gersten complex as well as purity (Corollary \ref{cor:purity-codim2}), for the second isomorphism we use that any line bundle on $S$ is trivial, and for the third we use strong $\A^1$-invariance of $M$ (together with $\dim S = 0$).
Writing $U$ for the open complement of $S$ in $\A^1_K$ we thus find (from the long exact sequence in cohomology with support) that the restriction map \[ H^2(\A^2_K,M) \to H^2(\A^1 \times U, M) \] is injective.
Taking the colimit over all such $U$ we find that $H^2(\A^2_K, M) \hookrightarrow H^2(\A^1_\eta, M)$, where $\eta$ is the generic point of $\A^1_K$.
This latter group vanishes since $\A^1_\eta$ has dimension $1$.
\end{proof}

\begin{corollary} \label{cor:coh-A1P1}
We have $H^2((\P^1 \times \A^1)_K, M) = 0$.
\end{corollary}
\begin{proof}
Consider the following part of the long exact sequence in cohomology with support \[ H^2_{\infty \times \A^1}((\P^1 \times \A^1)_K, M) \to H^2((\P^1 \times \A^1)_K, M) \to H^2(\A^2_K,M). \]
The right hand side vanishes by Lemma \ref{lemm:coh-A2}.
The left hand side identifies using Lemma \ref{lemm:gersten-cohomology-dim2} and Corollary \ref{cor:purity-codim2} with $H^1(\infty \times \A^1_K, M_{-1}(\omega_{\infty/\A^1}))$.
Since $\omega_{\infty/\A^1}$ is trivial, the left hand group vanishes by strong $\A^1$-invariance of $M$.
\end{proof}

We now arrive at the main computation that we need.
\begin{lemma} \label{lemm:coh-of-P1xP1}
Let $K/k$ be a finitely generated extension.
For any strongly $\A^1$-invariant sheaf of abelian groups $M$, the canonical map $H^2_{(\infty,\infty)}((\P^1 \times \P^1)_K, M) \to H^2((\P^1 \times \P^1)_K, M)$ is an isomorphism.
\end{lemma}
\begin{proof}
Consider the following part of the long exact sequence in cohomology with support \[ H^1((\P^1 \times \P^1)_K, M) \to H^1((\P^1 \times \A^1)_K, M) \to H^2_{\P^1_K \times \infty}((\P^1 \times \P^1)_K, M) \to H^2((\P^1 \times \P^1)_K, M) \to H^2((\P^1 \times \A^1)_K, M). \]
The right hand group vanishes by Corollary \ref{cor:coh-A1P1}.
Since the composite $\P^1 \times \A^1 \hookrightarrow \P^1 \times \P^1 \xrightarrow{pr_1} \P^1 \hookrightarrow \P^1 \times \A^1$ is $\A^1$-homotopic to the identity, the map $H^1((\P^1 \times \P^1)_K, M) \to H^1((\P^1 \times \A^1)_K, M)$ has a section and so is surjective.
We deduce that \[ H^2_{\P^1_K \times \infty}((\P^1 \times \P^1)_K, M) \xrightarrow{\wequi} H^2((\P^1 \times \P^1)_K, M). \]
The map $H^2_{\infty \times \infty}((\P^1 \times \P^1)_K, M) \to H^2_{\P^1_K \times \infty}((\P^1 \times \P^1)_K, M)$ is induced by the map of Gersten complexes $C^*_{\infty \times \infty}((\P^1 \times \P^1)_K, M) \to C^*_{\P^1_K \times \infty}((\P^1 \times \P^1)_K, M)$ (Lemma \ref{lemm:gersten-cohomology-dim2}).
Using purity for the Gersten complex (Corollary \ref{cor:purity-codim2}) we can rewrite this as $C^*_{\infty \times \infty}(\P^1_K \times \infty, M_{-1}(\omega_{\infty/\P^1})) \to C^*(\P^1_K \times \infty, M_{-1}(\omega_{\infty/\P^1}))$, which in cohomology yields the canonical map $H^1_\infty(\P^1_K, M_{-1}) \to H^1(\P^1_K, M_{-1})$.
This is an isomorphism by Lemma \ref{lemm:transfer-coh-P1}.
\end{proof}

\subsection{Proof of Theorem \ref{thm:transfers}} \label{label:proof-transfers-welldefined}
We can now prove that the transfers are well-defined on $M_{-2}$.
Let $L = K(x,y)$ with $K/k$ finitely generated.
The elements $x, y \in L$ define a point $(x,y) \in \A^2_K \subset (\P^1 \times \P^1)_K$.
Note that $\omega_{(x,y)/(\P^1 \times \P^1)_K} \wequi \omega_{L/K}$, using the canonical trivialization of $\omega_{\A^2_K/K}$.
\begin{lemma} \label{lemm:double-transfer-trick}
The composite \[ M_{-2}(L, \omega_{L/K}) \stackrel{\alpha_{(x,y)}}{\wequi} H^2_{(x,y)}((\P^1 \times \P^1)_K, M) \to H^2((\P^1\times\P^1), M) \stackrel{L.\ref{lemm:coh-of-P1xP1}}{\wequi} M_{-2}(K) \] coincides with $\lra{-1}\tr_{K(y)/K}^y\tr_{L/K(y)}^x$.
\end{lemma}
\begin{proof}
Consider the following diagram.
All schemes are over $K$, which we drop from the notation (e.g., $\P^1 \times \P^1$ means $(\P^1 \times \P^1)_K$).
\begin{equation*}
\begin{tikzcd}[cramped,sep=small]
{H^2_{(x,y)}}(\P^1 \times \P^1, M) \arrow[dd]     & {H^1_{(x,y)}(\P^1 \times y, M_{-1}(\omega_{K(y)/K}))} \arrow[l] \arrow[d] &                                         & M_{-2}(L, \omega_{L/K}) \arrow[d, "\tr_{L/K(y)}^x"] \arrow[ll]           \\
                             & H^1(\P^1 \times y, M_{-1}(\omega_{K(y)/K})) \arrow[ld]                    & {H^1_{(\infty,y)}(\P^1 \times y, M_{-1}(\omega_{K(y)/K}))} \arrow[ld] \arrow[l] & M_{-2}(K(y), \omega_{K(y)/K}) \arrow[l] \arrow[ld, "\lra{-1}\alpha"] \arrow[d, "\lra{-1}"] \\
H^2_{\P^1 \times y}(\P^1 \times \P^1,M) \arrow[d] & {H^2_{(\infty,y)}(\P^1 \times \P^1,M)} \arrow[l]      & {H^1_{(\infty,y)}(\infty \times \P^1, M_{-1})} \arrow[l] \arrow[d]  & M_{-2}(K(y), \omega_{K(y)/K}) \arrow[l] \arrow[d, "\tr_{K(y)/K}^y"]            \\
H^2(\P^1 \times \P^1,M)                          & H^2_{\infty\times \P^1}(\P^1 \times \P^1,M) \arrow[l]  & H^1(\infty \times \P^1,M_{-1}) \arrow[l]                           & M_{-2}(K) \arrow[l]                     
\end{tikzcd}
\end{equation*}
The diagram arises by looking at various subcomplexes of $C^*((\P^1 \times \P^1)_K, M)$ and applying $H^2$, using that the Gersten complex computes cohomology by Lemma \ref{lemm:gersten-cohomology-dim2}.
All the unlabelled arrows are ultimately induced by various morphisms denoted $\alpha$, in particular via Corollary \ref{cor:purity-codim2}.
We identify the various determinants of normal bundles in evident ways, in particular $\omega_{\infty/\P^1} = \scr O$ (via the coordinate $t^{-1}$) and $\omega_{y/\P^1} \wequi \omega_{K(y)/K}$ (trivializing $\omega_{\A^1/K}$ via $t$).
The cells involving transfers commute by Lemma \ref{lemm:transfer-coh-P1}.
The quadrangle involving $\lra{-1}\alpha$ commutes because the two identifications of $\omega_{(\infty,y)/\P^1 \times \P^1}$ with $\omega_{K(y)/K}$ differ by a factor of $-1$.\todo{???}
All other cells commute for trivial reasons.
The equality of the two extremal paths through the diagram is the desired result.
\end{proof}

\begin{proof}[Proof of Theorem \ref{thm:transfers}]
We must show that, if $L/K$ is a finite extension, then the map $\tr_{L/K}^{x_1, \dots, x_n}: M_{-2}(L, \omega_{L/K}) \to M_{-2}(K, \omega_{L/K})$ is independent of the choices of generators $\{x_i\}_i$.
As a first step, consider an extension $L=K(x)/K$, with $x \in K$, so $L=K$.
Lemma \ref{lemm:transfer-coh-P1}(1) shows that in this case, $\tr_{L/K}^x = \id$.

Now suppose that $L=K(x,y)/K$, i.e. an extension with two chosen generators.
Consider the following diagram
\begin{equation*}
\begin{CD}
M_{-2}(L, \omega_{L/K}) @>>> H^2_{(x,y)}((\P^1 \times \P^1)_K, M) @>>> H^2((\P^1 \times \P^1)_K, M) @>>> M_{-2}(K) \\
@V{\lra{-1}}VV @V{\tau}VV @V{\tau}VV @V{\lra{-1}}VV \\
M_{-2}(L, \omega_{L/K}) @>>> H^2_{(y,x)}((\P^1 \times \P^1)_K, M) @>>> H^2((\P^1 \times \P^1)_K, M) @>>> M_{-2}(K).
\end{CD}
\end{equation*}
Here $\tau$ denotes the swap map on $\P^1 \times \P^1$; the middle cell thus commutes trivially.
We claim that the outer cells commute, too.
To see this, recall from Remark \ref{rmk:alpha-etale-extension} that the map $\alpha_x: M_{-2}(x, \omega_{x/X}) \to H^2_x(X,M)$ is natural in automorphisms of $X$.
Thus, for example, the square
\begin{equation*}
\begin{CD}
M_{-2}(L, \omega_{(x,y)/(\P^1 \times \P^1)_K}) @>>> H^2_{(x,y)}((\P^1 \times \P^1)_K, M) \\
@V{\tau}VV @V{\tau}VV \\
M_{-2}(L, \omega_{(y,x)/(\P^1 \times \P^1)_K}) @>>> H^2_{(y,x)}((\P^1 \times \P^1)_K, M)
\end{CD}
\end{equation*}
commutes.
We also have isomorphisms $\omega_{(x,y)/(\P^1 \times \P^1)_K} \wequi \omega_{L/K} \wequi \omega_{(y,x)/(\P^1 \times \P^1)_K}$, but these are not compatible with $\tau$, since they involve trivializing $\omega_{\A^1\times\A^1/k}$ via $dt_1 \wedge dt_2$.
The switch map thus introduces a factor of $-1$, as claimed.
We have thus proved that the (larger) diagram commutes.
Since the top and bottom maps are determined by Lemma \ref{lemm:double-transfer-trick} to be $\lra{-1}\tr^y_{K(y)/K} \tr^x_{L/K(y)}$ and $\lra{-1}\tr^x_{K(x)/K}\tr^y_{L/K(x)}$ respectively, we deduce that $\tr^y_{K(y)/K} \tr^x_{L/K(y)} = \tr^x_{K(x)/K}\tr^y_{L/K(x)}$.

Now we treat the general case.
Given $L/K$ and sequences of generators $(x_1, \dots, x_n)$, respectively $(y_1, \dots, y_m)$, we find: \[ \tr_{L/K}^{(x_1, \dots, x_n)} = \tr_{L/K}^{(y_1, \dots, y_m, x_1, \dots, x_n)} = \tr_{L/K}^{(x_1, \dots, x_n, y_1, \dots, y_n)} = \tr_{L/K}^{(y_1, \dots, y_m)}. \]
Here the first equality comes from our first step (applied repeatedly), the second from our second step (also applied repeatedly), and the third from the first step again (repeatedly).
This concludes the proof.
\end{proof}

\section{The Rost--Schmid complex} \label{sec:RS}
In this final section we will prove that strongly $\A^1$-invariant sheaves of abelian groups (over a perfect field) are strictly $\A^1$-invariant.
To do this, we will write down a resolution and use it to compute cohomology.
This resolution is called the \emph{Rost--Schmid complex}.
A priori it will not be clear that this is even a complex (let alone a resolution), whence we will call it a \emph{quasi-complex}.

In \S\ref{subsec:def-RS} we define the Rost--Schmid quasi-complex and establish some basic first properties, such as functoriality in smooth maps.
Then in \S\ref{subsec:strict-A1-proof} we show how to prove the main theorem, \emph{assuming} that the Rost--Schmid complex is indeed a complex.
The technical heart of the proof is in \S\ref{subsec:RS-transfers}.
Here we define and study various kinds of transfer maps on Rost--Schmid quasi-complexes, and use them to eventually show that the Rost--Schmid quasi-complex is a complex.

\subsection{Definition and first properties} \label{subsec:def-RS}
\begin{definition} \label{def:rost-schmid-diff}
Let $M$ be a strongly $\A^1$-invariant sheaf of abelian groups, $X$ a smooth scheme, $x \in X$ a point, $y \in X$ an immediate specialization of $x$ and $n \ge 1$.
We define the \emph{Rost--Schmid differential} \[ \cpartial^x_y: M_{-n}(x, \omega_{x/X}) \to M_{-n}(y, \omega_{y/X}) \] as follows.
Let $Z=\overline{\{x\}}$ and $\tilde Z$ the normalization of $Z$.
Then $\tilde Z \to Z$ is finite \cite[Tag 035R]{stacks-project}; write $y_1, \dots, y_n$ for the finitely many points above $y$.
Write also $x \in \tilde Z$ for the canonical lift (recall that $\tilde Z \to Z$ is birational).
Being regular in codimension one \cite[Tag 031S]{stacks-project}, each $\tilde Z_{y_i}$ is essentially smooth and hence induces a canonical differential (see \S\ref{subsec:canonical-diff}) \[ \cpartial^x_{y_i}: M_{-n}(x) \to M_{-n-1}(y_i, \omega_{y_i/\tilde Z}) \] which we can moreover twist (Construction \ref{cons:twisted-canonical-differential}) by $\omega_{\tilde Z_{y_i}/X}$.
We define the Rost--Schmid differential as the composite \begin{gather*} M_{-n}(x, \omega_{x/X}) \wequi M_{-n}(x, \omega_{\tilde Z_{y_i}/X}|_x) \xrightarrow{\oplus \cpartial^x_{y_i}} \\ \bigoplus_i M_{-n-1}(y_i, \omega_{y_i/\tilde Z} \otimes \omega_{\tilde Z_{y_i}/X}) \wequi \bigoplus_i M_{-n-1}(y_i, \omega_{y_i/y} \otimes \omega_{y/X}) \xrightarrow{\Sigma \tr_{y_i/y}} M_{-n-1}(y, \omega_{y/X}). \end{gather*}
\end{definition}

\begin{remark} \label{rmk:RS-ess-smooth}
It is easy to see that the construction of $\cpartial^x_y$ is compatible with étale morphisms (Lemma \ref{lemm:smooth-pullback} below).
We can use this to extend its definition essentially smooth schemes by taking a colimit.

Suppose that $X$ is an \emph{locally noetherian} essentially smooth scheme.
We shall give a direct construction of $\cpartial^x_y$.
To define the differential way may localize $X$ in $y$.
Then $Z$ has dimension one and is noetherian, so the normalization $\tilde Z \to Z$ is defined and has finite fibers \cite[Tag 0C45]{stacks-project}.
This is actually all that we used in the construction.
\end{remark}

\begin{remark}
By construction, the Rost--Schmid differential is a composite of maps which are linear over $\ul{K}_0^{MW}(X_y)$ (see Lemmas \ref{lemm:differential-linear} and \ref{lemm:transfer-linear}).
It follows that we may twist it by a line bundle on $X$.
\end{remark}

\begin{remark}
For $n=0$ and $x \in X$ a generic point we can still make sense of the Rost--Schmid differential: it is just given by the canonical differential of \S\ref{subsec:canonical-diff}.
This can of course not be twisted unless $M$ is itself a contraction.
\end{remark}

We now aim to define the Rost--Schmid complex, which will turn out to be an acylic resolution of $M$.
Using it we shall be able to show that $M$ is strictly $\A^1$-invariant.
It will take us a while to establish these properties.
In fact, for the longest time, we will not even know that the Rost--Schmid complex is a complex, that is, it will not be clear that $\cpartial^2 = 0$.
For this reason we introduce the notion of a \emph{quasi-complex}.
This is a diagram of abelian groups \[ C^0 \xrightarrow{d} C^1 \xrightarrow{d} C^2 \xrightarrow{d} \dots \] where it is not required that $d^2 = 0$.
A morphism of quasi-complexes is still required to commute with the differentials.
In other words, the category of quasi-complexes is $\Fun(\N, \Ab)$.

\begin{definition}
Let $X$ be essentially smooth and $M$ a strongly $\A^1$-invariant sheaf of abelian groups.
We define the \emph{Rost--Schmid quasi-complex} as follows.
We set \[ C^i_{RS}(X,M) = \bigoplus_{x \in X^{(i)}} M_{-i}(x, \omega_{x/X}). \]
For $y \in X^{(i+1)}$ the differential $C^i_{RS}(X,M) \to C^{i+1}_{RS}(X,M)$ has a component from $M_{-i}(x, \omega_{x/X})$ to $M_{-i-1}(y,\omega_{y/X})$ only if $y \in \overline{\{x\}}$ (otherwise it is zero).
If $i \le 1$ then the component is the Gersten differential $\partial^x_y$, employing the identifications of Lemma \ref{lemm:identify-local-cohomology} and Definition \ref{def:identify-M-2}.
If $i > 1$, then the component is the Rost--Schmid differential $\cpartial^x_y$ of Definition \ref{def:rost-schmid-diff} (and Remark \ref{rmk:RS-ess-smooth}).

Now let $\scr L$ be a line bundle on $X$.
We define the \emph{twisted Rost--Schmid quasi-complex} with \[ C^i_{RS}(X,M_{-1}(\scr L)) = \bigoplus_{x \in X^{(i)}} M_{-1-i}(x, \scr L \otimes \omega_{x/X}). \]
Again we can identify the first three terms with the Gersten complex for $M_{-1}(\scr L)$, and define the first two differentials to be the Gersten differentials.
All further differentials are defined to be Rost--Schmid differentials.
\end{definition}

We shall later prove the following two results.
\begin{theorem} \label{thm:2nd-Gersten-diff}
Let $X$ be essentially smooth, $x \in X$ of codimension $1$ (or $0$) and $y \in X$ an immediate specialization.
Let $M = F_{-1}$ be a contraction.
Then the Gersten differential $\partial^x_y$ and the Rost--Schmid differential $\cpartial^x_y$ coincide.
\end{theorem}

\begin{theorem} \label{thm:RS-complex}
Let $X$ be essentially smooth and $M$ a strongly $\A^1$-invariant sheaf of abelian groups.
Then $C^*_{RS}(X, M)$ is a complex.
\end{theorem}

\subsubsection{Rost--Schmid complex with support}
Let $U \subset X$ be an open subscheme and $Z$ its complement.
We then have an evident surjective map $C^*_{RS}(X, M) \to C^*_{RS}(U,M)$.
Its kernel is the (quasi-)complex $C^*_{RS,Z}(X,M)$ with components \[ C^i_{RS,Z}(X,M) = \bigoplus_{x \in X^{(i)} \cap Z} M_{-i}(x, \omega_{x/X}). \]
If $M$ is a contraction, we also have a twisted variant.

As a consequence of Theorem \ref{thm:2nd-Gersten-diff}, the Rost--Schmid complex satisfies an automatic purity property.
\begin{proposition} \label{prop:RS-purity}
Let $X$ be essentially smooth, $Y \subset X$ an essentially smooth closed subscheme of codimension $d$.
The evident map \[ C^{*-d}_{RS}(Y, M_{-d}(\omega_{Y/X})) \to C^*_{RS,Y}(X, M) \] (obtained by using for $y \in Y^{(i)} \subset X^{(i+d)}$ the isomorphism $M_{-d-i}(y, \omega_{y/Y} \otimes \omega_{Y/X}) \wequi M_{-i-d}(y, \omega_{y/X})$) is an isomorphism of (quasi-)complexes.

If $M$ is a contraction, the twisted analog is also true.
\end{proposition}
\begin{proof}
By Theorem \ref{thm:2nd-Gersten-diff}, all differentials in $C^{*-d}_{RS}(Y, M_{-d}(\omega_{Y/X}))$ are Rost--Schmid differentials.
By definition these are the same for both complexes (since the differential from $x$ to $y$ only depends on the scheme $\overline{\{x\}}_y$).
The only problem that could conceivably arise is if the evident map lands in a component of $C^*_{RS,Y}(X,M)$ where the differential is the Gersten differential instead of the Rost--Schmid differential.
Since Gersten differentials are used only in $C^{* \le 1}$ and any Gersten differential starting at a smooth point coincides with the Rost--Schmid differential (Proposition \ref{prop:second-gersten-diff-smooth}), this does not happen.
\end{proof}

\subsubsection{Pullback along smooth morphisms}
Let $f: X \to Y$ be a flat morphism of essentially smooth schemes.
Let $y \in Y^{(i)}$.
Since flat morphisms preserve codimension \cite[Tags 02NM and 02R8]{stacks-project}, the preimage $f^{-1}(y) \subset X$ has generic points of codimension $i$.
Let $x$ be such a generic point.
By compatibility of the cotangent complex with base change, we have $\omega_{y/Y}|_x \wequi \omega_{x/X}$.
We can thus form the morphism $f^*: M_{-i}(y, \omega_{y/Y}) \to M_{-1}(x, \omega_{x/Y})$.
Taking the sum over all $x$, we obtain a morphism of graded abelian groups \[ f^*: C^*_{RS}(Y,M) \to C^*_{RS}(X,M). \]
Similarly for $\scr L$ a line bundle on $Y$ we can build \[ f^*: C^*_{RS}(Y, M_{-1}(\scr L)) \to C^*_{RS}(X, M_{-1}(f^*\scr L)). \]
It is not immediately clear that this is a morphism of quasi-complexes.
We now show that this is true for $f$ smooth; presumably it is true more generally.

\begin{lemma} \label{lemm:smooth-pullback}
Let $f: X \to Y$ be a smooth morphism of essentially smooth schemes.
Then $f^*: C^*_{RS}(Y,M) \to C^*_{RS}(X,M)$ is a morphism of quasi-complexes.

If $M$ is a contraction, the twisted analog is also true.

In fact, all Gersten and Rost--Schmid differentials commute with smooth pullbacks.
\end{lemma}
\begin{proof}
If $Z \subset Y$ has codimension $\ge d$, then so does $f^{-1}(Z) \subset X$.
It follows that the entire coniveau filtration of \S\ref{subsec:coniveau-filtn} is compatible with flat morphisms.
Hence the same is true for the coniveau spectral sequence, on the first page of which we find the Gersten differentials.
This proves that $f^*$ is a morphism of quasi-complexes in degrees $\le 2$, in the untwisted case.

In all other cases, we need to prove that the Rost--Schmid differential is compatible with smooth morphisms.
By construction, the Rost--Schmid differential is built using (0) cofiltered limits, (1) normalization of curves, (2) the Gersten differential and (3) transfers on $M_{-i}$, $i \ge 2$.
Using (0) we reduce to the case $Y$ (and hence $X$) smooth; we must show that (1)--(3) are compatible with smooth morphisms.
For (1), let $x \in X^{(i)}$ lie above $y \in Y^{(i)}$, and let $\tilde C$ be the normalization of $\overline{\{y\}}$.
We must show that $\tilde C \times_Y X \to \overline{\{x\}}$ is the normalization.
This holds by \cite[Tag 07TD]{stacks-project}.
For (2) we already argued.
Finally (3) is Lemma \ref{lemm:transfer-smooth-base-change}.
\end{proof}

\subsection{Proof of strict $\A^1$-invariance} \label{subsec:strict-A1-proof}
\begin{theorem} \label{thm:strict-A1-invariance}
Let $k$ be a perfect field.

For any $X \in \Sm_k$ and strongly $\A^1$-invariant sheaf of abelian groups $M$ on $\Sm_k$, $C^*_{RS}(X,M)$ is an acyclic resolution of $M|_{X_\Nis}$ (so that $H^* C^*_{RS}(X,M) \wequi H^*(X,M)$).
Moreover $M$ is strictly $\A^1$-invariant.
\end{theorem}
\begin{proof}
We know that $C^*_{RS}(X,M)$ consist of acyclic sheaves (see Lemma \ref{lemm:acyclic-sheaves} below).
We will now prove by induction on $d$ the following statements:
\begin{enumerate}
\item For an essentially smooth, local scheme $X$ of dimension $d$, the map $C^*_{RS}(X,M) \to C^*_{RS}(\A^1_X,M)$ is a quasi-isomorphism.
\item For an essentially smooth, local scheme $Y$ of dimension $d+1$, the complex $C^*_{RS}(Y,M)$ is a resolution of $M(Y)$.
\item For a smooth scheme $X$ of dimension $d$, we have $H^*(X, M) \wequi H^*(\A^1_X, M)$ and $C^*_{RS}(X,M) \to C^*_{RS}(\A^1_X,M)$ is a quasi-isomorphism.
\end{enumerate}

The base case is $d=0$.
In this case, by definition, we have $C^*_{RS}(\A^1_X, M) = C^*(\A^1_X, M)$, which computes $H^*(\A^1_X, M)$ (Lemma \ref{lemm:gersten-cohomology-dim2}).
This vanishes for $*>0$ since $M$ is strongly $\A^1$-invariant and $X$ has dimension $0$; hence (1) and (3) hold.
For (2) we similarly have $C^*_{RS}(Y,M) = C^*(Y,M)$ which computes $H^*(Y,M)$.
This coincides with $H^*_\Zar(Y,M)$ (Corollary \ref{cor:Zariski-vanishing}) and so has homology concentrated in degree $0$ ($Y$ being Zariski local), given there by $M(Y)$.
In other words $C^*_{RS}(Y,M)$ is a resolution of $M(Y)$.

Now we assume the claims proved for $d-1$ and establish (1) to (3).

(1) Let $x \in X$ be the closed point and $U$ the open complement, which has dimension $d-1$.
Let us write $E(S, M)$ for the fiber of $C^*_{RS}(S, M) \to C^*_{RS}(\A^1_S,M)$.
Via purity for the Rost--Schmid complex (Proposition \ref{prop:RS-purity}) we obtain an exact sequence \[ 0 \to C^{*-d}_{RS}(x, M_{-d}) \to C^*_{RS}(X,M) \to C^*_{RS}(U,M) \to 0. \]
We have a similar exact sequence for $\A^1_X$, and hence obtain a fiber sequence \[ \Sigma^d E(x,M_{-d}) \to E(X,M) \to E(U,M). \]
The left hand term vanishes by the case $d=0$, and the right hand term vanishes by (3) for $d-1$
Thus the middle term vanishes, as needed.

(2) Let $Y'$ be obtained from $Y$ by removing all closed subschemes of codimension $\ge 2$.
Then $M(Y') \wequi M(Y)$ by Lemma \ref{lem:remove-codim-2}, and also $C^{* \le 1}_{RS}(Y,M) \wequi C_{RS}^{* \le 1}(Y',M)$.
We have $C^*_{RS}(Y',M) = C^*(Y',M)$ and the latter computes $H^*(Y',M)$ (Lemma \ref{lemm:gersten-cohomology-dim2}).
Consequently \[ 0 \to M(Y) \to C^0(Y', M) \wequi C^0_{RS}(Y,M) \to C^1(Y,M) \wequi C^1_{RS}(Y',M) \] is exact.
It remains to prove exactness of $M(Y) \to C^*_{RS}(Y,M)$ at the other spots.
Thus let $a \in C^i_{RS}(Y,M)$ with $i \ge 1$ and $\cpartial(a) = 0$.
We must prove that $a$ is a boundary.
Write $Y = \lim_i Y_i$ as a cofiltered limit of smooth schemes along étale maps.
Then $a$ arises from a cycle $\tilde a \in C^*_{RS}(Y_0, M)$ for some $Y_0$.
There exist finitely many points $y_1, \dots, y_n \in Y_0^{(i)}$ such that $a$ is supported on the $y_i$.
Gabber's presentation lemma (Theorem \ref{thm:gabber-lemma}) applied to $Y_0$, the closed subscheme $\cup_i \overline{\{y_i\}}$ and the image of the closed point of $Y$, yields an étale morphism $Y_0' \to \A^1_S$ which is an étale neighborhood of each of the $y_i$.
Here $Y_0' \subset Y_0$ is an open neighborhood of the image of the closed point of $Y$, and $S$ is smooth of dimension $d$.
Since $Y$ is local, the map $Y \to Y_0$ factors through $W := Y_0' \times_{S} S_s$, where $S_s$ denotes the localization of $S$ in the image of the closed point of $Y$.
We shall show that the restriction of $\tilde a$ to $W$ is a boundary.
Let $\tilde y_i$ be the image of $y_i$ in $\A^1_{S_s}$.
The map $C^i_{RS}(\A^1_{S_s},M) \to C^i_{RS}(W,M)$ induces an isomorphism on the components corresponding to the $\tilde y_i$, and similarly for points in their closures.
It follows that $\tilde a$ lifts to a cycle $b \in C^i_{RS}(\A^1_{S_s},M)$.
By (1), $b$ is a boundary.
Hence $\tilde a|_W$ is a boundary, as needed.

(3) Having proved (2), we know that for any essentially smooth scheme $Y$ of dimension $\le d+1$ we have $H^*(Y,M) \wequi H^* C^*_{RS}(Y,M)$.
The two statements are thus equivalent.
To prove that $H^*(X,M) \wequi H^*(\A^1_X, M)$ we may reduce to the case where $X$ is local.\footnote{Consider the Zariski sheaf of spectra $E(\ph) = \fib(\Gamma_\Nis(\ph, M) \to \Gamma_\Nis(\ph \times \A^1, M))$. Then $E(U) = 0$ if and only if $H^*(U,M) \wequi H^*(\A^1_U,M)$. Since the former condition can be checked on stalks, so can the latter.}
By what we just said this is equivalent to proving that $C^*_{RS}(X,M) \to C^*_{RS}(\A^1_X, M)$ is a quasi-isomorphism, which is (1).
\end{proof}

\begin{lemma} \label{lemm:acyclic-sheaves}
Let $X$ be a scheme and $x \in X$ a point.
Write $i: x \to X$ for the inclusion.
Let $G$ be a sheaf of groups on $x_\Nis$.
We have $H^i(X, i_* G) = 0$ for $i=1$ (respectively for $i \ge 1$ if $G$ is abelian).

In particular, skyscraper sheaves (e.g., constant sheaves) on $X_\Nis$ are acyclic.
\end{lemma}
\begin{proof}
Set $E=K(G,1) \in \Shv_\Nis(x)$.
Then $i_* E$ is a pointed sheaf on $X_\Nis$ which is connected ($\pi_0 (i_* E)(U) = \pi_0 E(i^* U) \wequi H^1(i^* U, G) = *$ since $i^* U$ has dimension $0$) and $\Omega i_* E \wequi i_* \Omega E \wequi i_* G$.
In other words $i_* E \wequi K(i_* G, 1)$ and so $H^1(X, E) \wequi \pi_0 (i_* E)(X) = *$ as we have seen.
If $G$ is abelian, the same argument works with $K(G,n)$ in place of $K(G,1)$ for any $n \ge 1$.
\end{proof}

\subsection{Transfer along proper morphisms} \label{subsec:RS-transfers}
Let $p: X \to Y$ be a proper morphism of essentially smooth schemes of relative dimension $d$.
We attempt to define a morphism of graded abelian groups \[ p_*: C^*_{RS}(X, M(\omega_{X/Y})) \to C^{*-d}_{RS}(Y, M_{-d}) \] as follows.
Let $x \in X$ have image $y \in Y$.
If $x \to y$ is not finite, we define the corresponding component of $p_*$ to be zero.
If $x \to y$ is finite and $x$ has codimension $c$, then $y$ has codimension $c-d$\NB{$\dim \bar x = \dim \bar y$} and we define $p_*$ on the component corresponding to $x$ as \[ M_{-c}(x, \omega_{x/X} \otimes \omega_{X/Y}) \wequi M_{-c}(x, \omega_{x/y} \otimes \omega_{y/Y}) \xrightarrow{\tr_{x/y}} M_{-c}(y, \omega_{y/Y}) = (M_{-d})_{-(c-d)}(y, \omega_{y/Y}). \]
This may not quite make sense, because $M$ cannot be twisted (unless it is a contraction) and the transfers on $M_{-1}$ may not be well-defined.
We do obtain a well-defined map in the following cases:
\begin{itemize}
\item $M=F_{-1}$ and $p$ birational (no transfer is needed on $C^0_{RS}$, and starting from $C^1_{RS}$ we only have to deal with $F_{-i}$ for $i \ge 2$)
\item $M=F_{-2}$ and $p$ arbitrary
\item $X = \P^1_Y$ (see \S\ref{subsub:transfer-P1})
\end{itemize}

\subsubsection{Transfer from $\P^1$ in degrees $\le 2$} \label{subsub:transfer-P1}
Let $X$ be an essentially smooth scheme and write $p: \P^1_X \to X$ for the projection.
We define $p_*: C^*_{RS}(\P^1_X, M) \to C^{*-1}_{RS}(X,M_{-1})$ as follows:
\begin{itemize}
\item In degree $*=0$, the map is $0$.
\item In degree $*=1$, we need to construct for any point $x \in \P^1_X$ of codimension $1$, finite over a point $\eta \in X$ (automatically of codimension $0$) a map $p_*: M_{-1}(x, \omega_{x/\P^1_X}) \to M_{-1}(\eta)$.
  (If $x$ is not finite over its image, the transfer is defined to be zero.)
  We have $x \in \P^1_\eta$, and so if $x \ne \infty$ we can use the transfer morphism $\tr_{L/K}^x$ of \S\ref{subsec:contractions-transfers}, where $L$ is the residue field of $x$ and $K$ is the residue field of $\eta$.
  If $x=\infty$ then we need not perform any transfer at all.
\item In degree $*>1$ we can use the general construction.
\end{itemize}

\begin{lemma} \label{lemm:transfer-P1-gersten}
The morphism \[ p_*: C^{* \le 2}_{RS}(\P^1_X, M) \to C^{(*-1) \le 1}_{RS}(X,M_{-1}) \] is a morphism of (quasi-)complexes.
\end{lemma}
\begin{proof}
In other words we must show that the transfer map commutes with the first two Gersten differentials.
We may assume $X$ local of dimension $1$, with generic point $\eta$ and closed point $z$.
Write $\mu \in \P^1_X$ for the generic point.

In degree zero, we must show that the following composite vanishes: \[ M(\mu) \xrightarrow{\partial} \bigoplus_{x \in (\P^1_\eta)^{(1)}} M_{-1}(x, \omega_{x/\eta}) \xrightarrow{\Sigma_x \tr_{x/\eta}^x} M_{-1}(\eta). \]
Using Lemma \ref{lemm:transfer-coh-P1}(2), we see that the second map can be written as the composite \[ \bigoplus_{x \in (\P^1_\eta)^{(1)}} M_{-1}(x, \omega_{x/\eta}) \wequi C^1(\P^1_\eta, M) \to H^1(\P^1_\eta, M) \wequi M_{-1}(\eta). \]
It is now clear that this map vanishes on the image of $\partial$.

Now let $y \in  \P^1_X$ be of codimension $1$, with closure $Y \subset \P^1_X$.
The only way in which $y$ can not be finite over its image is if $Y = \P^1_z$.
In this case, the desired commutativity follows as in the degree zero case.
We may thus assume that $y$ is finite over its image $\eta \in X$.
Consider the following commutative diagram of cofiber sequences
\begin{equation*}
\begin{CD}
\P^1_\eta/\P^1_\eta \setminus  Y_\eta @>>> \P^1_X/\P^1_X \setminus Y @>>> \P^1_X/\P^1_\eta \cup \P^1_X \setminus Y @>>> \Sigma \P^1_\eta/\P^1_\eta \setminus Y_\eta \\
@AAA @AAA @AAA @AAA \\
\P^1_\eta  @>>> \P^1_X @>>> \P^1_X/\P^1_\eta @>>> \Sigma \P^1_\eta \\
\end{CD}
\end{equation*}
Mapping into $K(M,2)$, the right hand square yields the following commutative diagram
\begin{equation*}
\begin{CD}
C^1_{RS}(\P^1_X, M) \supset H^1_{y}(\P^1_X, M) @>{\partial}>> H^2_{Y^{(1)}}(\P^1_X, M) \subset C^2_{RS}(\P^1_X, M) \\
@V{p_*^G}VV                                 @V{p_*^G}VV \\
C^0_{RS}(X,M_{-1}) \wequi H^1(\P^1_\eta,M) @>{\partial}>> H^2_{\P^1_z}(\P^1_X, M) \subset C^1_{RS}(X, M_{-1}).
\end{CD}
\end{equation*}
The top and bottom rows identify with parts of the relevant Rost--Schmid complexes.
The result will be proven if we can show that $p_*^G = p_*$.
On the left hand side this follows easily from Lemma \ref{lemm:transfer-coh-P1}(2).
For the right hand side this is Lemma \ref{lemm:RS-transfer-geom} below.
\end{proof}

\begin{lemma} \label{lemm:RS-transfer-geom}
Let $X$ be essentially smooth, local of dimension $1$ with closed point $z$ and $x \in \P^1_X$ a closed point.
The composite \[ M_{-2}(x, \omega_{x/X}) \wequi H^2_x(\P^1_X, M) \to H^2_{\P^1_z}(\P^1_X, M) \wequi M_{-2}(z, \omega_{z/X}) \] is the canonical transfer.
\end{lemma}
\begin{proof}
We can use $C^*_{RS}(\P^1_X, M) = C^*(\P^1_X,M)$ to compute the relevant cohomology groups (Lemma \ref{lemm:gersten-cohomology-dim2}).
Everything is happening in the subcomplex \[ C^*_{\P^1_z}(\P^1_X, M) \wequi C^{*-1}(\P^1_z, M_{-1}(\omega_{z/X})). \]
This shows that the composite is the same as the following: \[ M_{-2}(x, \omega_{x/X}) \wequi H^1_x(\P^1_z, M_{-1}(\omega_{z/X})) \to H^1(\P^1_z, M_{-1}(\omega_{z/X})) \wequi M_{-2}(z, \omega_{z/X}) \] which is the canonical transfer by Lemma \ref{lemm:transfer-coh-P1}(2).\todo{$x=\infty$...}
\end{proof}

\subsubsection{Transfers via $M_{-1}$ in degrees $\le 2$}
\begin{lemma} \label{lemm:blowup}
Let $X$ be obtained from $Y$ by blowing up a point of codimension $2$.
Then \[ p_*: C^{* \le 2}_{RS}(X, M_{-1}(\omega_{X/Y})) \to C^{* \le 2}_{RS}(Y, M_{-1}) \] is a morphism of quasi-complexes.
\end{lemma}
\begin{proof}

Let $z \in Y$ be the point of codimension $2$ being blown up.
Write $\P^1_z \subset X$ for its preimage, and $\eta \in \P^1_z$ for the generic point.
The map $X \setminus \P^1_z \to Y \setminus z$ being an isomorphism, we see that the only differentials which do possibly not commute with $p_*$ are those with targets corresponding to $z$, i.e. emanating from points in $\tilde y \in X^{(1)}$ whose image specializes to $z$.
There are two distinct cases of such $\tilde y$: either $\tilde y = \eta$, or not.
The case $\tilde y = \eta$ is, up to twisting, contained in Lemma \ref{lemm:transfer-P1-gersten}.

Now we treat the other case.
Using that the Rost--Schmid differential is compatible with cofiltered limits, and defined ``pointwise'', we may as well assume that $Y=Y_z$ is the localization of a smooth\NB{needed?} scheme $Y$ in a point $z$.
Let $y \in Y$ be the image of $\tilde y$.
Write $Z = \overline{\{y\}}$ and $\tilde Z = \overline{\{\tilde y\}}$, so that $\tilde Z \to Z$ is finite (being proper and quasi-finite \cite[Tag 02LS]{stacks-project}) and birational.
We can find $y_0 \in Y^{(1)}$ such that $Z_0 := \overline{\{y_0\}}$ is smooth at $z$ (take the zero locus of function in the maximal ideal of $Y$ with non-zero differential\NB{more details?}).
Put $\tilde Z_0 := \overline{\{\tilde y_0\}}$, where $\tilde y_0$ denotes the unique lift of $y_0$. 
Write $\tilde z$ for the unique lift of $z$ to $\tilde Z_0$ (note that $\tilde Z_0 \to Z_0$ is an isomorphism, $Z_0$ being normal and $\tilde Z_0 \to Z_0$ being finite and birational).
We have the short exact sequence \[ H^1(Y, M_{-1}) \to H^1(Y \setminus Z_0, M_{-1}) \to H^2_{Z_0}(Y, M_{-1}) \] in which the first term vanishes since $Y$ is local (Corollary \ref{cor:Zariski-vanishing}) and the third term vanishes by purity: \[ H^2_{Z_0}(Y, M_{-1}) \stackrel{L.\ref{lemm:gersten-cohomology-dim2}}{\wequi} H^2 C^*_{Z_0}(Y, M_{-1}) \stackrel{C.\ref{cor:purity-codim2}}{\wequi} H^1 C^*(Z_0, M_{-2}(\omega_{Z_0/Y})) \stackrel{L.\ref{lemm:gersten-cohomology-dim2}}{\wequi} H^1(Z_0, M_{-2}(\omega_{Z_0/Y}) \stackrel{C.\ref{cor:Zariski-vanishing}}{=} 0. \]
It follows that $H^1(Y \setminus Z_0, M_{-1})$ vanishes and thus \begin{equation}\label{eq:partial-surj} \partial: C^0(Y) \to C^1(Y \setminus Z_0) = \bigoplus_{x \in Y^{(1)} \setminus y_0}M_{-2}(x, \omega_{x/Y}) \end{equation} is surjective.
Write $\mu \in Y^{(0)}$ for the (unique) generic point and $\tilde \mu \in X^{(0)}$ for its unique lift.

Pick $\tilde a \in C^1(X, M_{-1}(\omega_{X/Y}))$ in the component corresponding to $\tilde y$, that is, $\tilde a \in M_{-2}(\tilde y, \omega_{\tilde y/X} \otimes \omega_{X/Y})$.
We shall prove that $p_*$ commutes with the differential on $\tilde a$.
Under the isomorphisms $\omega_{\tilde y/X} \otimes \omega_{X/Y} \wequi \omega_{\tilde y/Y}$ and $\tilde y \wequi y$, $\tilde a$ corresponds to a class $a \in M_{-2}(y, \omega_{y/Y})$; in fact $a=p_*(\tilde a)$.
Surjectivity of the map in \eqref{eq:partial-surj} implies that we can find $b \in M_{-1}(\mu)$ with $\partial^\mu_y(b) = a$ and $\partial^\mu_{y'}(b) = 0$ for $y' \not\in\{y,y_0\}$.
As above, the transfer $C^0(X, M_{-1}(\omega_{X/Y})) \to C^0(Y, M_{-1})$ is bijective for trivial reasons, so we may pick $\tilde b \in M_{-1}(\tilde \mu, \omega_{X/Y})$ above $b$.
Using that transfer and differential commute on $X \setminus \P^1_z$, we find that $\partial \tilde b$ is supported on $\eta$, $\tilde y_0$ and $\tilde y$, with $\partial^{\tilde \mu}_{\tilde y} \tilde b = \tilde a$.
Since $\partial^2 = 0$ (recall that the first two differentials in the Rost--Schmid quasi-complex are given by the differentials from the Gersten complex) we deduce that \[ \partial \tilde a = -(\partial^{\tilde y_0}_{\tilde z} \partial^{\tilde \mu}_{\tilde y_0} \tilde b + \sum_{x \in (\P^1_z)^{(1)}} \partial^\eta_x \partial^{\tilde \mu}_{\eta} \tilde b). \]
Apply now $p_* = \tr_{?/z}$.
The second term vanishes, by the case of commutativity we have already dealt with.
The component of $\partial$ emanating at $\tilde y_0$ commutes with $p_*$ since $\tilde Z_0 \to Z_0$ is an isomorphism, and the differential emanating from $\tilde \mu$ was already treated.
Hence we find \[ p_* \partial \tilde a = - p_* \partial^{\tilde y_0}_{\tilde z} \partial^{\tilde \mu}_{\tilde y_0} \tilde b = -\partial^{y_0}_z p_* \partial^{\tilde \mu}_{\tilde y_0} \tilde b = -\partial^{y_0}_z \partial^{\mu}_{y_0} p_* \tilde b = -\partial^{y_0}_z \partial^{\mu}_{y_0} b. \]
Using again that $\partial^2 = 0$, we see that this is the same thing as $\partial^y_z \partial^\mu_y b = \partial^y_z a$.
This is what we set out to prove.
\end{proof}
\begin{remark}
In the above proof we crucially used that $C^*_{RS}(X,M)$ is a complex in degrees $\le 2$.
This is why we defined the first two differentials to be the Gersten differentials.
\end{remark}
\begin{discussion}
We can define $p_*: C^{1 \le * \le 2}(X, M(\omega_{X/Y})) \to C^{1 \le * \le 2}_{RS}(Y, M)$ (with no contraction on $M$!).
However the above proof does not show that this map is a morphism of quasi-complexes.
(There is no complex $C^{* \le 2}(X, M({\omega_{X/Y}}))$ to which we could apply $\partial^2(\tilde b) = 0$.)
\end{discussion}

\begin{corollary} \label{cor:finite-transfer}
Let $p: X \to Y$ be a finite morphism of finite presentation between essentially smooth schemes.
Then \[ p_*: C^{* \le 1}_{RS}(X, M_{-2}(\omega_{X/Y})) \to C^{* \le 1}_{RS}(Y, M_{-2}) \] is a morphism of quasi-complexes.
\end{corollary}
\begin{proof}
We may assume $X, Y$ smooth.\footnote{Write $Y = \lim_i Y_i$.
First, since $p$ is finite of finite presentation, there exists $p_0: X_0 \to Y_0$ inducing $p$ with $p_0$ finite \cite[Tags 01ZM and 01ZO]{stacks-project}.
Now $\emptyset = X_{sing} = \lim_i (X_i)_{sing}$ and so $(X_i)_{sing} = \emptyset$ for $i$ sufficiently large \cite[Tag 081A]{stacks-project}.}
By considering the localizations of $Y$ in points of codimension $1$, we may assume $Y$ local; in particular $X,Y$ are then affine of dimension $\le 1$.
Being essentially smooth, we may also assume $X,Y$ integral.
If we can prove the theorem for morphisms $p_1: X \to Y_1$ and $p_2: Y_1 \to Y$, then we have also proved it for $p_2 \circ p_1: X \to Y$.
Let $L/K$ be the function field extension corresponding to $p$, and let $L/E/K$ be an intermediate extension.
Let $Y_1$ be the normalization of $Y$ in $E$; then $X \to Y_1$ and $Y_1 \to Y$ are finite \cite[Tag 035R]{stacks-project} and $Y_1$ is normal of dimension $1$, so essentially smooth \cite[Tag 031S]{stacks-project}.
Applying this construction several times if necessary, we reduce to the case where $L/K$ is monogeneic.
Set $X = \Spec(B)$, $Y=\Spec(A)$.
We may assume that $L=K(t)$ for some $t \in B$ (note that $K \otimes_A B = L$, $B/A$ being finite).
Let $B_0 = A[t] \subset B$.
Note that $B_0$ is finite over $A$ ($t \in B$ satisfying an integral equation over $A$).
We can thus view $Z = \Spec(B_0)$ as a closed subscheme of $\P^1_A$; by construction its normalization is $X$.
Lemma \ref{lemm:blowup-res} below shows that by blowing up finitely many points of codimension $2$, we obtain a sequence of morphisms \[ W_n \to \dots \to W_1 = \P^1_A \] such that if $Z_i$ is the strict transform of $Z$ in $W_i$ then $Z_n \wequi X$.
Write $q: W_n \to \P^1_A \to Y$ for the composite morphism.
Applying Lemma \ref{lemm:blowup} finitely many times, as well as Lemma \ref{lemm:transfer-P1-gersten}, we see that $q_*$ is a morphism of quasi-complexes.
The restriction of this morphism of quasi-complexes to the part of the source corresponding to $X \wequi Z_n \subset W_n$ is the desired result.
\end{proof}

\begin{lemma}\label{lemm:blowup-res}
Let $A$ be a discrete valuation ring, $A \to B$ a finite morphism with $B$ a Dedekind domain (i.e. regular).
Let $t \in B$ generate the fraction field of $B$ and set $B_0 = A[t] \subset B$.
Set $X_0 = \Spec(B_0) \subset \A^1_A$.
There exists a sequence of blowups in closed points \[ \A^1_A \leftarrow W_1 \leftarrow W_2 \leftarrow \dots \leftarrow W_n \] such that if $X_i$ is the strict transform of $X_0$ in $W_i$, then $X_n = \Spec(B)$.
\end{lemma}
\begin{proof}
We obtain $W_{i+1}$ by blowing up in $W_i$ the closed points of $X_i$.
Since $B$ is Dedekind, every divisor is Cartier, and so by the universal property of blowing up \cite[Tag 0806]{stacks-project} the map $X = \Spec(B) \to W_i$ factors uniquely through $W_{i+1}$.
Let $X_i = \Spec(B_i)$, so that we have inclusions $A \subset B_0 \subset B_1 \subset \dots \subset B$.
Since $A$ is noetherian and $B$ is finite over $A$, this chain of submodules must be stationary, that is, $B_n = B_{n+1}$ for some $n$.
Let $I \subset B_n$ be the ideal defining the closed points.
Then $I B_{n+1}$ is locally principal, by the universal property of blowing up.
But $B_{n+1} = B_n$, so $I$ is locally principal.
Since $B_n$ is a noetherian domain, it follows that $B_n$ is Dedekind, i.e., integrally closed (e.g., use \cite[Tags 034X and 05KH]{stacks-project}).
Thus whenever $b \in B$ (so $b$ is integral over $A$) also $b \in B_n$, i.e., $B_n = B$.
\end{proof}

\subsubsection{Proof of Theorems \ref{thm:2nd-Gersten-diff} and \ref{thm:RS-complex}}
We use Lemma \ref{lemm:blowup-res} to prove that $p_*$ is a morphism of quasi-complexes in more degrees.
\begin{proposition} \label{prop:transfer-P1-gersten}
Let $X$ be an essentially smooth scheme and write $p: \P^1_X \to X$ for the projection.
The morphism \[ p_*: C^{*}_{RS}(\P^1_X, M) \to C^{*-1}_{RS}(X,M_{-1}) \] is a morphism of quasi-complexes.
\end{proposition}
\begin{proof}
We may assume $X$ smooth.
The claim was proved in degrees $\le 2$ in Lemma \ref{lemm:transfer-P1-gersten}.
Let $\tilde y \in \P^1_X$ be a point of codimension $c \ge 2$ and $y \in X$ its image.
If $\tilde y \to y$ is not finite, then $\tilde y$ must be the generic point of $\P^1_y$.
The required commutativity is thus obtained from the case $c=0$, suitably twisted.
We may hence assume that $\tilde y \to y$ is finite.
Write $\tilde Z, Z$ for the closures of $\tilde y, y$ and $\tilde Y, Y$ for their normalizations.
The required commutativity is obtained by applying Corollary \ref{cor:finite-transfer} to the morphism $\tilde Y \to Y$ and using the definition of the Rost--Schmid differential.\NB{details?}
\end{proof}

The remaining two results will be proven using the same strategy, which we outline now.
Let $X$ be an essentially smooth scheme and $z \in X$ a point.
We shall seek to prove something about (composites of) Gersten or Rost--Schmid differentials terminating at $z$.
Firstly, since the differentials are compatible with cofiltered limits (Remark \ref{rmk:RS-ess-smooth}), we may assume $X$ smooth, say of some dimension $d$; suppose $z$ has codimension $c$.
Secondly, since we only care about what happens in the end at $z$, and since the differentials compatible with étale extensions (Lemma \ref{lemm:smooth-pullback}), we may work instead with $X_z^h$.
Recall that there is an essentially smooth morphism $X_z^h \to z$, making $z$ into a $z$-rational point (see Corollary \ref{cor:ess-smooth-retraction}).
Writing $X_z^h \to z$ as a cofiltered limit of smooth $z$-schemes, we reduce to the case where $X$ is a smooth $z$-scheme.
Remark \ref{rmk:smooth-retraction-dimensionality} shows that in fact $X$ now has dimension $c$.
Let $Z \subset X$ be closed of positive codimension, containing all the starting points of differentials we are interested in.
Applying Gabber's lemma to $(X,Z,z)$ we find a smooth $z$-scheme $S_0$ and an étale map $X_0 \to \A^1_{S_0}$ which is an étale neighborhood of $Z$ and such that $Z \to S_0$ is finite.
It will follow that we may replace $X$ by $\P^1_{S_0}$.
Write $s \in S_0$ for the image of $z$; then of course $z \to s$ is an isomorphism ($z$ being a $z$-rational point).
Since our problem is étale local around $z$, it will follow further that we may replace $X$ by $\P^1_S$, where $S= (S_0)_s^h$.
Now $Z$ splits into a finite disjoint union of local schemes and one scheme having no points over $s$ \cite[Tag 04GJ]{stacks-project}.
Precisely one of the local schemes contains $z$; we may replace $Z$ by this component.
Since $Z \to S$ is a finite map of local schemes, $z$ is in fact the only point in the fiber above $S$.

In summary, we have arranged that:
\begin{itemize}
\item $X = \P^1_S$
\item $S \to z$ is a cofiltered limit of smooth $z$-schemes of dimension $c-1$
\item we need only treat differentials living on $Z \subset X$, where $Z \to S$ is finite and $z \in Z$ is the only point of $Z$ above $s \in S$.
\end{itemize}

\begin{proof}[Proof of Theorem \ref{thm:2nd-Gersten-diff}]
We must show that the second Gersten differential in $C^*_{RS}(X, M_{-1})$ coincides with the Rost--Schmid differential.
Let $z \in X$ a point of codimension $2$, $y \in X$ a point of codimension $1$ specializing to $z$, and $Z = \overline{\{y\}}$.
We shall treat the two differentials from $y$ to $z$.
Perform the above constructions and reductions; thus assume that $X$ is of the form $\P^1_S$, etc.
Note that through all this, $y$ remains a generic point of $Z$.
Let $\tilde Z$ denote the normalization of $Z$, write $\tilde y$ for the unique lift of $y$ to $\tilde Z$.
Consider the diagram
\begin{equation*}
\begin{CD}
M_{-2}(\tilde y, \omega_{\tilde y/S}) @>{\partial}>> \bigoplus_{\tilde z \in \tilde Z^{(1)}} M_{-3}(\tilde z, \omega_{\tilde z/S}) \\
@V{\wequi}VV                                             @V{\tr}VV \\
M_{-2}(y, \omega_{y/S}) @>{\partial}>> M_{-3}(z, \omega_{z/S}) \\
@V{\tr}VV                                @V{\wequi}VV \\
M_{-2}(\eta) @>{\partial}>> M_{-3}(s, \omega_{s/S}).
\end{CD}
\end{equation*}
The outer rectangle commutes by Corollary \ref{cor:finite-transfer} (applied to $\tilde Z \to S$), and the bottom rectangle commutes by Proposition \ref{prop:transfer-P1-gersten} (applied to $\P^1_S \to S$).
It follows that the top rectangle commutes, which is what we needed to prove.
\end{proof}

\begin{proof}[Proof of Theorem \ref{thm:RS-complex}]
We must prove that $C^*_{RS}(X,M)$ is complex.
This is true by definition in degrees $\le 2$.
By writing $X$ as a filtered colimit, we may assume $X$ smooth over $k$, say of dimension $d$.
We shall prove the result by induction on $d$; the case $d \le 2$ is clear as noted above.
Thus let $x \in X$ be a point of codimension $c-2 \ge 1$ and let $z \in X$ be a specialization of $x$ of codimension $c$.
We must prove that the component of $\cpartial^2$ from $x$ to $z$ vanishes.
Let $Z$ be the closure of $x$ and run the above constructions and reductions.
Write $w \in S$ for the image of $x$.
Proposition \ref{prop:transfer-P1-gersten} supplies us with the following commutative square
\begin{equation*}
\begin{CD}
M_{-c+2}(x, \omega_{x/S}) @>{\cpartial^2}>> M_{-c}(z, \omega_{z/S}) \\
@V{p_*}VV                                     @V{p_*}V{\wequi}V  \\
M_{-c+2}(w, \omega_{w/S}) @>{\cpartial^2}>> M_{-c}(s, \omega_{s/S}).
\end{CD}
\end{equation*}
To prove that the top map vanishes, it thus suffices to prove that the bottom map vanishes.
In other words it will be enough to prove the result for $S$.
Note that $S$ is a cofiltered limit of smooth $z$-schemes of dimension $c-1$, and $z$ is a cofiltered limit of smooth $k$-schemes of dimension $d-c$ (recall that originally, $z$ was a point of codimension $c$ on a smooth $k$-scheme of dimension $d$).
Thus $S$ is a cofiltered limit of smooth $k$-schemes of dimension $d-1$, and so the required result holds by induction.
\end{proof}

\bibliographystyle{alpha}
\bibliography{bibliography}

\end{document}